\documentclass[12pt,oneside, reqno]{amsart}
\usepackage{txfonts}
\usepackage{amsfonts}
\usepackage{mathrsfs}
\usepackage{amsmath}
\usepackage{amssymb, amsmath}
\usepackage{bbm}
\usepackage{stmaryrd}
\usepackage{esint}
\usepackage{xcolor}
\usepackage{color}

\usepackage{enumerate}

\newcommand{\be}{\begin{eqnarray}}
	\newcommand{\ee}{\end{eqnarray}}
\newcommand{\ce}{\begin{eqnarray*}}
	\newcommand{\de}{\end{eqnarray*}}
\newtheorem{theorem}{Theorem}[section]
\newtheorem{lemma}[theorem]{Lemma}
\newtheorem{remark}[theorem]{Remark}
\newtheorem{definition}[theorem]{Definition}
\newtheorem{proposition}[theorem]{Proposition}
\newtheorem{Examples}[theorem]{Examples}
\newtheorem{corollary}[theorem]{Corollary}

\def\[{{\Big[}}
\def\]{{\Big]}}
\def\<{{\langle}}
\def\>{{\rangle}}
\def\({{\Big(}}
\def\){{\Big)}}

\def\bx{{\mathbf{x}}}

\def\sgn{\mbox{\rm sgn}}

\def\min{{\mathord{{\rm min}}}}

\def\no{\nonumber}
\def\bt{\begin{theorem}}
	\def\et{\end{theorem}}
\def\bl{\begin{lemma}}
	\def\el{\end{lemma}}
\def\br{\begin{remark}}
	\def\er{\end{remark}}
\def\bx{\begin{Example}}
	\def\ex{\end{Example}}
\def\bd{\begin{definition}}
	\def\ed{\end{definition}}
\def\bp{\begin{proposition}}
	\def\ep{\end{proposition}}
\def\bc{\begin{corollary}}
	\def\ec{\end{corollary}}

\def\cF{{\mathcal F}}

\def\cJ{{\mathcal J}}

\def\cR{{\mathcal R}}

\def\cW{{\mathcal W}}

\def\mB{{\mathbb B}}

\def\mE{{\mathbb E}}

\def\mH{{\mathbb H}}
\def\mI{{\mathbb I}}

\def\mL{{\mathbb L}}

\def\mN{{\mathbb N}}

\def\mP{{\mathbb P}}

\def\mR{{\mathbb R}}

\def\sA{{\mathscr A}}
\def\sB{{\mathscr B}}

\def\sD{{\mathscr D}}
\def\sE{{\mathscr E}}
\def\sF{{\mathscr F}}
\def\sG{{\mathscr G}}
\def\sH{{\mathscr H}}

\def\sL{{\mathscr L}}
\def\sM{{\mathscr M}}
\def\sN{{\mathscr N}}

\def\sR{{\mathscr R}}

\def\sW{{\mathscr W}}

\def\geq{\geqslant}
\def\leq{\leqslant}

\def\s{\sigma}
\def\bx{{\bf x}}

\def\v{\vee}

\allowdisplaybreaks

\allowdisplaybreaks
\numberwithin{equation}{section}
\begin{document}
	\title{Tamed Euler Schemes for Singular SDEs with Multiplicative L\'{e}vy Noise}
	\date{}
	\author{ HUA ZHANG ${}^{1}$, MINGBO ZHANG${}^{2,\ast}$}
	
	%\thanks{}
	\dedicatory{
		${}^{1,2}$School of Statistics and data science, Jiangxi University of Finance and Economics\\
		Nanchang, Jiangxi 333000, P. R. China\\
	 H. Zhang: zh860801@163.com;\\	M. Zhang: zhangmb@mail2.sysu.edu.cn }
	
	\footnote[0]{
		
		${}^\ast$Corresponding author}
	
	\date{}

	\begin{abstract}
	We prove strong convergence rates for tamed Euler schemes of multidimensional stochastic differential equations with singular drift, multiplicative Brownian noise, and multiplicative L\'{e}vy noise. The drift is assumed to satisfy a Ladyzhenskaya–Prodi–Serrin type condition, while the diffusion and jump coefficients have Sobolev-type spatial regularity. The proof is based on a nonlocal Zvonkin transform, Krylov-type estimates, and a stochastic Gronwall argument.
	
	The main novelty is an error decomposition that isolates two jump-induced contributions: a compensated jump martingale error and a nonlocal compensator error caused by Euler freezing. This decomposition yields explicit rates and recovers the Brownian-type estimate when the jump coefficient vanishes.
		
	\end{abstract}
	%{\bf AMS Mathematics Subject Classification (2010):} \\
	%Primary: 60H07, 31C25\\
	%Secondary: 28C20, 46G12\\
    \maketitle
	\tableofcontents
	\section{Introduction}

	The purpose of this paper is to develop a strong approximation theory for stochastic differential equations with singular drift, multiplicative Brownian noise, and multiplicative L\'evy noise. More precisely, let
	$(\Omega,\sF,(\sF_t)_{t\geq0},\mP)$
	be a complete filtered probability space. Let $(B_t)_{t\geq0}$ be a $d$-dimensional standard $(\sF_t)$-Brownian motion, and let $N(dt,dz)$ be an $(\sF_t)$-Poisson random measure on $\mR^d$ with intensity measure $dt\upsilon(dz)$, where $\upsilon$ is a L\'evy measure satisfying
	\ce 
	\int_{\mR^d}(|z|^2\wedge 1)\upsilon(dz)<\infty,
	\qquad
	\upsilon(\{0\})=0.
	\de 
	We denote by
	\ce 
	\widetilde N(dt,dz)=N(dt,dz)-dt\,\upsilon(dz)
	\de 
	the compensated Poisson random measure. We consider the time-inhomogeneous stochastic differential equation
	
	\begin{equation}\label{EQU1}
		\begin{aligned}
			dX_t
			&=b(t,X_t)dt+\sigma(t,X_t)dB_t
			+\int_{|z|<R}g(t,X_{t-},z)\widetilde N(dt,dz),\\
			X_0&=x_0,\qquad t\in[0,1],
		\end{aligned}
	\end{equation}
	where $R>0$ is fixed. The drift coefficient
	$b:[0,1]\times\mR^d\to\mR^d$
	is assumed to be Borel measurable and to satisfy a strict Ladyzhenskaya--Prodi--Serrin (LPS) type condition,
	\begin{equation}\label{LPScondition1}
		\|b\|_{\mL_p^q([0,1])}
		:=
		\left(
		\int_0^1
		\left(
		\int_{\mR^d}|b(t,x)|^p\,dx
		\right)^{q/p}
		dt
		\right)^{1/q}<\infty,
		\qquad
		\frac{d}{p}+\frac{2}{q}<1.
	\end{equation}
	The diffusion coefficient $\sigma$ is non-degenerate, while both $\sigma$ and the jump coefficient $g$ may depend on the state variable and possess only Sobolev-type spatial regularity. The main objective of this paper is to construct an explicit tamed Euler--Maruyama scheme for \eqref{EQU1} and to establish strong convergence rates which separate the approximation error of the singular drift from the discretization errors generated by the Brownian and L\'evy components.
	
	SDEs with singular drift coefficients have attracted considerable attention over the last several decades. The fundamental mechanism behind many well-posedness results is the regularization-by-noise phenomenon introduced through Zvonkin's transformation \cite{Zvonkin-1974} and later developed by Veretennikov \cite{Veretennikov-1979}. In the Brownian case with $g\equiv0$ and $\sigma$ equal to the identity matrix, Krylov and R\"ockner \cite{Krylov-Rockner-2005} established strong well-posedness under the Ladyzhenskaya--Prodi--Serrin condition \eqref{LPScondition1}. This theory was subsequently extended to SDEs with multiplicative Brownian noise and Sobolev diffusion coefficients by Zhang \cite{Zhang-2005-1,Xicheng-2011,Zhang-2020}. For SDEs driven by L\'evy noise, the Zvonkin transform leads to parabolic integro-differential equations involving a nonlocal generator. Xie and Zhang \cite{Xie-Zhang_2020} developed a well-posedness theory for SDEs with jumps, singular coefficients, Sobolev diffusion coefficients and multiplicative jump coefficients. These works provide the analytical foundation for studying singular SDEs with Brownian and L\'evy noises.
	
	Despite these well-posedness results, quantitative numerical approximation for \eqref{EQU1} remains far less understood. In the purely Brownian case, several convergence results are available for Euler-type schemes with irregular drifts. When $g\equiv0$ and $\sigma$ is the identity matrix, Jourdain and Menozzi \cite{Benjamin-2021} proved convergence rates for the marginal densities of tamed Euler--Maruyama schemes with truncated drifts. Gy\"ongy and Krylov \cite{Krylov-2021-1} studied convergence in probability for tamed Euler--Maruyama schemes with truncated drifts, although without an explicit rate. More recently, L\^e and Ling \cite{Ling2022} introduced a generic explicit tamed Euler--Maruyama scheme for Brownian SDEs with singular drift under \eqref{LPScondition1}. Their strong convergence rate is expressed in terms of the approximation error of the singular drift measured in a suitable, possibly weak, topology. This work provides a sharp Brownian benchmark for explicit numerical schemes in the LPS singular setting.
	
	There is also a substantial literature on numerical approximation of SDEs with jumps under more regular coefficient assumptions. The classical Euler scheme for L\'{e}vy-driven SDEs was studied by Protter and Talay \cite{Denis-PP-1997}. Weak Euler approximation for non-degenerate L\'{e}vy-driven SDEs was investigated by Mikulevi$\v{c}$ius and Zhang \cite{Zhang-M-R-2021}, while strong Euler approximation for SDEs driven by L\'evy processes was studied by Mikulevi$\v{c}$ius and Xu \cite{Fanhui2018}. For stochastic functional differential equations with jumps, convergence rates of numerical solutions were obtained in \cite{Mao-B-Y-2021}. Truncated Euler--Maruyama methods for SDEs with Poisson jumps and super-linearly growing coefficients were developed by Deng and Mao \cite{Deng-Mao-2019}. Approximation problems for stable SDEs and their invariant measures have also been studied recently; see \cite{chen-xu-2023}. These results cover important classes of jump equations, but they typically rely on Lipschitz, local Lipschitz, H\"older, or other regularity assumptions which exclude the LPS-type singular drift considered in the present paper.
	
	For SDEs with discontinuous or singular drift coefficients, the numerical literature is more limited. Kaneko and Nakao \cite{Kaneko-1988-Nakao} initiated the study of approximation for one-dimensional SDEs with irregular coefficients. Strong rates for SDEs with H\"older continuous drift were obtained in \cite{Taguchi-P-O-2017}. For irregular Brownian SDEs, several recent works have developed refined techniques, including quadrature-type reductions \cite{Michaela-NA-2021}, Gaussian--Besov regularity methods \cite{Bao-Huang-2020}, stochastic sewing arguments \cite{Butkovsky-DKG-2021}, and regularization-by-noise estimates for Euler--Maruyama schemes \cite{Dareiotis-KGM-2020,Dareiotis-kg-2023}. We also mention the work of De Angelis, Germain and Issoglio \cite{Angelis-TGMI-2022} on numerical schemes for one-dimensional SDEs with distributional drift. Compared with these works, the present paper treats multidimensional jump SDEs with LPS-type singular drift, Sobolev diffusion and jump coefficients, and multiplicative L\'evy noise. To the best of our knowledge, no explicit strong convergence rate has been established for this class of equations.

	It is also useful to compare the convergence rates in the present paper with those available in the existing literature. For SDEs with sufficiently regular coefficients, the classical Euler--Maruyama scheme typically attains the strong order $1/2$ in the mean-square sense, and for jump SDEs with Lipschitz-type coefficients the $p$th-moment convergence rate is usually of order $1/p$ for $p\geq2$, see for instance, \cite{Denis-PP-1997,Mao-B-Y-2021,Fanhui2018}. These rates rely essentially on regularity assumptions which are not available under the singular condition \eqref{LPScondition1}. In the Brownian singular-drift setting, the convergence rate depends on the integrability exponents of the drift and on the spatial regularity of the diffusion coefficient. For example, Jourdain and Menozzi \cite{Benjamin-2021} obtained a density convergence rate of order
	\ce
	\frac12-\frac{d}{2p}-\frac1q
	\de
	for a tamed Euler scheme with truncated drift and additive noise, while L\^e and Ling \cite{Ling2022} derived strong convergence rates for multiplicative Brownian SDEs in terms of the drift approximation error and the Brownian discretization error. Our result preserves the Brownian-type contributions appearing in this singular framework and identifies the additional rate loss caused by the multiplicative L\'evy noise. More precisely, for $\bar p\geq2$ the jump-induced nonlocal contribution appears at the order
	\ce
	n^{-(1-\frac1q)(1-\frac dp)},
	\de
	whereas for $1<\bar p<2$ it appears at the order
	\ce
	n^{-\frac{\bar p}{2}(1-\frac dp)}.
	\de
	These additional terms originate from the compensated jump martingale error and the nonlocal compensator error produced by Euler freezing. Thus, the rate obtained here is not intended to compete with the classical order $1/2$ available under Lipschitz or smooth coefficients. Rather, it quantifies the price paid for simultaneously allowing LPS-type singular drift, Sobolev diffusion coefficients, multiplicative jump coefficients, and explicit drift approximation.

	At first sight, the jump part may appear to be only a lower-order perturbation of the Brownian dynamics, since the diffusion coefficient is assumed to be uniformly non-degenerate. Indeed, at the level of parabolic regularity, the nonlocal operator
	\ce
	\mathscr L_{\upsilon,R}^{g}u(t,x)
	:=
	\int_{|z|<R}
	\big[
	u(t,x+g(t,x,z))-u(t,x)-g(t,x,z)\cdot\nabla u(t,x)
	\big]\upsilon(dz)
	\de
	can be absorbed into the estimates for the associated integro-differential equation. However, this perturbative viewpoint is insufficient for the strong error analysis of the Euler scheme. After applying the Zvonkin transform, the multiplicative jump coefficient interacts with the time discretization and gives rise to two error mechanisms which are absent in the purely Brownian case.
	
	The first one is a compensated jump martingale freezing error. It is caused by the difference between the jump coefficient evaluated along the Euler trajectory and its frozen value on the time grid, namely terms of the form
	\ce
	\int_0^t\int_{|z|<R}
	\big[
	g(r,X_{r-}^n,z)-g(r,X_{r_n-}^n,z)
	\big]\widetilde N(dr,dz).
	\de
	Its control requires estimates for compensated Poisson martingales, together with suitable moment bounds for the Euler increments.
	
	The second one is a nonlocal compensator error. It is generated by the interaction between Euler freezing and the transformed jump remainder
	\ce
	u(r,x+g(r,x,z))-u(r,x)-g(r,x,z)\cdot\nabla u(r,x).
	\de
	Equivalently, it is reflected in differences of the form
	\ce
	\int_0^t \mathscr L_{\upsilon,R}^{g}u(r,X_r^n)\,dr
	-
	\int_0^t \mathscr L_{\upsilon,R}^{g}u(r,X_{r_n}^n)\,dr,
	\de
	where $u$ denotes the solution of the Zvonkin equation. To estimate this term, one needs bounds for second-order spatial increments and maximal-function estimates under Sobolev regularity.
	
	Therefore, although the nonlocal generator can be treated perturbatively in the regularity theory of the associated integro-differential equation, the numerical strong error decomposition is genuinely different from that in the Brownian setting. One of the main contributions of this paper is to isolate these two jump-induced errors and quantify their separate contributions to the final strong convergence rate.

	%%%%%%%%%%%%%%%%%%%%%%%%%%%%%%%%%%%%%%%%%%%%%%%%%%%%%
	
	For $n\in\mN$, define
	\ce 
	r_n=\frac{k}{n},
	\qquad
	\frac{k}{n}\leq r<\frac{k+1}{n},
	\qquad
	k=0,1,\ldots .
	\de 
	We approximate \eqref{EQU1} by the explicit scheme
	\begin{equation}\label{EulerAPP3.81}
		\begin{aligned}
			X_t^n
			=&\,x_0^n+\int_0^t b^n(r,X_{r_n}^n)\,dr
			+\int_0^t\sigma(r,X_{r_n}^n)\,dB_r  \\
			&+\int_0^t\int_{|z|<R}
			g(r,X_{r_n-}^n,z)\widetilde N(dr,dz),
			\qquad t\in[0,1],
		\end{aligned}
	\end{equation}
	where $x_0^n$ is an $\sF_0$-measurable approximation of $x_0$ and $b^n$ is a suitable approximation of the singular drift $b$. Since $b$ is merely integrable, the direct choice $b^n=b$ is generally not stable in an explicit Euler scheme. The numerical trajectory may visit regions where the drift is highly singular, and the corresponding drift increment cannot be controlled by standard arguments. Therefore, the drift must be regularized, truncated, or otherwise tamed.
	
	The framework of \eqref{EulerAPP3.81} allows a broad class of drift approximations. Typical examples include the spatially mollified drift
	\begin{equation}\label{truncated0}
		b_r^n(x)=p_{1/n^\chi}*b_r(x),
	\end{equation}
	where $\chi>0$ and $p_t$ denotes the Gaussian density with variance $t$, and the truncated drifts
	\begin{equation}\label{truncated1}
		b_r^n(x)=b_r(x)\mathbf 1_{\{|b_r(x)|\leq Cn^\chi\}},
	\end{equation}
	or
	\begin{equation}\label{truncated2}
		b_r^n(x)=b_r(x)\mathbf 1_{\{|b_r(x)|\leq Cn^\chi |b_r|_{L^p(\mR^d)}\}},
	\end{equation}
	for suitable constants $C,\chi>0$. The same formulation also includes multiresolution approximations based on wavelets \cite{Meyer-1992} and truncated discrete $\varphi$-transforms \cite{Frazier-1991}. In this broad sense, we call \eqref{EulerAPP3.81} a tamed Euler--Maruyama scheme. This terminology is inspired by the tamed scheme introduced by Hutzenthaler, Jentzen and Kloeden \cite{Hutzenthaler-2012} for SDEs with regular but super-linearly growing coefficients, and it also covers truncated Euler--Maruyama schemes of the type studied in \cite{Mao-2015}.
	
	The main result of this paper is an explicit strong convergence theorem for \eqref{EulerAPP3.81}. The convergence rate is expressed in terms of the approximation error of $b^n$ to $b$ in a topology adapted to the singular nature of the drift. This formulation has two advantages. First, it separates the deterministic approximation error of the singular drift from the stochastic discretization errors. Second, it allows one to obtain concrete rates for different choices of $b^n$, including mollified, truncated, wavelet, and $\varphi$-transform approximations. In particular, since convergence in the natural space $\mL_p^q([0,1])$ does not by itself provide an explicit rate, we use suitable negative-order spaces $\mH_{-\theta,p}^q([0,1])$, with $\theta\in[0,1)$, to quantify weak approximation errors of the drift.
	
	Our strong error estimate possesses a transparent jump-sensitive structure. More precisely, the total approximation error is decomposed into five contributions: the initial error, the drift approximation error, the Brownian discretization error, the compensated jump martingale error, and the nonlocal compensator error. The last two terms are genuinely induced by the multiplicative L\'evy noise and have no analogue in the purely Brownian setting. They are explicitly quantified by the norms of the jump coefficient appearing in the assumptions on $g$. Consequently, when the jump coefficient vanishes, these two terms disappear and the estimate reduces exactly to the Brownian result. This decomposition demonstrates that, although the nonlocal operator acts as a lower-order perturbation at the PDE level, it produces genuinely new discretization errors in the strong approximation analysis. In particular, the additional jump-induced errors arise only after combining the nonlocal Zvonkin transform with the Euler freezing procedure, and therefore cannot be detected from the PDE regularity theory alone.

The main contributions of this paper are summarized as follows. First, we develop a Zvonkin-transform framework for the strong error analysis of explicit Euler approximations for singular SDEs driven by multiplicative L\'evy noise, which separates the Brownian and jump-induced discretization errors at the transformed level. Second, based on this decomposition, we establish explicit strong convergence rates for a tamed Euler scheme with a general drift approximation sequence $b^n$. Third, we identify two genuinely jump-specific error mechanisms, namely the compensated jump martingale freezing error and the nonlocal compensator freezing error, and quantify their contributions explicitly in the final convergence estimate. Finally, the analysis is based on new estimates for nonlocal increments, discrete Krylov estimates for Euler approximations, maximal regularity of parabolic integro-differential equations with distributional drifts, pathwise Burkholder--Davis--Gundy estimates for compensated Poisson integrals, and a stochastic Gronwall argument adapted to the jump setting.

	The proof contains several ingredients. The nonlocal operator requires estimates for second-order increments.These estimates are used to control the nonlocal compensator error. Since the density of the discretized process $X^n$ is not directly available, even in the drift-free case, we establish Krylov-type estimates for the Euler approximation by exploiting the independence structure of the increments. The drift approximation error is treated through parabolic integro-differential equations with distributional forcing terms. The compensated jump martingale terms are controlled by pathwise Burkholder--Davis--Gundy type inequalities, while the final stability estimate is obtained through a stochastic Gronwall argument. These ingredients together provide a quantitative approximation theory for \eqref{EQU1} under the singular condition \eqref{LPScondition1}.
	
	The rest of the paper is organized as follows. Section 2 introduces the functional spaces, assumptions, and main results, and also collects several analytic and probabilistic tools used throughout the paper. Section 3 proves key estimates for nonlocal operators and discretized approximations, which are essential for treating the jump terms. Section 4 establishes Krylov-type estimates for the Euler--Maruyama scheme, first in the drift-free case and then in the general case through a measure transformation. Section 5 develops a priori estimates for parabolic integro-differential equations with distributional forcing terms. Sections 6 and 7 study the path properties of the Euler--Maruyama approximation and of the exact solution, respectively. Finally, Section 8 proves the main convergence theorem.

\section{Preliminaries and Main results}
To present our primary findings, we begin by defining certain spaces and notations.  Let $p,q\in [1,\infty]$ and $s,t\in[0,1]$ with the condition that $s\leq t$. We denote by  $\mL_p^q([s,t])$ the space comprising all Borel functions defined on $[s,t]\times \mR^d$ equipped with the norm defined as:
\ce 
\|f\|_{\mL_p^q([s,t])}:=\bigg(\int_s^t\bigg(\int_{\mR^d}|f(r,x)|^pdx\bigg)^{\frac{q}{p}}\,dr\bigg)^{\frac{1}{q}}<\infty,~\frac{d}{p}+\frac{2}{q}<1.
\de 
For the cases where  $p=\infty$ and $q=\infty$, the aforementioned norm is interpreted as the standard $\mL^{\infty}-$norm denoted by $\|\cdot\|_{\mL^{\infty}([s,t])}$. To simplify notation, we will denote:
\ce 
\mL^p([s,t]):=\mL_p^p([s,t]).
\de 

For each $\theta\in\mR$, , we define the Bessel potential space $\mH_{\theta,p}(\mR^d)$ as $(\mI-\Delta)^{-\theta/2}(L^p(\mR^d))$ equipped with the norm: 
\ce 
\|f\|_{\mH_{\theta,p}(\mR^d)}:=\|(\mI-\Delta)^{\theta/2}f\|_{p},
\de 
where $\|\cdot\|_p$ is the usual $L_p-$norm of Lebesgue space $L_p(\mR^d)$. The operator $(\mI-\Delta)^{\theta/2}$ is defined via the Fourier transform $\cF$ and its inverse $\cF^{-1}$ as 
\ce 
(\mI-\Delta)^{\theta/2}f:=\cF^{-1}((1+|\cdot|^2)^{\theta/2}\cF f).
\de 
And we denote $\mH_{\alpha,p}^q([0,t]) := L^q([0,t]; \mH_{\alpha,p}(\mR^d))$.

For $r\in(0,2]$, $u\in C_0^{\infty}(\mR^d)$ denote 
\ce 
\partial^ru(x)=\cF^{-1}[|\xi|^r\cF u(\xi)](x).
\de 
 It is worth noting that  for  $n\in \mN$ and $p\in(1,\infty)$ an equivalent norm in $\mH_{n,p}$ is given by (cf. \cite{Stein1970})
\ce 
\|f\|_{\mH_{n,p}(\mR^d)}=\|f\|_p+\|\nabla^n f\|_p.
\de 
Denote by $W^{\alpha, p}(\mR^d)$ the completion of the space $C_0^{\infty}(\mR^d)$ with respect to the norm, 
\ce 
\|f\|_{W^{\alpha, p}(\mR^d)}=\|f\|_p+\|\nabla^\alpha f\|_p.
\de 
For $1<p<\infty $, then  $W^{\alpha, p}(\mR^d)=\mH_{\theta,p}(\mR^d)$ (cf. \cite[Theorem]{Shieh2014}).

It can be easily seen that if  $r\in(0,2)$ there exists a constant $N_r$ such that 
\ce 
\partial^ru(x)=N_r\int \bigg[u(x+y)-u(x)-\nabla u(x)\cdot y(\textbf{1}_{|y|\leq 1}\textbf{1}_{r=1}+\textbf{1}_{1<r<2})\bigg]\frac{dy}{|y|^{d+r}},
\de 
and 
\ce 
\partial^2u(x)=\sum_{i,j=1}^{d}\partial_{x_i,x_j}^2u(x).
\de 

For a measurable function $g(t,x,z):\mathbb{R}_+\times \mathbb{R}^d\times \mathbb{R}^d\to \mathbb{R}^d$,
and $0\leq \epsilon< R\leq \infty$, we introduce the following functions, which will be used frequently below: for $j=0,1$ and $\eta\geq 1$
\begin{equation}\label{formula1}
	\varGamma_{j,\eta}^{\epsilon,R}(g)(t,x):=	\varGamma_{j,\eta}^{\epsilon,R}(g(t))(x):=\|\nabla_x^jg(t,x,\cdot)\|_{L^{\eta}(B_R\setminus B_{\epsilon}; \upsilon)}^\eta:=\int_{\epsilon\leq |z|<R}|\nabla_x^jg(t,x,z)|^{\eta}\upsilon(dz),
\end{equation}
where $B_R:=\{x\in\mathbb{R}^d:|x|<R\}$ for given $R>0$.

For a $d\times d$-matrix $A$, $A^T$ denotes its transpose and $\|A\|$ denotes its Hilbert-Schmidt norm.

Put $D_n=\{i/n:i=0,1,\cdots,n\}$. For each $S\leq T$, we put $\Delta([S,T])=\{(s,t)\in[S,T]^2:s\leq t\}$ and $\Delta_2([S,T])=\{(s,u,t)\in[S,T]^3:s\leq u\leq  t\}$. We abbreviate $\Delta=\Delta([0,1])$ and $\Delta_2=\Delta_2([0,1])$. We say that a function $\sW:\Delta([S,T])\to [0,\infty)$ is a control 
provided $\sW(s,u)+\sW(u,t)\leq \sW(s,t)$ for every $(s,u,t)\in \Delta_2([S,T])$. 

We give some examples of continuous controls as follows.
\begin{Examples}\label{Examples2.1}
	(1) For any $f\in \mL^q([0,1])$, $q\in [1,\infty)$, then $\sW(s,t):=\|f\|_{ \mL^q([s,t])}^q$ is a continuous control on $\Delta([0,1])$.
	
	(2)For any $\vartheta\geq 0$, $\sW(s,t):=s^{-\vartheta}(t-s)$ is a continuous control on $\Delta([S,T])$ for any $0<S\leq T$.
	
	(3) For any controls $\sW_1$,  $\sW_2$ on $\Delta([S,T])$ and any number $\theta\in [0,1]$, then  $\sW(s,t):=\sW_1^\theta  \sW_2^{1-\theta}$ is another control on $\Delta([S,T])$.
	
	(4) For any controls $\sW_1$,  $\sW_2$ on $\Delta([S,T])$ and some constants $1\leq \eta$, then $\sW(s,t):=\sW_1+ \sW_2^{\eta}$ is another control on $\Delta([S,T])$ (see \cite[Excersice1.10]{{Friz-2010}}).
\end{Examples}

The following conditions are enforced throughout unless noted otherwise.

Condition ($\sH_\sigma^g$): The diffusion coefficient $\sigma$ is a $d\times d$-matrix-valued measurable function on $[0,1]\times \mR^d$.  We assume that the following conditions hold.
\begin{enumerate}
	\item [(i)] 
	There exists a constant $c_0\in [1,\infty)$ such that for every $t\in [0,1]$ and $x\in\mR^d$
	\begin{equation}
		c_0^{-1}|\xi|\leq |\sigma^T(t,x)\xi|\leq c_0|\xi|,~\forall \xi\in\mathbb{R}^d.
	\end{equation}
	\item [(ii)] There are  constants $\alpha\in(0,1]$ and $c_1\in (0,\infty)$ such that for every $t\in [0,1]$ and $x,y\in\mR^d$,
	\ce 
	|(\sigma\sigma^T)(t,x)-(\sigma\sigma^T)(t,y)|\leq c_1|x-y|^\alpha,
	\de 
	\item [(iii)] Let $\varGamma_{j,\alpha}^{0,R}(g)$ be defined as in (\ref{formula1}). For some  $\beta\in [1,2)$, 
	\begin{equation}\label{formula0+1}
		\varGamma_{0,\beta}^{0,R}(g)\in \mathbb{L}^{\infty}([0,1]),~~~ \varGamma_{0,2}^{0,R}(g)\in{\mL^{\infty}([0,1])}.
	\end{equation}
\end{enumerate}

Condition ($\widehat{\sH}_\sigma^g$):  We assume that 
 $\sigma(s,\cdot)$ is weakly differentiable for a.e. $s\in[0,1]$ and for some $p_0,q_0\in(2,\infty)$ with $\frac{d}{p_0}+\frac{2}{q_0}<1$,
\begin{equation}
|\nabla \sigma|\in\mL_{p_0}^{q_0}([0,1]),
\end{equation}
and
\begin{equation}
	\|\varGamma_{1,2}^{0,R}(g)\|_{\mL^{\infty}([0,1])}<\infty.
\end{equation}

Condition($\sH_b$): Let $b$ belong to $\mL_{p}^q([0,1])$ for some $p\in (2(1-\alpha)^{-1}(d/\beta\vee 1)\vee 2, \infty)$, $q\in [2,\infty)$  satisfying $\frac{d}{p}+\frac{2}{q}<1$. For each $n$, let
$b_n$ belong to $\mL_p^q([0,1])\cap \mL_{\infty}^q([0,1])$ with $p,q$  as specified above. Furthermore, there exist finite positive constants $c_2>0$ and $\theta$, as well as continuous controls $\{m_n\}_n$ such that 
\ce 
\sup_{n\geq 1}(\|b_n\|_{\mL_p^q([0,1])}+m_n(0,1))\leq c_2
\de 
and 
\ce 
(1/n)^{\frac{1}{2}-\frac{1}{q}}\|b_n\|_{\mL_{\infty}^{q}([s,t])}\leq m_n(s,t)^\theta.
\de 
 
 We always assume that both $x_0$ and, for each $n$, $x_0^n$ belong to the space $L_p(\Omega, \sF_0)$, where the exponent $p$ is as above. 

\begin{remark}
	If $g_t(x,z)=\hat{\sigma}_t(x)(1\wedge z^{\frac{2}{\beta}})$ with $\hat{\sigma}_t(x)\in \mL^{\infty}([0,T])$ and $\nabla\hat{\sigma}_t(x)\in \mL^{\infty}([0,T])$ in the above conditions, then the assumptions on $\varGamma_{j,2}^{0,R}(g)$ and $\varGamma_{0,\beta}^{0,R}(g)$ automatically hold.
\end{remark}

Not only are we interested in establishing that the tamed Euler-Maruyama scheme converges in probability, but we also aim to explore its convergence rate. To accomplish this, it becomes necessary to control the difference between $b^n$ and $b$.

Let  $u$ be a function on $\mR^d$. We define $\sL_2^\sigma$ as the second-order differential operator associated with the diffusion coefficient  $\sigma$. Specifically, it is given by: 
\ce 
\sL_2^\sigma u(x):=\frac{1}{2}(\sigma\sigma^T)^{ij}\partial_{ij}^2 u(x).
\de 
Here and affter, we adopt Einstein's summation convention, where repeated indices in a product are summed over automatically.

Next, we define $\sL_1^b$ as the first-order differential operator associated with the drift coefficient $b$, 
\ce 
\sL_1^b u(x):=b\cdot \nabla u(x).
\de 
Furthermore, we introduce the nonlocal operator $\sL_\upsilon^g$ associated with the jump coefficient $g$. This operator is defined as:
\ce 
\sL_\upsilon^g u(x)&:=&\int_{|z|<R}[u(x+g(x,z))-u(x)-g(x,z)\cdot \nabla u]\upsilon(dz)\\
&&+\int_{|z|\geq R}[u(x+g(x,z))-u(x)]\upsilon(dz)=:\sL_{\upsilon,R}^g u(x)+\bar{\sL}_{\upsilon,R}^g u(x).
\de 
Let $\lambda\geq 0$ be a fixed number. We consider the following backward equation system, which arises from a Zvonkin transformation,
\be\label{PDE_b}
\partial_t u+\sL_2^\sigma u(x)+\sL_1^b u(x)+\sL_\upsilon^g u(x)=\lambda u-b,~~u(1,\cdot)=0.
\ee

Next, we define an important quantity that controls the strong convergence rate.
\begin{definition}\label{definition_varpi1}
	Let $U=(u^1,\cdots,u^d)$ where for each $k=1,\cdots,d$, $u^k$ is the solution to the following partial differential equation
	\ce 
	\partial_t u+\sL_2^\sigma u+\sL_1^{b_n} u+\sL_{\upsilon,R}^g u=\lambda u-b^{n,k},~~u(1,\cdot)=0,
	\de 
	where $b^{n,k}$ is the $k$-th component of $b^n$. Let $X$ be the solution to Eq. (\ref{EQU1}). For each $\bar{p}\in[1,\infty)$, we put 
	\ce 
	\varpi_b^n(\bar{p})=\|\sup_{t\in[0,1]}|\int_0^t(\mI+\nabla U)[b-b^n](r,X_r)dr|\|_{L^{\bar{p}}(\Omega)}.
	\de 
\end{definition}

Now, let us present the first main results as follows.

\begin{theorem}\label{M-theorem2.2}
	Assume that Conditions $\sH_\sigma^g$, $\sH_b$ and $\widehat{\sH}_\sigma^g$ hold. 
	Let $(X_t^n)_{t\in[0,1]}$ be the solution to Eq. \eqref{EulerAPP3.81}, and let $(X_t)_{t\in[0,1]}$ be the solution to Eq. \eqref{EQU1}. 
	Let
	\ce
	\bar p\in(1,p)\cap\left(1,\frac{2}{d}\left(	p_0\wedge p\wedge \sqrt{pd)}	\right)\right),
	\qquad
	\theta\in(0,1).
	\de
	Assume further that
	\ce
	\varGamma_{0,\bar p}^{0,R}|g|,
	\ \varGamma_{1,\bar p}^{0,R}|g|,
	\in
	\mL_\infty^q([0,1]),
	\qquad
	\varGamma_{1,1}^{0,R}|g|
	\in
	\mL_p^q([0,1]).
	\de
	Define the jump-dependent constants
	\ce
	C_J^{(m)}
	:=
	\|\varGamma_{1,\bar p}^{0,R}|g|\|_{\mL_\infty^q([0,1])}^{\bar p}
	+
	\|\varGamma_{1,2}^{0,R}|g|\|_{\mL_\infty^q([0,1])}^{\bar p},
	\de
	and
	\ce
	C_J^{(c)}
	:=
	\|\varGamma_{1,1}^{0,R}|g|\|_{\mL_p^q([0,1])}^{\bar p}.
	\de
	
	If $\bar p\ge2$, set
	\ce
	\mathcal E_{B,n}
	:=
	\left(
	n^{-\alpha/2}
	+
	n^{-1/2}\log n
	\right)^{\bar p}
	+
	n^{-\alpha(1-\frac1q)},
	\de
	\ce
	\mathcal E_{J,m,n}
	:=
	C_J^{(m)}n^{-(1-\frac1q)},
	\qquad
	\mathcal E_{J,c,n}
	:=
	C_J^{(c)}n^{-(1-\frac1q)(1-\frac dp)},
	\de
	and
	\ce
	\mathcal E_n
	:=
	\mE |x_0-x_0^n|^{\bar p}
	+
	\big(\varpi_b^n(\bar p)\big)^{\bar p}
	+
	\mathcal E_{B,n}
	+
	\mathcal E_{J,m,n}
	+
	\mathcal E_{J,c,n}.
	\de
	Then there exists a finite positive constant
	\ce
	N=N(c_0,c_1,c_2,\alpha,\beta,\theta,p,q,p_0,q_0,\bar p,d)>0
	\de
	such that
	\begin{equation}\label{M-theorem2.2-estimate-large-p}
		\mE\left[
		\sup_{t\in[0,1]}|X_t^n-X_t|^{\theta\bar p}
		\right]
		\le
		N\mathcal E_n^\theta .
	\end{equation}
	
	If $1<\bar p<2$, set
	\ce
	\mathcal E_{B,n}^{<2}
	:=
	\left(
	n^{-\alpha/2}
	+
	n^{-1/2}\log n
	\right)^{\bar p}
	+
	n^{-\alpha(1-\frac1q)},
	\de
	\ce
	\mathcal E_{J,m,n}^{<2}
	:=
	C_J^{(m)}n^{-\bar p/2},
	\qquad
	\mathcal E_{J,c,n}^{<2}
	:=
	C_J^{(c)}n^{-\frac{\bar p}{2}(1-\frac dp)},
	\de
	and
	\ce
	\mathcal E_n^{<2}
	:=
	\mE |x_0-x_0^n|^{\bar p}
	+
	\big(\varpi_b^n(\bar p)\big)^{\bar p}
	+
	\mathcal E_{B,n}^{<2}
	+
	\mathcal E_{J,m,n}^{<2}
	+
	\mathcal E_{J,c,n}^{<2}.
	\de
	Then
	\begin{equation}\label{M-theorem2.2-estimate-small-p}
		\mE\left[
		\sup_{t\in[0,1]}|X_t^n-X_t|^{\theta\bar p}
		\right]
		\le
		N\left(\mathcal E_n^{<2}\right)^\theta .
	\end{equation}
\end{theorem}

\begin{remark}
	The term $\mathcal E_{B,n}$ is the Brownian-type contribution. 
	It comes from the diffusion discretization, the frozen drift estimate, and the stochastic Gronwall argument. 
	The term $\mathcal E_{J,m,n}$ comes from the compensated jump martingale freezing error, while $\mathcal E_{J,c,n}$ comes from the nonlocal compensator freezing error. 
	In particular, these two jump-induced contributions vanish when the spatial variation of the jump coefficient is absent. 
	This separation is used explicitly in the proof of Proposition~\ref{prop:jump-specific-estimates}.
\end{remark}

\begin{remark}
For conciseness, we restrict our analysis to the case of small jumps. Consider the processes $X_t^n$ and $X_t$ defined respectively by the following SDE:
	\ce 
	X_t&=&x_0+\int_{0}^{t}b(r,X_r)dr+\int_{0}^{t}\sigma(r,X_r)dB_r
	+\int_{0}^{t}\int_{|z|< R}g(r,X_{r-},z)\widetilde{N}(dr\,dz)\\
	&&+\int_{0}^{t}\int_{|z|\geq R}g(r,X_{r-},z)N(dr\,dz)
	\de 
	and 
	\ce 
	X_t^{n}(x)&=&x_0^n+\int_{0}^{t}b^n(r,X_{r_n}^{n}(x))dr+\int_{0}^{t}\sigma(r,X_{r_n}^{n}(x))dB_r+\int_{0}^{t}\int_{|z|<R}g(r,X_{{r_n-}}^{n}(x),z)\widetilde{N}(dr\,dz)\\
	&&+\int_{0}^{t}\int_{|z|\geq R}g(r,X_{{r_n-}}^{n}(x),z)N(dr\,dz).
	\de 
It is worth emphasizing that our framework retains theoretical validity for the broader case incorporating both small and large jumps.
\end{remark}

\begin{theorem}\label{M-theorem2.4}
	Assume that Conditions $\sH_\sigma^g$, $\sH_b$ and $\widehat{\sH}_\sigma^g$ hold. Further assume that 
	
	(i) Let $p_1, q_1\in [1,\infty]$ be such that $\frac{d}{p_1}+\frac{2}{q_1}<2$. Then for every $m\geq 1$, there exists a constant $N(c_0,c_1,c_2,\alpha,\beta,\theta,p,q,p_0,q_0,p_1,q_1,\bar{p},d)>0$ such that 
	\ce 
	\varpi_b^n(m)\leq N\|b-b^n\|_{\mL_{p_1}^{q_1}([0,1])}.
	\de 
	
	(ii) Let $q_0=\infty$ and $p>1$. Let $\theta\in[0,1)$ be such that 
	\be \label{condition_theta1}
	\theta<\frac{3}{2}-\frac{d}{2p}-\frac{2}{q}.
	\ee 
	Then for every $\bar{p}\in [1,p)$ and $n\geq 1$, there exists a constant $N$ depending on $C$ such that 
	\ce 
	\varpi_b^{n}(\bar{p})\leq N\|b-b^n\|_{\mH_{-\theta,p}^{q}([0,1])}.
	\de 
	
(iii) Assume furthermore that $\frac{d}{p}+\frac{4}{q}<1$ and that there exists a continuous control $\sW_0$ on $\Delta$ and a constant $\Pi\geq 0 $ such that 
	\be\label{M-condtion-2.7}
	\|b-b^n\|_{\mH_{-1,p}^q([0,1])}\leq \Pi\sW_0(s,t)^{\frac{1}{q}}
	\ee 
	and 
		\be\label{M-condtion-2.8}
	\|b-b^n\|_{\mL_{p}^q([0,1])}\leq \Pi\sW_0(s,t)^{\frac{1}{q}}
	\ee 
	for every $(s,t)\in \Delta$. Then for every $\bar{p}\in [1,p)$, there exists a constant $N$ depending on $p$ such that 
	\ce 
	\varpi_b^{n}(\bar{p})\leq N\Pi(1+|\log\Pi|)\sW_0(s,t)^{\frac{1}{q}}.
	\de 
\end{theorem}
\begin{remark}
	The restriction to the unit time interval in Theorems \ref{M-theorem2.2} and \ref{M-theorem2.4} is admittedly artificial, and it is straightforward to generalize the aforementioned results to arbitrary finite time intervals. In such cases, the constants in our estimates will additionally depend on the length of the time interval.
\end{remark}
\begin{corollary}
  Assume that Conditions $\sH_\sigma^g$, $\sH_b$ and $\widehat{\sH}_\sigma^g$ hold, and let $\bar{p}$ and $\theta$ be as defined in Theorem \ref{M-theorem2.2}.
  
  (a) Let $C>0$ and $\chi\in(0,1/2-1/q]$ be constants, and define $b^n$ by (\ref{truncated2}). Let $\rho\in (1,p]$ be a number such that $\rho\frac{d}{p}+\frac{2}{q}<2$. Then there exists a constant $$N(c_0,c_1,c_2,\alpha,\beta,\theta,p,q,p_0,q_0,\rho,\bar{p},d)>0$$ such that 
  	\be \label{truncedapp1}
  &&\mE[\sup_{t\in[0,1]}|X_t^n-X_t|^{\theta\bar{p}}]\no\\
  &\leq& N\bigg\{ \mE[|x_0-x_0^n|^{\bar{p}}]+\bigg((1/n)^{\frac{\alpha}{2}}+(1/n)^{\chi(\rho-1)}\no \\
 && +(1/n)^{\frac{1}{2}}\log(n)\bigg)^{\bar{p}}+(1/n)^{\alpha(1-\frac{1}{q})}+(1/n)^{(1-\frac{1}{q})(1-\frac{d}{p})}\bigg\}^\theta,~~\mbox{if}~\bar{p}\geq 2,
  \ee 
  and 
  	\be \label{truncedapp2}
  &&\mE[\sup_{t\in[0,1]}|X_t^n-X_t|^{\theta\bar{p}}]\no\\
  &\leq& N\bigg\{ \mE[|x_0-x_0^n|^{\bar{p}}]+\bigg((1/n)^{\frac{\alpha}{2}}+(1/n)^{\chi(\rho-1)}\no \\
  && +(1/n)^{\frac{1}{2}}\log(n)\bigg)^{\bar{p}}+(1/n)^{\alpha(1-\frac{1}{q})}+(1/n)^{\frac{\bar{p}}{2}(1-\frac{d}{p})}\bigg\}^\theta,~~\mbox{if}~1<\bar{p}< 2.
  \ee 
 
  (b)  Let $C>0$ and $\chi\in(0,3/2-2/q]$ be constants, and define $b^n$ by (\ref{truncated1}). Then there exists a constant $N$ depending on $p$ such that (\ref{truncedapp1}) holds for any $\rho\in (1,p\vee q]$ satisfying $\rho\bigg(\frac{d}{p}+\frac{2}{q}\bigg)<2$.
  
  (c) Let $\chi\in\bigg(0,\frac{p}{d}\bigg(1-\frac{2}{q}\bigg)\bigg]$ be constants and define $b^n$ by (\ref{truncated0}). Let $\rho\in (0,1]$ be any number such that (\ref{condition_theta1}) holds. Then there exists a constant $N$ depending on $p$ such that 
  	\ce
  &&\mE|\sup_{t\in[0,1]}|X_t^n-X_t|^{\theta\bar{p}}|\\
  &\leq& N\bigg\{ \mE|x_0-x_0^n|^{\bar{p}}+\bigg((1/n)^{\frac{\alpha}{2}}+(1/n)^{\chi(\rho/2)}\no \\
  && +(1/n)^{\frac{1}{2}}\log(n)\bigg)^{\bar{p}}+(1/n)^{\alpha(1-\frac{1}{q})}+(1/n)^{(1-\frac{1}{q})(1-\frac{d}{p})}\bigg\}^\theta,~~\mbox{if}~\bar{p}\geq 2,
  \de 
  and 
  \ce
  &&\mE|\sup_{t\in[0,1]}|X_t^n-X_t|^{\theta\bar{p}}|\\
  &\leq& N\bigg\{ \mE|x_0-x_0^n|^{\bar{p}}+\bigg((1/n)^{\frac{\alpha}{2}}+(1/n)^{\chi(\rho/2)}\no \\
  && +(1/n)^{\frac{1}{2}}\log(n)\bigg)^{\bar{p}}+(1/n)^{\alpha(1-\frac{1}{q})}+(1/n)^{\frac{\bar{p}}{2}(1-\frac{d}{p})}\bigg\}^\theta,~~\mbox{if}~1<\bar{p}<2.
  \de 
\end{corollary}
\begin{proof}
Given Theorem \ref{M-theorem2.2}, our task reduces to estimating $\varpi_b^{n}(\bar{p})$. The subsequent portion of the proof closely mirrors the methodology employed in  \cite[Corollary 2.4]{Ling2022}. For brevity, we elect to omit this segment of the proof. 
\end{proof}
\begin{remark}
	The logarithmic factor present in Theorem \ref{M-theorem2.2} originates from the application of the stochastic Davie-Gronwall lemma, particularly when dealing with critical exponents (as detailed in Lemma 3.1 of this work). Insights from \cite{Dareiotis-KGM-2020}, which provide explicit estimations for square moments, hint at the potential for refining this logarithmic factor in Theorem \ref{M-theorem2.2}. Given the pivotal role that the stochastic Davie-Gronwall lemma plays in the analysis of rough and stochastic ordinary/partial differential equations, determining the precise nature of this logarithmic factor constitutes a significant problem. Nevertheless, this particular aspect is not delved into further in the current study.
\end{remark}

\textbf{Convention .} Whenever convenient, we place temporal variables in the subscript immediately after the function, e.g. $f_t(x)=f(t,x)$.  The relation $A\lesssim B$ means that $A\leq CB$ for some constant $C\geq 0$. The implicit constant $C$ may change from one inequality to another, and its values may depend on other parameters that are clear from the context.

	For any one-parameter process $t\mapsto A_t$ and any two-parameter process $(s,t)\mapsto J_{s,t}$, we denote $\delta A_{s,t}=A_t-A_s$ and $\delta J_{s,u,t}=J_{s,t}-J_{s,u}-J_{u,t}$ for every $s\leq u\leq t$. We say $J$ is adapted if $J_{s,t}$ is $\sF_t$ measurable whenever $s\leq t$; The process $J$ (resp. $A$) called $L_{m}$-integrable means that  $\|J_{s,t}\|_{L_m(\Omega)}$ (resp.  $\|A_{t}\|_{L_m(\Omega)}$) is finite for each $(s,t)$(resp. $t$). Let $u\in[0,1]$ and let $\mP|\sF_u$ be the probability measure conditional on $\sF_u$. We denote by $L_p(\Omega|\sF_u)$
    the space of random variables $\xi$ such that 
    $$\|\xi\|_{L_p(\Omega|\sF_u)}:=\mbox{ess~sup}_{\omega}[\mE(|\xi|^p|\sF_u)]^{\frac{1}{p}}<\infty.$$
    
    The next result is a variant of stochastic Davie-Gronwall lemma from \cite{Le2021Friz}. 
  \begin{lemma}[Stochastic Davie-Gronwall lemma] \label{DGlemma1}
  	Let $\epsilon_j>0$, $j=1,\cdots,5$; $\Gamma_i\geq 0$, $i=1,2,3$; $C_1,C_2,C_3,u,S,T\geq 0$ be fixed numbers such that $0\leq u<S<T$. Let $\sW$ be a determinstic control on $\Delta([S,T])$ which is continuous. Let $J$
  	be a $L_p$-integrable adapted process indexed by $\Delta([S,T])$ such that 
  	\be 
  	\label{sewing_comdition_1}&&\|J_{s,t}\|_{L_p(\Omega|\sF_u)}\leq C_1\sW(s,t)^{\frac{1}{2}+\epsilon_1},~~ \|\mE_s J_{s,t}\|_{L_p(\Omega|\sF_u)}\leq C_2\sW(s,t)^{1+\epsilon_2},\\
  	\label{sewing_comdition_2}&&\|\delta J_{s,v,t}\|_{L_p(\Omega|\sF_u)}\leq \Gamma_1\sW(s,t)^{\frac{1}{2}}+\Gamma_2(\sW(s,t)^{\frac{1}{2}+\epsilon_3}+C_3\sW(s,t)^{\frac{1}{2}+\epsilon_4}),\\
  	\label{sewing_comdition_3}&&\|\mE_s\delta J_{s,v,t}\|_{L_p(\Omega|\sF_u)}\leq \Gamma_3\sW(s,t)^{\frac{1}{2}+\epsilon_5} 
  	\ee 
  	for every $(s,v,t)\in \Delta_2([S,T])$. Then there exists a constant $N=N(p,\epsilon)$ such that for every $(s,t)\in \Delta([S,T])$ and every integer $m\geq 1$
  	\be \label{sewing_lemma}
  	\|J_{s,t}\|_{L_p(\Omega|\sF_u)}&\leq& N[C_12^{-m\epsilon_1}\sW(s,t)^{\frac{1}{2}+\epsilon_1}+C_22^{-m\epsilon_2}\sW(s,t)^{1+\epsilon_2}+ m\Gamma_1\sW(s,t)^{\frac{1}{2}}\no\\
  	&&+\Gamma_2(\sW(s,t)^{\frac{1}{2}+\epsilon_3}+\sW(s,t)^{\frac{1}{2}+\epsilon_4})+\Gamma_3\sW(s,t)^{\frac{1}{2}+\epsilon_5}].
  	\ee 
  \end{lemma}  
\begin{lemma}[Stochastic Gronwall's inequality]\label{Stochastic Gronwall inequality}
	Let $\eta_t$, $V_t$ be nonnegative, nondecreasing processes, with $\eta_t$ being a càdlàg process. Let $A_t^i (i=1,2)$ be continuous, nondecreasing,  $\sF_t$-adapted processes such that $A_0^i=0 (i=1,2)$. Furthermore, let $M_t$ be $\sF_t$-local martingale with $M_0=0$. Suppose that there exists a constant $\theta\in (0,\infty)$  such that, with probability one,
	\be 
	\eta_t\leq \bigg(\int_0^t\eta_s^{1/\theta}dA_s^1\bigg)^{\theta}+\int_0^t\eta_sdA_s^2+M_t+V_t,~~\forall t\geq 0.
	\ee
	Then for any $0<p<1$ and  any bounded stopping time $\tau$, we have 
		\ce 
	\mE \eta_{\tau}^p\leq \bigg(\mE \exp\bigg[\frac{p}{1-p}(2^{1/\theta+1} \theta^{1/\theta}A_\tau^1+A_\tau^2)\bigg]\bigg)^{1-p}(\mE V_{\tau})^p, ~\mbox{when}~ \theta\leq 1
	\de 
	and 
	\ce 
	\mE \eta_{\tau}^p\leq \bigg(\mE \exp\bigg[\frac{p}{1-p}(\theta^{1/\theta} (A_\tau^1)^\theta+A_\tau^2)\bigg]\bigg)^{1-p}(\mE V_{\tau})^p, ~\mbox{when}~ \theta> 1.
	\de 
\end{lemma}
\begin{proof}
	Let 
	\be 
	\bar{\eta}_t=\bigg(\int_0^t\eta_t^{1/\theta}dA_s^1\bigg)^{\theta}+\int_0^t\eta_sdA_s^2+M_t+V_t
	\ee 
	and 
	\ce 
	\bar{A}_t^2=\int_0^t\eta(s)/\bar{\eta}(s)dA_s^2.
	\de 
	By It\^{o} formula, one has 
	\ce 
	e^{-\bar{A}_t^2}\bar{\eta}(t)&=&V_0+\theta\int_0^te^{-\bar{A}_r^2}\bigg(\int_0^r\eta_s^{1/\theta}dA_s^1\bigg)^{\theta-1}\eta_r^{1/\theta}dA_r^1+\int_0^te^{-\bar{A}_r^2}dM_r+\int_0^te^{-\bar{A}_r^2}dV_r\\
	&\leq &V_0+\theta\int_0^t\bigg(\int_0^r(e^{-\bar{A}_s^2}\eta_s)^{1/\theta}dA_s^1\bigg)^{\theta-1}(e^{-\bar{A}_r^2}\eta_r)^{1/\theta}dA_r^1+\int_0^te^{-\bar{A}_r^2}dM_r+\int_0^te^{-\bar{A}_r^2}dV_r\\
	&\leq&V_0+\bigg(\int_0^t(e^{-\bar{A}_r^2}\eta_r)^{1/\theta}d(\theta^{1/\theta}A_r^1)\bigg)^\theta+\int_0^te^{-\bar{A}_r^2}dM_r+\int_0^te^{-\bar{A}_r^2}dV_r.
	\de 
	By noticing $\eta(t)\leq \bar{\eta}(t)$ and $e^{-\bar{A}_t^2}\leq1$  for all $t>0$, we have 
	\ce 
	e^{-\bar{A}_t^2}\eta(t)\leq \bigg[\int_0^t(e^{-\bar{A}_r^2}\eta_r)^{1/\theta}d(\theta^{1/\theta}A_r^1)\bigg]^\theta+\int_0^te^{-\bar{A}_r^2}dM_r+V_t+V_0.
	\de 
    From the proof of Lemma 3.8 in \cite{Ling2022}, we have  for any bounded stopping time $\tau$	
	\ce 
	\mE[\exp(-2^{1/\theta+1}\ln(2) \theta^{1/\theta}A_\tau^1-\bar{A}_\tau^2)\eta(\tau)]\leq \mE V_\tau,
	 ~\mbox{when}~ \theta\leq 1,
	\de 
and
	\ce 
	\mE[\exp(-\ln(2)\theta^{1/\theta} (A_\tau^1)^\theta-\bar{A}_\tau^2)\eta(\tau)]\leq \mE V_\tau,  ~\mbox{when}~ \theta> 1.
	\de 
	By noticing $\bar{A}_t^2\leq A_t^2$, one has 
	\ce 
	\mE[\exp(-2^{1/\theta+1}\ln(2) \theta^{1/\theta}A_\tau^1-A_\tau^2)\eta(\tau)]\leq \mE V_\tau,
	~\mbox{when}~ \theta\leq 1,
	\de 
	and
	\ce 
	\mE[\exp(-\ln(2)\theta^{1/\theta} (A_\tau^1)^\theta-A_\tau^2)\eta(\tau)]\leq \mE V_\tau,  ~\mbox{when}~ \theta> 1,
	\de 
	which yields by H\"{o}lder's inequality that for any $p\in(0,1)$,
	\ce 
	\mE \eta_{\tau}^p\leq \bigg(\mE \exp\bigg[\frac{p}{1-p}(2^{1/\theta+1} \theta^{1/\theta}A_\tau^1+A_\tau^2)\bigg]\bigg)^{1-p}(\mE V_{\tau})^p, ~\mbox{when}~ \theta\leq 1,
	\de 
	and 
	\ce 
	\mE \eta_{\tau}^p\leq \bigg(\mE \exp\bigg[\frac{p}{1-p}(\theta^{1/\theta} (A_\tau^1)^\theta+A_\tau^2)\bigg]\bigg)^{1-p}(\mE V_{\tau})^p, ~\mbox{when}~ \theta> 1.
	\de 
	The proof is complete.
\end{proof}
   \begin{lemma}[Quantitative Khasminskii's lemma]\label{Khasminskii's lemma}
   	Let $0\leq S\leq T$ and let $\xi(t)$ be a nonnegative measurable $(\sF_t)$-adapted process. Assume that for all $S\leq s\leq t\leq T$, 
   	\be \label{Khasminskii_condition_1}
   	\mE_s\bigg[\int_s^t\xi(r)dr\bigg]\leq \beta(s,t),
   	\ee
   	where $(s,t)\mapsto \beta(s,t)$ is a deterministic function on $\Delta([S,T])$ satisfying:
   	$\beta(t_1,t_2)\leq \beta(t_3,t_4)$ if $(t_1,t_2)\subset (t_3,t_4)$ and there exists $\delta_0>0$ such that
   	\be\label{Khasminskii_condition_2}
   	 \sup_{S\leq s<t\leq T,|t-s|\leq \delta_0}\beta(s,t)\leq \lambda,~\lambda\geq 0. 
   	 \ee
   	 Then for any real number $\kappa<\lambda^{-1}$ if $\lambda\neq 0$, $\lambda^{-1}=\infty$ otherwise and any integer $m\geq 1$, 
   	\ce 
   	\|\int_S^T\xi(r)dr\|_{L_m(\Omega|\sF_S)}^m\leq m!\beta(S,T)^m,
   	\de
   	and there exists positive constant $N$ only depending on $\delta_0$ such that
   	\ce 
   	 \mE_S\bigg[\exp\bigg(\kappa\int_S^T\xi(r)dr\bigg)\bigg]\leq \bigg(\frac{1}{1-\lambda\kappa}\bigg)^N.
   	 \de 
   	 We furthermore suppose that there exists $\gamma>0$ and a continuous control $\sW$ on $\Delta([S,T])$ such that $\beta(s,t)\leq \sW(s,t)^\gamma$ for $(s,t)\in \Delta([S,T])$. Then for every $\lambda>0$,
   	 \be\label{Khasminskii_condition_3}
   	 \mE_S\bigg[\exp\bigg(\kappa\int_S^T\xi(r)dr\bigg)\bigg] \leq 2^{1+(2\kappa)^{1/\gamma}\sW(S,T)}.
   	 \ee
   \end{lemma}
\begin{proof}
	The proof is given in Appendix  \ref{appendixA}.
\end{proof}

The advantages of considering the conditional moment norms over the usual moment norms are summarized in the following result, which is implicit in \cite{Le2021Friz}.
\begin{lemma}\label{lemma_sup_1}
	Let $A_t=(A_t)_{t\in [0,1]}$ be a  continuous adapted process and let $0<p, N<\infty$ be some fixed constants. Assume that $A_0=0$ and 
$$\sup_{0\leq s\leq t\leq 1}\|\delta A_{s,t}\|_{L_p(\Omega|\sF_s)}\leq N.$$
Then we have 

(i) There exists a constant $c_p$ such that $\|A_{\tau}\|_{L_p(\Omega)}\leq c_pN$ for any stopping time $\tau\leq 1$.

(ii) For every $\bar{p}\in (0,p)$, there exists a constant $c_{\bar{p},p}$ such that 
$$\|\sup_{t\in[0,1]}|A_t|\|_{L_{\bar{p}}(\Omega)}\leq c_{\bar{p},p}N.$$
\end{lemma}

This following result can be found in \cite[Lemma2.3]{BPM-2021}.

\begin{lemma}\label{sum-Lemma3.5}
	Let $(\sE,\|\cdot\|)$ be a normed vector space, $s_0,\zeta_i,\gamma_i\in[0,1]$, $i=1,\cdots,n$ be fixed numbers. Suppose $R:(0,1]\to \sE$ is a function such that 
	\be 
	\|R_t-R_s\|\leq \sum_{i=1}^{n}C_is^{-\zeta_i}(t-s)^{\gamma_i},~~\forall~s_0\leq s\leq t\leq 1,s\neq 0
	\ee 
	for some constants $C_i\geq 0$,  $i=1,\cdots,n$. Assuming that $\gamma_i-\zeta_i>0$ for each $i$,   it follows that 
	$$\|R_t-R_s\|\leq \sum_{i=1}^{n}C_i(1-2^{\zeta_i-\gamma_i})^{-1}(t-s)^{\gamma_i-\zeta_i},~\forall~s_0\leq s\leq t\leq 1,s\neq 0.$$
\end{lemma}

	The following result is an excerpt from \cite[Lemma3.10]{Ling2022}.
\begin{lemma}\label{integrale-Lemma3.6}
	Let $\epsilon>0$, $s\in D_n$ and $r>s$. Then 
	\be
	\int_s^t(r-\theta_n)^{-1-\epsilon}d\theta&\leq& N_{\epsilon}[\min(r-s,1/n)]^{-\epsilon}+\textbf{1}_{r\notin D_n}(r-r_n)^{-\epsilon},\\
	\int_s^t(r-\theta_n)^{-1}d\theta&\leq&\log(n(r_n-s))+2,\\
	\int_s^t(r-\theta_n)^{-1+\epsilon}d\theta&\leq& N_{\epsilon}(r-s)^{\epsilon},
	\ee 
	where 
	$$
	\theta_n=\frac{k}{n},~\mbox{whenever}~\frac{k}{n}\leq \theta<\frac{k+1}{n},~\mbox{for}~k\geq 0.
	$$
\end{lemma}

\section{Auxiliary Sobolev and Difference Estimates}

In this section, we present two significant analytical results that will be utilized to establish the primary conclusions in the forthcoming sections.

\begin{lemma}\label{lemma3.2}
	\begin{enumerate}
		\item[(i)]	Let $\eta\in (0,2]$, and $y^{(\eta)}:=y\textbf{1}_{\eta\in[1,2]}$. For any $p\in (d/\eta\vee 1,\infty]$, there is a constant $C=C(p,d,\eta)>0$ such that for all $f\in \mH_{\eta,p}(\mR^d)$,
		\begin{equation}\label{formula1+5}
			\| \sup_{y\neq 0}|y|^{-\eta}|f(\cdot+y)-f(\cdot)-y^{(\eta)}\cdot \nabla f(\cdot) \|_p\leq C\| f\|_{\mH_{\eta,p}(\mR^d)},
		\end{equation}
		where we adopt the convention $\nabla f =0$ when $\eta\in (0,1)$.
		\item[(ii)] For any $p\in (d,\infty]$, there is a constant $C=C(p,d)>0$ such that for all $f\in \mH_{1,p}(\mR^d)$,
		\be\label{imparticular-estimate-1}
		\| \sup_{y\neq 0}|y|^{-1}|f(\cdot+y)-f(\cdot)|\|_p\leq C\| f\|_{\mH_{1,p}(\mR^d)}.
		\ee 
		\item[(iii)] For any $\delta,\gamma\in (0,1]$ and $p\in (d/\delta\vee d/\gamma,\infty]$, there is a constant $C=C(p,d,\delta,\gamma)>0$ such that for all $f\in \mH_{\delta+\gamma,p}(\mR^d)$,
		\be \label{analysis-operater-l}
		\|\sup_{z,w\neq 0}|w|^{-\delta}|z|^{-\gamma}|(f(\cdot+z+w)-f(\cdot+z))-(f(\cdot+w)-f(\cdot))|\|_p\leq C\| f\|_{\mH_{\delta+\gamma,p}(\mR^d)}.
		\ee
	\end{enumerate}
\end{lemma}
\begin{proof}
	(i)	The inequality (\ref{formula1+5}) is a direct result of lemma \cite[Lemma 5]{MP-1992}.
	
	(ii)  Since $f\in \mH_{1,p}$, by approximation we may assume that $f\in C_0^{\infty}(\mR^d)$. For any $x,y\in\mR^d$, the following  can be drawn by direct estimation by (\ref{formula1+5}) and get
	\ce 
	&&\|\sup_{y\neq 0}|y|^{-1}|f(x+y)-f(x)|\|_p\\
	&\leq&\|\sup_{y\neq 0}|y|^{-1}|f(x+y)-f(x)-y\cdot \nabla f(x)|\|_p+\|\nabla f(x)\|_p.
	\de 
	
	(iii) By approximation we may assume that $f\in C_0^{\infty}(\mR^d)$. We first consider the case $\delta,\gamma\in (0,1)$. Let $\Lambda_z(x):=f(x+z))-f(x)$. Then we have
	\ce 
	&&(f(x+z+w)-f(x+w))-(f(x+z)-f(x))\\
	&=&\Lambda_z(x+w)-\Lambda_z(x).
	\de 
	From Lemma 2.1 \cite{komatsu1984martingale} for any $x, w\in\mR^d$,  $\delta\in (0,1)$ we obtain 
	\ce 
	\Lambda_z(x+w)-\Lambda_z(x)=N(\delta,d)\int_{\mR^d}(|y+w|^{-d+\delta}-|y|^{-d+\delta})\partial^{\delta}\Lambda_z(x-y)dy.
	\de 
	For any $z\in\mR^d$, again applying Lemma 2.1 \cite{komatsu1984martingale} to get for any $\gamma\in (0,1)$ 
	\ce 
	&&\Lambda_z(x+w)-\Lambda_z(x)\\
	&=&N(\delta,d)\int_{\mR^d}(|y+w|^{-d+\delta}-|y|^{-d+\delta})(\partial^{\delta}f(x-y+z)-\partial^{\delta}f(x-y))dy\\
	&=&N(\delta,d)N(\gamma,d)\int_{\mR^d}(|y+w|^{-d+\delta}-|y|^{-d+\delta})\\
	&&\cdot\int_{\mR^d}(|u+z|^{-d+\gamma}-|u|^{-d+\gamma})(\partial^{\gamma}\partial^{\delta}f(x-y-u)dudy.
	\de 
	We denote 
	\ce 
	\Sigma_z(y):=|\int_{\mR^d}(|u+z|^{-d+\gamma}-|u|^{-d+\gamma})(\partial^{\gamma}\partial^{\delta}f(x-y-u)du|.
	\de 
	Then we have,  for $t>0$, $|z|\leq 1$
	\ce 
	\Sigma_{tz}(y)\leq N(\delta,d,p)(J_1+J_2),
	\de 
	where 
	\ce 
	J_1&:=&|\int_{|u|>2t}(|u+tz|^{-d+\gamma}-|u|^{-d+\gamma})(\partial^{\gamma}\partial^{\delta}f(x-y-u)du|,\\
	J_2&:=&|\int_{|u|\leq2t}(|u+tz|^{-d+\gamma}-|u|^{-d+\gamma})(\partial^{\gamma}\partial^{\delta}f(x-y-u)du|.
	\de 
	Elementary computation shows that 
	\ce 
	||u+tz|^{-d+\gamma}-|u|^{-d+\gamma}|\leq N(\gamma,d)t|u|^{-d+\gamma-1}, \mbox{when}~|u|>2t, |z|\leq 1.
	\de 
	From now on, we can estimate $J_1$ as following
	\ce 
	J_1&\leq& N(\gamma,d)t\int_{|u|>2t} |u|^{-d+\gamma-1}|\partial^{\gamma}\partial^{\delta}f(x-y-u)|du\\
	&\lesssim& t^{\gamma} \int_{|u|>2t} t^{-d}|\frac{u}{t}|^{-d+\gamma-1}|\partial^{\gamma}\partial^{\delta}f(x-y-u)|du\\
	&\lesssim&t^{\gamma}\sup_{\theta>0}\theta^{-d}|\frac{\cdot}{\theta}|^{-d+\gamma-1}\mI_{|\cdot|/\theta\geq 2}*|\partial^{\gamma}\partial^{\delta}f|(x-y)\\
	&\lesssim&t^{\gamma}\int_{|y|\geq2}|y|^{-d+\gamma-1}dy\sM(|\partial^{\gamma}\partial^{\delta}f|)(x-y)\\
	&\lesssim&t^{\gamma}\sM(|\partial^{\gamma}\partial^{\delta}f|)(x-y),
	\de 
	where $\sM f:=\sup_{r>0}\fint_{B_r}|f(x+y)|dy$ with $\fint_{B_r}:=\frac{1}{|B_r|}\int_{B_r}$ denoting the average integral over the ball and $|B_s|$ representing the Lebesgue measure of $B_r$, is the  Hardy–Littlewood maximal function of $f$. The penultimate inequality follows from \cite[Theorem 2.1.10]{Loukas2008}.
	For the term $J_2$,by again invoking \cite[Theorem 2.1.10.]{Loukas2008} and applying H\"{o}der's inequality with $p>d/\gamma\vee d/\delta$, and selecting any $p>p_0>d/\gamma\vee d/\delta$ along with $q_0\geq 1$ such that $\frac{1}{p_0}+\frac{1}{q_0}=1$, we have 
	\ce 
	J_2&\leq &\bigg(\int_{|u|\leq2t}|u+tz|^{(-d+\gamma )q_0}du\bigg)^{\frac{1}{q_0}}\bigg(\int_{|u|\leq2t}|(\partial^{\gamma}\partial^{\delta})f|^{p_0}(x-y-u)du\bigg)^{\frac{1}{p_0}}\\
	&&+\bigg(\int_{|u|\leq2t}|u|^{(-d+\gamma )q}du\bigg)^{\frac{1}{q_0}}\bigg(\int_{|u|\leq2t}|(\partial^{\gamma}\partial^{\delta})f|^{p_0}(x-y-u)du\bigg)^{\frac{1}{p_0}}\\
	&\leq &\bigg(\int_{|u|\leq3t}|u|^{(-d+\gamma )q}du\bigg)^{\frac{1}{q_0}}\bigg(\int_{|u|\leq2t}|(\partial^{\gamma}\partial^{\delta})f|^{p_0}(x-y-u)du\bigg)^{\frac{1}{p_0}}\\
	&\lesssim&t^\gamma \sup_{\theta>0}\bigg(\theta^{-d}\mI_{|\frac{\cdot}{\theta}|\leq 2}*|(\partial^{\gamma}\partial^{\delta})f|^{p_0}(x-y)\bigg) ^{\frac{1}{p_0}}\\
	&\lesssim&t^\gamma \bigg(\sM|(\partial^{\gamma}\partial^{\delta})f|^{p_0}\bigg) ^{\frac{1}{p_0}}(x-y).
	\de 
	These estimates of $J_1,~J_2$ imply that for any $p>p_0>d/\gamma\vee d/\delta$
	\ce 
	|\Sigma_z(y)|\lesssim |z|^{\gamma}(\sM(|\partial^{\gamma}\partial^{\delta}f|)(x-y)+ \bigg(\sM(|\partial^{\gamma}\partial^{\delta}f|^{p_0})\bigg)^{\frac{1}{p_0}}(x-y)),
	\de 
	where the implicit constants only depend on $\alpha,\gamma,p,d$. By completely similarly estimating as $\Sigma_{z}(y)$, we obtain for any $p>p_0>d/\gamma\vee d/\delta$
	\ce 
	&&|\Lambda_z(x+w)-\Lambda_z(x)|\\
	&\lesssim&\int_{\mR^d}||y+w|^{-d+\delta}-|y|^{-d+\delta}||\Sigma_z(y)|dy\\
	&\lesssim& |z|^{\gamma}|w|^{\delta}\bigg[\sM\bigg(\sM(|\partial^{\gamma}\partial^{\delta}f|)\bigg)(x)+
	\bigg(\sM\bigg(\sM(|\partial^{\gamma}\partial^{\delta}f|)\bigg)^{p_0}(x)\bigg)^{\frac{1}{p_0}}\\
	&&+\sM\bigg(\bigg(\sM\big(|\partial^{\gamma}\partial^{\delta}f|^{p_0}\big)\bigg)^{\frac{1}{p_0}}\bigg)(x)+
	\bigg(\sM\big(\sM\big(|\partial^{\gamma}\partial^{\delta}f|^{p_0}\big)\big)\bigg)^{\frac{1}{p_0}}(x)\bigg].
	\de 	
	From \cite[Theorem 1.1.1]{Stein1970}, we get (\ref{analysis-operater-l}) when $\delta,\gamma\in(0,1)$.

	Consider the case $\delta=\gamma=1$, we rewrite the term $\Lambda_z(x+w)-\Lambda_z(x)$  as the follows
	\ce 
	&&\sup_{w,z\neq 0}(|w||z|)^{-1}|\Lambda_z(x+w)-\Lambda_z(x)|\\
	&\leq&\sup_{z\neq 0}(|z|)^{-1}\sup_{t>0,|w|\leq 1}|\Lambda_z(x+tw)-\Lambda_z(x)|\\
	&\leq&\sup_{z\neq 0}(|z|)^{-1}\sup_{t>0}t^{-d}\int_{|w|\leq t}|\nabla \Lambda_z(x+w)|dw\\
	&\leq&\sup_{t>0}t^{-d}\int_{|w|\leq t}\sup_{z\neq 0}(|z|)^{-1}|\nabla \Lambda_z(x+w)|dw,
	\de 
	and similarly we have
	\ce 
	&&\sup_{z\neq 0}(|z|)^{-1}|\nabla \Lambda_z(x+w)|\\
	&=&\sup_{z\neq 0}(|z|)^{-1}|\nabla f(x+w+z)-\nabla f(x+w)|\\
	&=&\sup_{s>0,|z|\leq 1}(|s|)^{-1}|\nabla f(x+w+s z)-\nabla f(x+w)|\\
	&\lesssim& \sup_{s>0}(|s|)^{-d}\int_{|z|\leq s}|\nabla^2 f(x+w+z)|dz\\
	&\leq& \sM|\nabla^2 f|(x+w).
	\de 
	Then we have obtained  
	\ce 
	\sup_{w,z\neq 0}(|w||z|)^{-1}|\Lambda_z(x+w)-\Lambda_z(x)|\lesssim \sM\bigg(\sM(|\nabla^2 f|)\bigg)(x).
	\de 
	From  \cite[Theorem 1.1.1]{Stein1970}, we get (\ref{analysis-operater-l}) when $\delta=1,\gamma=1$.
	
	For the case $\delta=1$, $\gamma<1$, we have for any $x, z\in\mR^d$
	\ce 
	&&\sup_{w,z\neq 0}|z|^{-1}|w|^{-\gamma}|f(x+z+w)-f(x+z)-(f(x+w)-f(x))|\\
	&=&\sup_{w,z\neq 0}|z|^{-1}|w|^{-\gamma}|\Lambda_w(x+z)-\Lambda_w(x)|\\
	&\lesssim&\sup_{w\neq 0}(|w|)^{-\gamma}\sup_{s>0}s^{-d}\int_{|z|\leq s}|\nabla \Lambda_w(x+z)|dz\\
	&\leq&\sup_{s>0}s^{-d}\int_{|z|\leq s}\sup_{w\neq 0}(|w|)^{-\gamma}|\nabla \Lambda_w(x+z)|dz.
	\de 
	From Lemma 2.1 \cite{komatsu1984martingale} for any $x, z\in\mR^d$,  $\gamma\in (0,1)$ we obtain 
	\ce 
	|\nabla \Lambda_w(x+z)|&=&|\nabla f(x+z+w)-\nabla f(x+z)|\\
	&=&N(\gamma,d)\int_{\mR^d}(|y+w|^{-d+\gamma}-|y|^{-d+\gamma})\partial^{\gamma}(\nabla f)(x+z-y)dy.
	\de 
	By completely similarly estimating as $\Sigma_z(y)$, we obtain for any $p>p_0>d/\gamma$
	\ce 
	|\nabla\Lambda_w(x+z)|&\lesssim&|w|^{\gamma}\bigg(\sM(|\partial^{\gamma}\nabla f|)(x+z)+ \bigg(\sM(|\partial^{\gamma}\nabla f|^{p_0})\bigg)^{\frac{1}{p_0}}(x+z)\bigg).
	\de 
	Hence we get
	\ce 
	&&\sup_{w,z\neq 0}|z|^{-1}|w|^{-\gamma}|f(x+z+w)-f(x+z)-(f(x+w)-f(x))|\\
	&\lesssim &\sM\bigg(\sM(|\partial^{\gamma}\nabla f|)(x+z)+ \bigg(\sM(|\partial^{\gamma}\nabla f|^{p_0})\bigg)^{\frac{1}{p_0}}(x+z)\bigg)(x).
	\de 
	From Theorem 1 \cite[1.1.1]{Stein1970}, we derive (\ref{analysis-operater-l}) for the case when  $\delta=1,\gamma\in(0,1)$. For the scenario where  $\delta\in (0,1)$, $\gamma=1$, the proof proceeds analogously to the case where $\delta=1,\gamma\in(0,1)$.Thus, we have completed the proof.
\end{proof}

 To prove Lemma \ref{lemma1+4.1}, we require the following lemma that was introduced in \cite{DM-Le-2021}.
\begin{lemma}\label{lemma1+4.2}
	Let $K>0$ be a constant and let $B,\widetilde{B}$ be symmetric invertible matrices such that $K^{-1}I\leq B\widetilde{B}^{-1}\leq KI$. Then, for all $x\in\mR^d$, we have the following bound
	\be\label{formula+2}
	|(p_B-p_{\widetilde{B}})(x)|\leq N\|I-B\widetilde{B}^{-1}\|(p_{\frac{B}{2}}+p_{\frac{\widetilde{B}}{2}})(x)
	\ee 
	where $p_B(x)=\frac{1}{\sqrt{(2\pi)^d\det{(B)}}}\exp{\bigg(-\frac{1}{2}x^TB^{-1}x\bigg)}$ and $N$ is a constant that depends only on $d,K$. 
\end{lemma}

\begin{lemma}\label{lemma1+4.1}
	Let $\lambda,l,\epsilon>0$, $\alpha,\delta\in (0,1]$ be fixed numbers. Let $a_1$ be a symmetric $d\times d$ matrix, and let $a(x),\widetilde{a}(x)$ be a $d\times d$ matrix-valued functions. Assume that for any $x$, $a(x),\widetilde{a}(x)$ are both symmetric, satisfying
	\ce 
	\lambda^{-1}lI\leq a(x)+a_1,~\widetilde{a}(x)+a_1\leq \lambda lI,
	\de 
	where $I$ is the $d\times d$ unit matrix. Additionally, assume that
	\ce 
	\sup_y\|a(y)-\widetilde{a}(y)\|\leq \epsilon.
	\de 
	Let $h$ be a real function such that 
	\ce 
	|h(x)-h(y)|\leq \lambda|x-y|^\alpha 
	\de 
	for all $x,y\in\mR^d$. Let $\zeta,\xi,\eta$ be independent d-dimensional random vectors. Assume that $\zeta$ has normal distribution $\sN(0,I)$, $\eta$ has normal distribution $\sN(0,\sqrt{a_1})$ and $E[|\xi|^\alpha]< \infty$. Define an operater $T^*$ by the formula for Borel measure function $f$ on $\mR^d$
	\ce 
	T^*f(y)=\mE[f(y+\sqrt{a(y)}\zeta)]
	\de 
	and $T$ be the adjoint of $T^*$ in the $L_2$ sense. Let $1\leq i,j\leq d$ be fixed and define 
	\ce
	H(x)=\mE[h(x)(\partial_{x_i,x_j}^2Tf)(x+\eta+\xi)-(\partial_{x_i,x_j}^2T[hf])(x+\eta+\xi)],
	\de 
	\ce 
	F(x)=H(x)+\mE[\mathscr{L}_{\upsilon,R}^{g}Tf(x+\eta+\xi)].
	\de 
	Define $\widetilde{T}^*$, $\widetilde{T}$ and $\widetilde{F}(x)$ analogously, but with $\widetilde{a}$ replacing $a$. 
	
	(i)  We assume that there exists a constant $\beta\in(0,2]$ such that the function $g$ on $\mR^d$ satisfies 
	\ce 
	\|\varGamma_{0,\beta}^{0,R}g\|_{\infty}<\infty.
	\de 
	Then, for any $p\in (d/\beta\vee 1,\infty]$, $p'\in [p,\infty]$, and bounded Borel $f\in L_p(\mR^d)$,
	\be \label{formula1+1+6}
	\|F\|_{p'}\leq N\|f\|_{p}[(1+l^{-\frac{\alpha}{2}}\mE[|\xi|^\alpha])l^{\frac{\alpha}{2}-1 +\frac{d}{2p'}-\frac{d}{2p}}+l^{-\frac{\beta}{2}+\frac{d}{2p'}-\frac{d}{2p}}],
	\ee 
	where the constant $N$ depends only on $\lambda,\alpha,\beta, p, p'$.
	
	(ii)  We assume that there exists a constant $\beta\in(1,2]$ such that functions $g$ and $\widetilde{g}$ on $\mR^d$ satisfy
	\ce 
	\|\varGamma_{0,\beta}^{0,R}g\|_{\infty}<\infty,~	\|\varGamma_{0,\beta}^{0,R}\widetilde{g}\|_{\infty}<\infty.
	\de 
	Then, for any $p\in (2d/\beta,\infty]$, $p'\in [p,\infty]$, and bounded Borel function $f\in L_p(\mR^d)$,
	\be \label{formula1+1+7}
	\|F(x)-\widetilde{F}(x)\|_{p'}
	\leq N\epsilon\|f\|_{p}\bigg[(1+l^{-\frac{\alpha}{2}}\mE[|\xi|^\alpha])l^{\frac{\alpha}{2}-2 +\frac{d}{2p}-\frac{d}{2p'}}+ (1+\epsilon)l^{-\frac{1}{2}(\beta-\frac{d}{p'}+\frac{d}{p})}
	+l^{-\frac{1}{2}(1+\beta-\frac{d}{p'}+\frac{d}{p})}\bigg],
	\ee 
	where the constant $N$ depends only on $\lambda,\alpha,\beta, p, p'$.
\end{lemma}
\begin{proof}
	(i)	By the definition of $H(x)$ and the independent of $\zeta$ and $\xi$, it is obvious that 
	\ce
	H(x)=\mE[\mE[h(x)(\partial_{x_i,x_j}^2Tf)(x+z+\eta)-(\partial_{x_i,x_j}^2T[hf])(x+z+\eta)]|_{z=\xi}].
	\de 
	Denote 
	\ce 
	K(x,z)&:=&\mE[h(x)(\partial_{x_i,x_j}^2Tf)(x+z+\eta)-(\partial_{x_i,x_j}^2T[hf])(x+z+\eta)].
	\de 
	By direct computations (see also \cite{gyngy1996existence}, p.11)
	\ce 
	K(x,z)&=&\int_{\mR^d}[h(x)-h(y)]f(y)[(A(y)w)^i(A(y)w)^j-A^{i,j}(y)]p_{a(y)+a_1}(w)|_{w=y-x-z}dy,
	\de 
	where $A(y)=(a(y)+a_1)^{-1}$. A similar formula for $\widetilde{K}$ is valid with $\widetilde{A}(y)=(\bar{a}(y)+a_1)^{-1}$. By ellipticity of $a(y)+a_1$ and H\"{o}lder continuity of $h$, we have
	\ce 
	|K(x,z)|&\lesssim&\int_{\mR^d}|f(y)||y-x|^\alpha(l^{-2}|y-x-z|^2+l^{-1})p_{2\lambda l}(y-x-z)dy\\
	&\lesssim&\int_{\mR^d}|f(y)|(|y-x-z|^\alpha+|z|^\alpha)(l^{-2}|y-x-z|^2+l^{-1})p_{2\lambda l}(y-x-z)dy.
	\de 
	Setting $q=\frac{p}{p-1}$ and applying H\"{o}lder inequality, we get 
	\ce 
	|K(x,z)|&\lesssim&\|f\|_{p}\bigg(\int_{\mR^d}|y|^{q\alpha}(l^{-2}|y|^2+l^{-1})^qp_{2\lambda l}(y)^qdy\bigg)^{\frac{1}{q}}\\
	&&+\|f\|_{p}|z|^\alpha\bigg(\int_{\mR^d}(l^{-2}|y|^2+l^{-1})^qp_{2\lambda l}(y)^qdy\bigg)^{\frac{1}{q}}\\
	&\lesssim&\|f\|_{p}(l^{\frac{\alpha}{2}-1-\frac{d}{2p}}+|z|^\alpha l^{-1-\frac{d}{2p}}).
	\de 
	Then we get the following estimate of $H(x)$
	\ce 
	|H(x)|&\lesssim& \|f\|_{p}(l^{\frac{\alpha}{2}-1-\frac{d}{2p}}+\mE[|\xi|^\alpha] l^{-1-\frac{d}{2p}}).
	\de 
	Use Minkowski's inequality  to get 
	\ce 
	&&\|H\|_{p}\\
	&\lesssim& \bigg(\int_{\mR^d}\bigg(\mE\bigg[\int_{\mR^d}|f(l^{\frac{1}{2}}w+x+\xi)|(l^{\frac{\alpha}{2}}|w|^\alpha+|\xi|^\alpha)(l^{-1}|w|^2+l^{-1})p_{1}(w)dw\bigg]\bigg)^pdx\bigg)^{\frac{1}{p}}\\
	&=&\bigg[\int_{\mR^d}\bigg(\int_{\mR^d}\mE[|f(l^{\frac{1}{2}}w+x+\xi)|(l^{\frac{\alpha}{2}}|w|^\alpha+|\xi|^\alpha)(l^{-1}|w|^2+l^{-1})p_{1}(w)dw]\bigg)^pdx\bigg]^{\frac{1}{p}}\\
	&\leq&\int_{\mR^d}\mE\bigg[\bigg(\int_{\mR^d}|f(l^{\frac{1}{2}}w+x+\xi)(l^{\frac{\alpha}{2}}|w|^\alpha+|\xi|^\alpha)(l^{-1}|w|^2+l^{-1})p_{1}(w)|^pdx\bigg)^{\frac{1}{p}}\bigg]dw\\
	&=&\int_{\mR^d}\mE\bigg[\bigg(\int_{\mR^d}|f(l^{\frac{1}{2}}w+x+\xi)|^pdx\bigg)^{\frac{1}{p}}|(l^{\frac{\alpha}{2}}|w|^\alpha+|\xi|^\alpha)(l^{-1}|w|^2+l^{-1})p_{1}(w)|\bigg]dw\\
	&\leq&\|f\|_{p}\int_{\mR^d}\mE[|(l^{\frac{\alpha}{2}}|w|^\alpha+|\xi|^\alpha)(l^{-1}|w|^2+l^{-1})p_{1}(w)|]dw\\
	&\lesssim&\|f\|_{p}(l^{\frac{\alpha}{2}}+\mE[|\xi|^\alpha])l^{-1}.
	\de 
	From the aforementioned estimates of $H(x)$ and by applying the H\"{o}lder interpolation inequality, we obtain 
	\be\label{formula2+1}
	\|H\|_{p'}&\leq& \|H\|_{p}^{\frac{p}{p'}}\|H\|_{\infty}^{1-\frac{p}{p'}}\no\\
	&\leq&\|f\|_{p}(1+l^{-\frac{\alpha}{2}}E[|\xi|^\alpha])l^{\frac{\alpha}{2}-1 +\frac{d}{2p'}-\frac{d}{2p}}.
	\ee 
	For the term $\mE[\mathscr{L}_{\upsilon,R}^{g}Tf(x+\eta+\xi)]$, we  have 
	\ce 
	\mE[\mathscr{L}_{\upsilon,R}^{g}Tf(x+\eta+\xi)]=\mE[\mathscr{L}_{\upsilon,R}^{g}Tf(x+w)]|_{w=\eta+\xi}].
	\de 
	Recalling the definitions of $\mathscr{L}_{\upsilon,R}^{g}$ and $\varGamma_{0,\beta}^{0,R}$, for  $\beta\in(d/p\vee1,2]$, we get
	\ce 
	&&|\mathscr{L}_{\upsilon,R}^{g}Tf(x+w)|\\
	&\leq &\sup_{y\neq 0}|y|^{-\beta}|Tf(x+w+y)-Tf(x+w)-\nabla (Tf)(x+w)\cdot y||\varGamma_{0,\beta}^{0,R}g(x)|.
	\de 
	Thus, thanks to $p'\geq p>d/\beta\vee 1$, by (\ref{formula1+5}) to obtain 
	\ce 
	\|\mathscr{L}_{\upsilon,R}^{g}Tf(\cdot+w)\|_{p'}\lesssim\|\varGamma_{0,\beta}^{0,R}g\|_{\infty}\|Tf\|_{\mH_{\beta,p'}(\mR^d)}.
	\de 
	Then, it is obvious that 
	\be\label{formula1+1+3} 
	\|\mE[\mathscr{L}_{\upsilon,R}^{g}Tf(t,x+\eta+\xi)]\|_{p'}\lesssim \|\varGamma_{0,\beta}^{0,R}\|_{\infty}\|Tf\|_{\mH_{\beta,p'}(\mR^d)}.
	\ee 
	The next step in our analysis involves utilizing the interpolation theorem to estimate the term denoted as $\|Tf\|_{\mH_{\beta,p'}(\mR^d)}$. For $j=1,2$, by elementary calculus, we have 
	\ce 
	|\partial^jTf(x)|&\lesssim& \int_{\mR^d}|f(y)|l^{-j}|y-x|^jp_{2\lambda l}(y-x)dy.
	\de 
	Let $p'\in [p,\infty]$ and $\frac{1}{r}=1+\frac{1}{p'}-\frac{1}{p}$. By Young's inequality, we get 
	\ce 
	\|\partial^jTf\|_{p'}&\lesssim& \|f\|_{p}\bigg(\int_{\mR^d}l^{-rj}|y|^{rj}p_{2\lambda l}(y)^rdy\bigg)^{\frac{1}{r}}\\
	&\lesssim&\|f\|_{p}l^{-\frac{1}{2}(j-\frac{d}{p'}+\frac{d}{p})},
	\de 
	which implies by the interpolation theorem that for any  $\widehat{\beta}\in(1,2]$ and $p'\in[p,\infty]$,
	\be\label{formula1+3}
	\|Tf\|_{\mH_{\widehat{\beta},p'}(\mR^d)}&\lesssim&\|f\|_{p}l^{-\frac{\widehat{\beta}}{2}+\frac{d}{2p'}-\frac{d}{2p}}.
	\ee 
	From (\ref{formula2+1}) and (\ref{formula1+3}), we obtain, for each  $p>d/\beta\vee 1$ and $p'\in[p,\infty]$
	\ce 
	\|F\|_{p'}\lesssim \|f\|_{p}[(1+l^{-\frac{\alpha}{2}}\mE[|\xi|^\alpha])l^{\frac{\alpha}{2}-1 +\frac{d}{2p}-\frac{d}{2p'}}+l^{-\frac{\beta}{2}+\frac{d}{2p'}-\frac{d}{2p}}].
	\de 
	
	(ii) Next, we proceed to prove (\ref{formula1+1+7}). Given two invertible matrices $D$ and $E$, we have the inequality 
	\ce 
	\|D^{-1}-E^{-1}\|\leq \|D^{-1}\|\|E^{-1}\|\|E-D\|,
	\de 
	which follows from the identity $D^{-1}-E^{-1}=D^{-1}(E-D)E^{-1}$. It is also evident that $\|A(y)\|$, $\|\widetilde{A}(y)$ are uniformly bounded by $l^{-1}$ for all $y$. Consequently,
	\ce 
	\|A(y)-\widetilde{A}(y)\|\leq \|A(y)\|\|\widetilde{A}(y)\|\|a(y)-\widetilde{a}(y)\|\lesssim l^{-2}\epsilon,
	\de  
and similarly,
	\ce 
	\|(A(y)z)^i(A(y)z)^j-(\widetilde{A}(y)z)^i(\widetilde{A}(y)z)^j\|\lesssim |z|^2l^{-3}\epsilon.
	\de 
	Under our assumptions, it is straightforward to verify the existence of a constant $K>0$ such that 
	\ce 
	K^{-1}I\leq (a(y)+a_1)(\widetilde{a}(y)+\widetilde{a}_1)^{-1}\leq KI
	\de 
	and 
	\ce 
	\|I- (a(y)+a_1)(\widetilde{a}(y)+\widetilde{a}_1)^{y}\|\leq Kl^{-1}\epsilon.
	\de 
	Hence, from (\ref{formula+2}) and the fact that $\lambda^{-1}lI\leq a(y)+a_1,\widetilde{a}(y)+a_1\leq \lambda lI$, we have 
	\be\label{formula1+4}
	|p_{a(y)+a_1}(z)-p_{\widetilde{a}(y)+a_1}(z)|&\lesssim& l^{-1}\epsilon[(p_{(\widetilde{a}(y)+a_1)/2})+p_{(a(y)+a_1)/2}]\no\\
	&\lesssim&l^{-1}\epsilon p_{l\lambda I}(z).
	\ee 
	It follows that 
	\ce 
	|H(x)-\widetilde{H}(x)|\lesssim \epsilon\int_{\mR^d}|f(y)||y-x|^\alpha(|w|^2l^{-3}+l^{-2})p_{\lambda l I}(w)|_{w=y-x-z}dy.
	\de 
	Then we apply the H\"{o}lder inequality and Minkowshi inequality as previously to obtain 
	\ce 
	|H(x)-\widetilde{H}(x)|\leq N\epsilon \|f\|_{p}(1+l^{-\frac{\alpha}{2}}E[|\xi|^\alpha] )l^{\frac{\alpha}{2}-2-\frac{d}{2p}}
	\de 
	and 
	\ce 
	\|H(x)-\widetilde{H}(x)\|_{p}\leq N\epsilon \|f\|_{p}(1+l^{-\frac{\alpha}{2}}E[|\xi|^\alpha])l^{\frac{\alpha}{2}-2}.
	\de 
	Using the above estimates of $H(x)-\widetilde{H}(x)$ and the H\"{o}lder interpolation inequality, we have
	\be\label{formula2+2}
	\|H(x)-\widetilde{H}(x)\|_{p'}&\leq& \epsilon\|H(x)-\widetilde{H}(x)\|_{p}^{\frac{p}{p'}}\|H(x)-\widetilde{H}(x)\|_{\infty}^{1-\frac{p}{p'}}\no\\
	&\leq&\epsilon\|f\|_{p}(1+l^{-\frac{\alpha}{2}}E[|\xi|^\alpha])l^{\frac{\alpha}{2}-2 +\frac{d}{2p}-\frac{d}{2p'}}.
	\ee 
	For the term $E[\mathscr{L}_{\upsilon,R}^{g}Tf(x+\eta+\xi)-\mathscr{L}_{\upsilon,R}^{\widetilde{g}}\widetilde{T}f(x+\eta+\xi)]$, we write
	\ce 
	&&|\mE[\mathscr{L}_{\upsilon,R}^{g}Tf(x+\eta+\xi)-\mathscr{L}_{\upsilon,R}^{\widetilde{g}}\widetilde{T}f(x+\eta+\xi)]|\\
	&\leq&\mE[|\mathscr{L}_{\upsilon,R}^{g}Tf(x+\eta+\xi)-\mathscr{L}_{\upsilon,R}^{g}\widetilde{T}f(x+\eta+\xi)|]\\
	&&+\mE[|\mathscr{L}_{\upsilon,R}^{g}\widetilde{T}f(x+\eta+\xi)-\mathscr{L}_{\upsilon,R}^{\widetilde{g}}\widetilde{T}f(x+\eta+\xi)|]\\
	&=:&I_1(x)+I_2(x).
	\de 
	For the cheap part of $I_1$, we have 
	\ce 
	I_1(x)&=&\mE[|\mathscr{L}_{\upsilon,R}^{g}Tf(x+\eta+\xi)-\mathscr{L}_{\upsilon,R}^{g}\widetilde{T}f(x+\eta+\xi)|]\\
	&=&\mE[\mathscr{L}_{\upsilon,R}^{g}(T-\widetilde{T})f(x+\eta+\xi)].
	\de 
	We apply (\ref{formula1+4}) as in the proof of (\ref{formula1+1+3}) to obtain, for any $p\in(d/\beta\vee 1,\infty]$ and $p'\in[p,\infty]$
	\ce 
	\|I_1\|_{p'}&\lesssim&\|Tf-\widetilde{T}f\|_{\mH_{\beta,p'}(\mR^d)}\\
	&\lesssim&\epsilon\|f\|_{p}l^{-\frac{1}{2}(1+\beta-\frac{d}{p'}+\frac{d}{p})}.
	\de
	For the term $I_2$, we replace $\widetilde{T}$ with $T$ for simplicity and rewrite it as follows for any $\beta\in (1,2]$
	\ce 
	I_2&=&\mE[|\mathscr{L}_{\upsilon,R}^{g}Tf(x+\eta+\xi)-\mathscr{L}_{\upsilon,R}^{\widetilde{g}}Tf(x+\eta+\xi)|]\\
	&=&\mE\bigg[|\int_{|z|< R}Tf(x+\xi+\eta+g(x,z))-Tf(x+\xi+\eta+\widetilde{g}(x,z))\\
	&&-(g(x,z)-\widetilde{g}(x,z))\cdot\nabla Tf(x+\xi+\eta)dz|\bigg]\\
	&\leq&B_1(x)+B_2(x),
	\de 
	where 
	\ce 
	B_1(x)&=&\mE\bigg[\int_{|z|< R}|Tf(x+\xi+\eta+\widetilde{g}(x,z)+[g(x,z)-\widetilde{g}(x,z)])\\
	&&-Tf(x+\xi+\eta+[g(x,z)-\widetilde{g}(x,z)])\\
	&&-Tf(x+\xi+\eta+\widetilde{g}(x,z))+Tf(x+\xi+\eta)|dz\bigg],\\
	B_2(x)&=&\mE\bigg[\int_{|z|< R}|Tf(x+\xi+\eta+[g(x,z)-\widetilde{g}(x,z)])
	-Tf(x+\xi+\eta)\\
	&&-(g(x,z)-\widetilde{g}(x,z))\cdot\nabla Tf(x+\xi+\eta)|dz\bigg].
	\de 
	Let us first deal with  $B_1(x)$. For  $\delta=\beta/2$, we have 
	\ce 
	&&|B_1(x)|\\
	&\leq &\sup_x\int_{|z|< R}|\widetilde{g}(x,z)|^{\delta}|g(x,z)-\widetilde{g}(x,z)|^{\delta}dz\\
	&&\cdot \mE[\sup_{z,w\neq 0}|z|^{-\delta}| w|^{-\delta}|(Tf(x+\xi+\eta+z+w)-Tf(x+\xi+\eta+w))\\
	&&-(Tf(x+\xi+\eta+z)-Tf(x+\xi+\eta))|].
	\de 
	Apply (\ref{analysis-operater-l}), (\ref{formula1+3}) and the Cauchy-Schwarz inequality, for any  $p\in(2d/\beta,\infty]$ and $p'\in[p,\infty]$, we get 
	\ce 
	\|B_{1}\|_{p'}
	&\leq &\sup_x\int_{|z|< R}|\widetilde{g}(x,z)|^{\delta}|g(x,z)-\widetilde{g}_t(x,z)|^\delta dz
	\|Tf\|_{\mH_{2\delta,p'}(\mR^d)}\\
	&\leq &\sup_x[(\varGamma_{0,2\delta}^{0,R} |\widetilde{g}(x,z)|)^{\frac{1}{2}}(\varGamma_{0,2\delta}^{0,R} |g(x,z)-\widetilde{g}(x,z)|)^{\frac{1}{2}}]
	l^{-\frac{1}{2}(2\delta-\frac{d}{p'}+\frac{d}{p})}\|f\|_{p}\\
	&\lesssim &\epsilon l^{-\frac{1}{2}(2\delta-\frac{d}{p'}+\frac{d}{p})}\|f\|_{p},
	\de 
	the last inquality we used the assumptions about $g$ and $\widetilde{g}$.  
	For the term $B_2(x)$, it is  evident that for $\beta\in (1,2]$, we have  
	\ce 
	&&|B_2(x)|_{p}\\
	&\leq &\|\varGamma_{0,\beta}^{0,R}\|_{\infty}|g_t(x,z)-\widetilde{g}_t(x,z)| \mE[\|\sup_w|w|^{-\beta}Tf(x+\xi+\eta+w)
	-Tf(x+\xi+\eta)\\
	&&-w\cdot\nabla T_{r,t}f(x+\xi+\eta)|].
	\de 
	Applying (\ref{formula1+5}) and (\ref{formula1+3}), we deduce that for  $p\in(d/\beta\vee1,\infty]$ and $p'\in[p,\infty]$, 
	\ce 
	\|B_2\|_{p'}
	&\leq &\|\varGamma_{0,\beta}^{\varepsilon,R}\|_{\infty} |g_t(x,z)-\widetilde{g}_t(x,z)| \|Tf\|_{\mH_{\beta,p'}(\mR^d)}\\
	&\lesssim & \epsilon^{2} l^{-\frac{1}{2}(\beta-\frac{d}{p'}+\frac{d}{p})}\|f\|_{p}.
	\de 
	
	Combining these estimates together and setting  $2\delta=\beta$,  we obtain for any $p\geq (2d/\beta,\infty]$ and $p'\in[p,\infty]$
	\be \label{formula1+1+4}
	&&\|\mE[\mathscr{L}_{\upsilon,R}^{g}Tf(x+\eta+\xi)-\mathscr{L}_{\upsilon,R}^{\widetilde{g}}\widetilde{T}f(x+\eta+\xi)]\|_{p'}\no\\
	&\leq& N \epsilon\|f\|_{L_p} [(1+\epsilon)l^{-\frac{1}{2}(\beta-\frac{d}{p'}+\frac{d}{p})}
	+l^{-\frac{1}{2}(1+\beta-\frac{d}{p'}+\frac{d}{p})}],
	\ee 
	where the constant $N$ depends only on $\lambda, \alpha, \beta, p, p', \epsilon$. Now, from (\ref{formula2+2}) and (\ref{formula1+1+4}), we can derive (\ref{formula1+1+7}), thereby completing the proof.
\end{proof}

\section{Discrete Krylov estimates for the Euler scheme}

In this section, we establish the discrete Krylov estimates needed for the Euler--Maruyama approximation \eqref{EulerAPP3.81}.
These estimates control functionals of the form
\ce 
\int_s^t f(r,X_r)\,dr,
\qquad
\int_s^t f(r,X_r^n)\,dr,
\de
where $f$ is typically a measurable function in $\mathbb L_p^q([0,1])$.
In the main application, $f$ will be either the approximating drift $b^n$, or a nonlocal quantity generated by the jump coefficient. Since $b$ is only assumed to satisfy the strict Ladyzhenskaya--Prodi--Serrin condition, the standard pointwise arguments for Euler schemes are not available.
The estimates below provide the substitute: they allow us to integrate irregular functions along the discretized trajectory $X^n$ with constants independent of the mesh size.

The estimates in this section are used at three specific points in the proof of the main convergence theorem.
First, the discrete Krylov estimate is used to control the frozen drift error  in the Zvonkin transformed error identity. This term contains the difference
\ce 
b^n(r,X_r^n)-b^n(r,X_{r_n}^n),
\de 
and therefore requires estimates for singular functions evaluated along the continuous Euler trajectory and its frozen version.
Theorem~\ref{proposition0+4.13}, together with the auxiliary transition estimates in Theorem~\ref{theorem4.5} and Proposition~\ref{proposition4.90}, provides the required control.

Second, the increment estimates for the Euler scheme are used to bound the compensated jump martingale freezing error.
This error arises from
\ce 
g(r,X_r^n,z)-g(r,X_{r_n}^n,z)
\de 
inside a compensated Poisson integral.
Its estimate depends on moment bounds for
\ce 
X_r^n-X_{r_n}^n,
\de 
together with the spatial Sobolev regularity of the jump coefficient.
In particular, the case $m=\bar p$ in the Euler increment estimate will be used to control  in Proposition~\ref{prop:jump-specific-estimates}.

Third, the same increment estimates enter the bound for the nonlocal compensator freezing error.
After the Zvonkin transform, the jump compensator contains the nonlocal remainder
\ce 
u(r,x+g(r,x,z))-u(r,x)-g(r,x,z)\cdot\nabla u(r,x).
\de 
Freezing the Euler scheme produces a difference between the values at $X_r^n$ and $X_{r_n}^n$.
	
	 The next we will use the notations $p_t(x)$  and $p_{s,t}f(x)$ defined as 
	$$p_t(x):=(2\pi t)^{-d/2}e^{-|x|^2/(2t)}$$ and $$p_{s,t}f(x):=p_{t-s}*f(x).$$
	
The rest of the current section is devoted for the proof of Theorem \ref{proposition0+4.13}. First we derive some analytic estimates on the transition operators associated the discrete Euler-Maruyama scheme without drift. By means of the stochastic Davie-Gronwall lemma as Lemma \ref{DGlemma1} and Girsanov's theorem, theose analytic estimates are utilized to get the desired moment bound.

	\subsection{Some analytic estimates} For each $s\in D_n$ and $x\in\mR^d$, let $\widetilde{X}_t^n(s,x)$ be the solution to the Euler-Maruyama scheme given by
	\be\label{EulerAPP3.8}
	\widetilde{X}_t^{n}(s,x)&=&x+\int_{s}^{t}\sigma(r,\widetilde{X}_{r_n}^{n}(s,x))dB_r+\int_{s}^{t}\int_{|z|<R}g(r,\widetilde{X}_{{r_n-}}^{n}(s,x),z)\widetilde{N}(dr\,dz),~t\geq s.
	\ee
	For each $t\geq s$ and bounded measurable function $f$, we define the function $Q_{s,t}^nf$ by 
	\be
	Q_{s,t}^nf(x)=\mE[f(\widetilde{X}_t^{n}(s,x))].
	\ee
	Let 
	\ce 
	T_{s,t}^*f(y)=\mE\left[f(y+\int_{s}^{t}\sigma(r,y)dB_r)\right],
	\de 
	and let operator $T_{s,t}$ be the conjugate of $T_{s,t}^*$ in the $L_2-$sence. This operator can be computed explicitly as
	\ce 
	T_{s,y}f(x)=\int_{\mathbb{R}^d}f(y)P_{A_{s,t}(y)}(y-x)dy,~~s<t,
	\de 
	and $T_{s,s}f(x)=f(x)$, where $A_{s,t}=\frac{1}{2}\int_s^ta_rdr$, $a_r=\sigma_r\sigma_r^T$. Whenever $s<t$, the function $T_{s,t}f(x)$ infinitely differentiable and satisfies 
	\ce 
	\partial_s T_{s,t}f(x)=-\partial_{x_i,x_j}^2 T_{s,t}[a_s^{i,j}f](x).
	\de
	We also define for every $r<t$,
	\ce 
	\eta_r(x)=\int_{r_n}^{r}\sigma(u,x)dB_u+\int_{r_n}^{r}\int_{|w|< R}g(u,x,w)\widetilde{N}(dudw),
	\de 
	\ce 
	(H_{r,t}^nf)(x)=\mE[a_{\tau}^{i,j}(x)(\partial_{x^i,x^j}^2T_{\tau,t}f)(x+\eta_r(x))-(\partial_{x^i,x^j}^2T_{\tau,t}[a_{\tau}^{i,j}f])(x+\eta_r(x))],
	\de 
	\ce 
	(G_{r,t}^nf)(x)=\mE[(\mathscr{L}_{\upsilon,R}^{g}[T_{r,t}f])(r,x+\eta_{r}(x))]
	\de 
	and 
	\ce 
	(F_{r,t}^nf)(x)=(H_{r,t}^nf)(x)+(G_{r,t}^nf)(x).
	\de 
	Although the function $\eta$ also depends on $n$, we omit this dependence in the notation for simplicity. It is convenient to first obtain analytic estimates for $H^n$ and $T$.
	
	\begin{lemma}\label{lemma4.3+}
		Let Condition $\sH_\sigma^g$ holds.
		
			(i) For any $f$ in $L_{p}(\mR^d)$ with $1\leq p\leq p'\leq \infty$, the following inequality holds for all $r<t\leq 1$
		\be\label{lemma4.3+1} 
		\|T_{r,t}f\|_{p'}\leq N(t-r)^{\frac{d}{2p'}-\frac{d}{2p}}\|f\|_{p},
		\ee 
		where $N$ depends only on $d, p, p',c_0, C_1$.
		
	(ii) For any $f$ in $L_{p}(\mR^d)$, where $p,p'\in (d/\beta\vee 1,\infty]$and $p\leq p'$, the following estimate holds for all $0\leq r<t\leq 1$
		\be\label{lemma4.3+2} 
		\|{F_{r,t}^nf}\|_{p'}\leq N\|f\|_{p}[(t-r_n)^{\frac{\alpha}{2}-1 +\frac{d}{2p}-\frac{d}{2p'}}+(t-r_n)^{-\frac{\beta}{2}+\frac{d}{2p'}-\frac{d}{2p}}],
		\ee
		where $N$ is a constant that depends exclusively on the dimensions  $d$, the exponents $p, p'$, and the constants $c_0, C_1$ associated with Condition $\sH_\sigma^g$.
	\end{lemma}
	\begin{proof}
	(i)	By the property of uniform ellipticity, there exists a constant  $C>0$ such that for every $x,y\in\mR^d$,  and any time points $r,t$ satisfying $0\leq r\leq t$, the following estimate holds
		\be\label{estimate_a1} 
		C^{-1}(t-r_n)\leq \int_{r}^{t}a_u(y)du+\int_{r_n}^{r}a_u(x)du\leq C(t-r_n),
		\ee 
	where $a_r=(\sigma\sigma^T)_r$ represents the matrix-valued function corresponding to the diffusion coefficient at time $r$.	From this estimate (\ref{estimate_a1}), we can derive (\ref{lemma4.3+1}) from using Gaussian estimates. Specifically, the integrals of a $a_u(y)$ 
 and $a_u(x)$ over the intervals $[r,t]$ and $[r_n,r]$ respectively, can be related to the variances of certain Gaussian processes. The uniform bounds provided by (\ref{estimate_a1}) then allow us to bound the norms of these Gaussian processes, leading to the desired inequality (\ref{lemma4.3+1}).
	
	(ii) Denote $\xi$ by $\int_{r_n}^{t}\int_{|z|\leq R}g(s,y,u)\widetilde{N}(dsdu)$. By applying the Cauchy-Schwarz inequality and the Burkholder-Davis-Gundy inequality, we can derive the following estimate
	\ce 
	E[|\xi|^\alpha]
	&\leq& \bigg(\int_{r_n}^{t}\int_{|z|\leq R}|g(s,y,z)|^2\upsilon(dz)ds\bigg)^{\frac{\alpha}{2}}\\
	&\leq&\|\varGamma_{0,2}^{0,R}|g|\|_{\mL^{\infty}([0,1])}^{\frac{\alpha}{2}}(r-r_n)^{\frac{\alpha}{2}}\\
	&\lesssim&(t-r_n)^{\frac{\alpha}{2}}.
	\de 
	
	 Next, consider the operator $(F_{r,t}^nf)(x)$. By utilizing the estimate for $\xi$ , along with the results from  (\ref{formula1+1+6})and (\ref{estimate_a1}) , we can conclude the following bound on the norm of $(F_{r,t}^nf)(x)$ as follows
		\ce 
		\|(F_{r,t}^nf)\|_{p'}&\leq& \|f\|_{p}[(1+(t-r_n)^{-\frac{\alpha}{2}}E[|\xi|^\alpha])(t-r_n)^{\frac{\alpha}{2}-1 +\frac{d}{2p'}-\frac{d}{2p}}+(t-r_n)^{-\frac{\beta}{2}+\frac{d}{2p'}-\frac{d}{2p}}]\\
		&\lesssim& \|f\|_{p}[(t-r_n)^{\frac{\alpha}{2}-1 +\frac{d}{2p'}-\frac{d}{2p}}+(t-r_n)^{-\frac{\beta}{2}+\frac{d}{2p'}-\frac{d}{2p}}].
		\de 
	
	\end{proof}

	\begin{lemma}\label{lemma4.4+}
		Let $s\in D_n$ and $f$ be a bounded uniformly continuous function. Then, for every $t>s$ and $x\in \mathbb{R}^d$, we have
		\be \label{formul1+1}
		Q_{s,t}^nf(x)=T_{s,t}f(x)+\int_{s}^{t}	 Q_{s,r_n}^n(F_{r,t}^nf)(x)]dr.
		\ee 
	\end{lemma}	
	\begin{proof}
		Let $\widetilde{X}_t^n=\widetilde{X}_t^n(s,x)$ and $r\in (s,t)$. By applying It\^{o} formula to the function $r\mapsto T_{r,t}f(\widetilde{X}_{\tau}^n)$, for any $t>s$, we get 
		\ce 
		&&\mE[T_{r,t}f(\widetilde{X}_t^n)]\\
		&=&\mE[T_{s,t}f(\widetilde{X}_r^n)]+\mE\bigg[\int_{s}^{r}a_\tau^{i,j}(\widetilde{X}_{\tau_n}^n)(\partial_{x_i,x_j}T_{\tau,t}f)(\widetilde{X}_\tau^n)-\partial_{x_i,x_j} T_{\tau,t}(a_\tau^{i,j}f)(\widetilde{X}_\tau^n)d\tau\\
		&&+\int_{s}^{r}\int_{|z|\leq R}\bigg(T_{\tau,t}f(\widetilde{X}_\tau^n+g(\tau,\widetilde{X}_\tau^n,z))-T_{\tau,t}f(\widetilde{X}_\tau^n)-(\nabla T_{\tau,t}f)(\widetilde{X}_\tau^n)\cdot g(\tau,\widetilde{X}_\tau^n,z) \upsilon(dz)\bigg)d\tau\bigg]\\
		&=&\mE[T_{s,t}f(\widetilde{X}_r^n)]+\mE\bigg[\int_{s}^{r}a_\tau^{i,j}\bigg((\widetilde{X}_{\tau_n}^n)(\partial_{x_i,x_j}T_{r,t}f)(\widetilde{X}_\tau^n)-\partial_{x_i,x_j} T_{\tau,t}(a_\tau^{i,j}f)(\widetilde{X}_\tau^n)\\
		&&+(\mathscr{L}_{\upsilon,R}^{g}[T_{\tau,t}f])(\widetilde{X}_{\tau})\bigg)d\tau\bigg].
		\de 
		Writing $\widetilde{X}_{\tau}^n=\widetilde{X}_{\tau_n}^n+\eta_{\tau}(\widetilde{X}_{\tau_n}^n)$, we thake conditional expectation given $\mathscr{F}_{\tau_n}\supset\mathscr{F}_{s} $. This yields 
		\be\label{formula+1} 
		\mE[T_{r,t}f(\widetilde{X}_r^n)]=T_{s,t}f(\widetilde{X}_s^n)+\int_{s}^{r}\mE[(F_{\tau,t}^nf)(\widetilde{X}_{\tau_n}^n)]d\tau.
		\ee 
		We now take the limit $r\to s$ in the above formula. By uniform continuity of $f$ and right continuous of $\widetilde{X}^n$ in probability, $\lim_{r\downarrow s}\mE[T_{r,t}f(\widetilde{X}_r^n)]=\mE[T_{s,t}f(\widetilde{X}_s^n)]$. From (\ref{lemma4.3+2}), we have
		\ce 
		&&\int_{s}^{r}|\mE[(F_{\tau,t}^nf)(\widetilde{X}_{\tau_n}^n)]|d\tau\\
		&\lesssim&\int_{s}^{r}\|f\|_{L_{\infty}}[(\tau-\tau_n)^{\frac{\alpha}{2}-1}+(\tau-\tau_n)^{-\frac{\beta}{2}}]d\tau\\
		&\lesssim&\|f\|_{L_{\infty}}[(r-s)^{\frac{\alpha}{2}}+(r-s)^{1-\frac{\beta}{2}}],
		\de  
		which allows us to apply the limit $r\uparrow t$ to the last term in (\ref{formula+1}). Hence, we have  
		\ce 
		\mE[f(\widetilde{X}_t^n)]=T_{s,t}f(\widetilde{X}_s^n)+\int_{s}^{t}\mE[(F_{\tau,t}^nf)(\widetilde{X}_{r_n}^n)]d\tau,
		\de 
		which deduces to (\ref{formul1+1}).
	\end{proof}
	\begin{theorem}\label{theorem4.5}
		Assume that Condtion $(\sH_\sigma^g)$ holds. Let $p,p'\in(d/\beta\vee 1,\infty]$, with $p\leq p'$, $p<\infty$. Let $f\in L_{p}(\mR^d)$. There exists a constant $N(d,p,p',\alpha,c_0,C_1)$ such that for every $s\in D_n$ and $t\in(s,1]$, we have the following critical estimate
		\be\label{cricial_estimate_1}
		\|Q_{s,t}^nf\|_{p'}\leq N(t-s)^{\frac{d}{2p'}-\frac{d}{2p}}\|f\|_{p}.
		\ee 	
	\end{theorem}
	\begin{proof}
	The proof is given in Appendix \ref{appendixB}.
	\end{proof}

	\subsection{Krylov’s estimate about Euler-Maruyama scheme} 
	We consider the Euler-Maruyama scheme given by
	\be\label{formula1+12+1}
	\widetilde{X}_t^{n}(x_0)&=&x_0+\int_{0}^{t}\sigma(r,\widetilde{X}_{r_n}^{n}(x))dB_r+\int_{0}^{t}\int_{|z|\leq R}g(r,\widetilde{X}_{{r_n-}}^{n}(x),z)\widetilde{N}(dr\,dz),
	\ee 
	where $x_0$ is a $\sF_0$-random variable. By markov property, for every $s\in D_n$ and  any bounded measurable $f$, we have 
	\ce 
	\mE[f(\widetilde{X}_t^{n})|\sF_s]=Q_{s,t}^nf(\widetilde{X}_s^{n}).
	\de 
	\begin{proposition}\label{proposition4.90}
		Let $\widetilde{X}$ satisfy (\ref{formula1+12+1}).
		Suppose  $h$ be a measurable function such that  $h\in L_{\bar{p}_0\rho}(\mR^d)$ for some $\rho\in (0,\infty]$, $\bar{p}_0>0$ and $\rho\bar{p}_0\in (d/\beta\vee 1, \infty]$. Then, for every $r,u\in [0,1]$ with $r-u\geq 2/n$, we have
			\be\label{formula1+130} 
			\|h(\widetilde{X}_r^{n})\|_{L_{\rho}(\Omega|\sF_u)}\leq N(r-u)^{-\frac{d }{2\rho \bar{p}_0}}\|h\|_{\bar{p}_0\rho}.
			\ee 
	\end{proposition}
	\begin{proof}
			Put $\widetilde{u}=u_n+1/n$. In the case when $\rho<\infty$, using Theorem \ref{theorem4.5} with the $p'=\infty$, we have 
		\ce 
		\mE[|h(\widetilde{X}_r^{n})|^{\rho}|{\sF_{\widetilde{u}}}]=Q_{\widetilde{u},r}^n[|h|^{\rho}](\widetilde{X}_{\widetilde{u}}^{n})\lesssim (r-\widetilde{u})^{-\frac{d}{2\bar{p}_0}}\||h|^\rho\|_{\bar{p}_0}.
		\de 
		Noting that $r-\widetilde{u}\geq (r-u)/2$,  and using the tower property of conditional expectations, we get
		\ce 
		\mE[|h(\widetilde{X}_r^{n})|^{\rho}|{\sF_{u}}]\leq \mE[\mE[|h(\widetilde{X}_r^{n})|^{\rho}|{\sF_{\widetilde{u}}}]|{\sF_{u}}]\lesssim (r-u)^{-\frac{d}{2\bar{p}_0}}\|h\|_{\rho \bar{p}_0}^\rho,
		\de 
		 Thus, we obtain (\ref{formula1+130}) for any $\rho\in(0,\infty)$. When $\rho=\infty$, (\ref{formula1+130}) is trivial.
	\end{proof}
\begin{remark}\label{remark5.6+Krylov-es}
	Paticularly, if $\rho\in (d/\beta\vee 1, \infty]$, we can take $\bar{p}_0=1$ to get, for every $r,u\in [0,1]$, $r-u\geq 2/n$,
	\be\label{formula1+130+0} 
	\|h(\widetilde{X}_r^{n})\|_{L_{\rho}(\Omega|\sF_u)}\leq N(r-u)^{-\frac{d }{2\rho}}\|h\|_{\rho}.
	\ee 
	Regarding Proposition \ref{proposition4.90}, if $u\in D_n$, then Theorem \ref{theorem4.5} implies that there exists a constant $N=N(d,\rho,\alpha,)$ such that for every $r>u$, we have 
	\ce
	\|h(\widetilde{X}_r^{n})\|_{L_{\rho}(\Omega|\sF_u)}\leq N(r-u)^{-\frac{d }{2\rho \bar{p}_0}}\|h\|_{\bar{p}_0\rho},~\mbox{if}~\rho< \infty.
	\de
	Since the inequality is trivial when $\rho=\infty$, we have 
	\be\label{formula1+1300} 
	\|h(\widetilde{X}_r^{n})\|_{L_{\rho}(\Omega|\sF_u)}\leq N(r-u)^{-\frac{d }{2\rho \bar{p}_0}}\|h\|_{\bar{p}_0\rho},~\mbox{if}~\rho\leq \infty.
	\ee 
\end{remark}

	\begin{lemma}\label{krlvy-estimate1}
		Let  $\widetilde{X}^n$ satisfy (\ref{formula1+12+1}) and suppose $f\in \mL_{p}^{q}([0,1])$ for some $p\in (d/\beta\vee 1,\infty)$, $q\in [1,\infty]$ saisfying $\frac{d}{p}+\frac{2}{q}<2$. Then, for any $(s,t)\in \Delta$ and any $m\geq 1$, we have 
		\be \label{formula1+18--1}
		\|\int_{s}^{t}f(r,\widetilde{X}_r^{n})dr\|_{L_m(\Omega|\sF_s)}\leq N_m\|f\|_{\mL_{p}^{q}([s,t])}(t-s)^{1-\frac{d}{2p}-\frac{1}{q}},
		\ee 
		where $N_m$ depends on $d,\alpha,\beta, p,q,c_0,c_1,m$, and
		\be\label{formula1+18} 
		\mE_s\exp\bigg(\int_{s}^{t}f(r,\widetilde{X}_r^{n})dr\bigg)\leq 2\exp(N\|f\|_{\mL_{p}^{q}([0,1])}^{1/(1-\frac{d}{2p})}),
		\ee 
		where $N$ depends on $d,\alpha,\beta, p,q,c_0,c_1$.
		
		Furthermore, assume additionally that there exists a continuous control $\sW_0$ on $\Delta$ and positive constants $M$, $\gamma_0$ such that 
		\be\label{control_sw_0} 
		(1/n)^{1-\frac{1}{q}}\|f\|_{\mL_{\infty}^q([s,t])}\leq \sW_0^{\gamma_0},~\forall 0\leq t-s\leq 1/n
		\ee
		and 
		\ce 
		\|f\|_{\mL_p^q([0,1])}+\sW(0,1)\leq M.
		\de  
		Then there exists a finite constant $N_0$, which depends only on  $d,\alpha,\beta, p,q,c_0,c_1,M,\gamma_0$, such that
		\be\label{exponention-11}
		\mE_s\bigg[\exp\bigg(\int_{s}^{t}f(r,\widetilde{X}_{r_n}^{n})dr\bigg)\bigg]\leq N_0.
		\ee
	\end{lemma}
	\begin{proof}
		The proof is given in Appendix \ref{appendixC}.
	\end{proof}

	\begin{theorem}[Krylov's estimate]\label{proposition0+4.13}
		Assume that  the Conditions  ($\sH_\sigma^g$)  and $\sH_b$ hold, let $X^n$ be the solution to  Eq. (\ref{EulerAPP3.81}). Furthermore,  let $f\in \mL_{p}^q([0,1])$ for some $2\leq p\in(d/\beta\vee 1,\infty)$, $q\in [2,\infty]$ satisfying $\frac{d}{p}+\frac{2}{p}<2$. Then we have
		\be\label{formula0+23} 
		\mE_s\int_{s}^{t}f(r,X_r^n)dr\leq N\|f\|_{\mL_{p}^q([0,1])},
		\ee 
		where $N$ depends only on $d,\alpha,\beta,p,q,c_0,c_1$.
	\end{theorem}
	\begin{proof}
		We may assume without loss of generality that $f$ is nonnegative. Let $\widetilde{X}^n$ be the solution to (\ref{formula1+12+1}).  Define
		\ce 
		\rho:=\exp\bigg(\int_{s}^{t}(\sigma^{-1}b_n)(r,\widetilde{X}_{r_n}^n)dB_r-\frac{1}{2}\int_{s}^{t}|(\sigma^{-1}b_n)(r,\widetilde{X}_{r_n}^n)|^2dr\bigg).
		\de 
		Using the fact that $\sigma^{-1}b_n\in\mL_{\infty}^q([0,1])$, we see that $\rho$ is a probability density. Let $\widetilde{P}$ denote the probability measure defined by $d\widetilde{P}/dP=\rho$ and use the notation $\widetilde{E}$ for the expectation under $\widetilde{P}$. Notice 
		\ce 
		dX_r^n=\sigma(r,X_{r_n}^n)d\widetilde{B}_r+\int_{|z|\leq R}g(r,X_{{r_n^-}}^{n}(x),u)\widetilde{N}(drdu)
		\de 
		with 
		\ce 
		\widetilde{B}_r=-\int_{0}^{r}(\sigma^{-1}b_n)(s,X_{s_n}^n)ds+B_r,~t\in[0,1],
		\de 
		which is a Brownain motion under  $\widetilde{P}$ by Girsanov's theorem and $N(drdu)$ is still a Possion random measure with the same compensator $dt\upsilon(du)$. Then, it follows from Girsanov's theorem and H\"{o}lder inequality for $1< u\leq p$  that
		\be\label{formula1+222} 
		\mE_s\bigg[\int_{s}^{t}f(r,X_r^n)dr\bigg]&=&\widetilde{\mE}_s\bigg[\rho^{-1}\int_{s}^{t}f(r,\widetilde{X}_r^n)dr\bigg]\no\\
		&\leq &\widetilde{\mE}_s[\rho^{-\frac{u}{u-1}}]^{1-\frac{1}{u}} \widetilde{E}_s\bigg[\bigg(\int_{0}^{t}f(r,\widetilde{X}_r^n)dr\bigg)^u\bigg]^{\frac{1}{u}}.
		\ee
		From (\ref{formula1+18--1}), we immediately get that 
		\ce 
		\widetilde{\mE}_s\bigg[\bigg(\int_{s}^{t}f(r,\widetilde{X}_r^n)dr\bigg)^u\bigg]^{\frac{1}{u}}\lesssim \|f\|_{\mL_{p}^q([0,1])}.
		\de 
		Set $w_r=(\sigma^{-1}b_n)(r,\widetilde{X}_{r_n}^n)$ and 
		\ce 
		\cW=\exp\bigg(\frac{1}{2}\bigg(\frac{u}{u-1}+(\frac{u}{u-1})^2\bigg)\int_{s}^t|w_r|^2dr\bigg).
		\de 
		Using the Cauchy-Schwarz inequality, we have 
		\ce 
		\widetilde{\mE}_s[\rho^{-\frac{u}{u-1}}]&=&\widetilde{\mE}_s[\cW\exp(\frac{u}{u-1}\int_{s}^tw_rd\widetilde{B}_r-\frac{1}{2}(\frac{u}{u-1})^2\int_{s}^t|w_r|^2dr)]\\
		&\leq &\widetilde{\mE}_s[|\cW|^2]^{\frac{1}{2}}\widetilde{E}_s[\exp(\frac{2u}{u-1}\int_{s}^tw_rd\widetilde{B}_r-(\frac{u}{u-1})^2\int_{s}^t|w_r|^2dr)]^{\frac{1}{2}}.
		\de 
		From the martingale properties, we have 
		\ce 
		\widetilde{\mE}_s[\exp(\frac{2u}{u-1}\int_{s}^tw_rd\widetilde{B}_r-(\frac{u}{u-1})^2\int_{s}^t|w_r|^2dr)]=1.
		\de 
		For the term $\widetilde{E}_s[|\cW|^2]$, we recall Conditions  ($\sH_\sigma^g$), $\sH_b$  and the uniform ellipticity of $\sigma$, which imply that the function $f:=|\sigma^{-1}b_n|^2$ belongs to $\mL_{p\bar{p}_0/2}^{q/2}([0,1])\cap\mL_{\infty}^{q/2}([0,1])$ and satisfies that 
		\ce 
		(1/n)^{1-\frac{2}{q}}\|f\|_{\mL_{\infty}^{q/2}([s,t])}\lesssim [(1/n)^{\frac{1}{2}-\frac{1}{q}}\|b_n\|_{\mL_{\infty}^q}([s,t])]^2\lesssim m_n(s,t)^{2\theta},~\forall~0\leq t-s\leq 1/n.
		\de 
       Applying Lemma \ref{Khasminskii's lemma} , we see that $\widetilde{E}_s[|\cW|^2]$ is bounded uniformly in $n$ and then $E_s[\rho^{m}]$. From (\ref{formula1+222}) and the above estimates, we complete the proof.
	\end{proof}

	\section{Parabolic integral-differential equations with distributional forcing}

For each $r\in[1,\infty]$, we denote its H\"{o}lder conjugate by $r^*$, satisfying $\frac{1}{r}+\frac{1}{r^*}=1$. Given a Banach space $\sE$, it's dual space denoted by $\sE^*$, and the dual pairing between $\sE$ and $\sE^*$ is represented by $\langle \cdot,\cdot \rangle_{\sE,\sE^*}$. We now consider the following second-order integral-differential equations
\be \label{EquationA.1}
(\partial_s+a^{ij}\partial_{i,j}+\sL_{\upsilon,R}^g)u=f,~~u(1,\cdot)=0,
\ee
and 
\be \label{EquationA.2+1}
\partial_sv-\partial_{ij}(a^{ij}v)+\sL_{\upsilon,R}^gv+h=0,~~v(0,\cdot)=0,
\ee
under the assumptions outlined in the following condition ($\sH_a^f$).

\textbf{Condition} ($\sH_a^f$) 
\begin{enumerate}
	\item[1.] $a$ is a $d\times d$-symmetric matrix-valued measurable function on $[0,1]\times \mR^d$. There exists a constant $C_1\geq 1$ such that for every $t\in[0,1]$ and $x\in\mR^d$ 
	\be\label{coefficientsA_1} 
	C_1^{-1} \mI\leq a(t,x)\leq C_1 \mI.
	\ee 
	Furthermore, there exist constants $\alpha\in (0,1]$ and $C_2>0$ such that for every $t\in[0,1]$ and $x,y\in\mR^d$ 
	\be \label{formula0+A.0.}
	|a(t,x)-a(t,y)|\leq C_2|x-y|^\alpha.
	\ee  
	Additionally, $a(t,\cdot)$ is weakly differentiable for a.e. $t\in[0,1]$ and $C_3:=\|\nabla a\|_{\mL_{p_0}^{\infty}([0,1])}$ is finite for some $p_0\in (d,\infty)$.
	\item[2.] $f\in\mH_{-1,p}^q([0,1])$ and $h\in\mH_{-1,p^*}^{q^*}([0,1])$ for some $p,q\in(1,\infty)$ with $\frac{1}{p}+\frac{1}{p_0}<1$.
	\item[3.] $g$ is a function on $\in[0,1]\times \mR^d\times \mR^d$ stisfying, for some $\beta\in[1,2)$
	\ce 
	\varGamma_{0,\beta}^{0,R}g, \varGamma_{0,2}^{0,R}g\in\mL^{\infty}([0,1]),~\lim_{\varepsilon\downarrow 0}\|\varGamma_{0,2}^{0,\varepsilon}g\|_{\mL^{\infty}([0,1])}=0
	\de 
	and  
	\ce 
	\varGamma_{1,2}^{0,R}g\in\mL^{\infty}([0,1]).
	\de 
\end{enumerate}

\begin{definition}
	A measurable function $u:[0,1]\times \mR^d\mapsto\mR$ is considered a solution to Eq. (\ref{EquationA.1}) if $u\in \mH_{1,p}^q([0,1])$, $\partial_su\in \mH_{-1,p}^q([0,1])$, $u(1,\cdot)=0$ and Eq. (\ref{EquationA.1}) holds in $\mH_{-1,p^*}^{q^*}([0,1])$. Specifically, for every $\phi\in\mH_{1,p^*}^{q^*}([0,1])$, the following equality must be satisfied
	\be\label{Equ-Solu-1}
	\int_{0}^{1}\langle(\partial_t+a^{i,j}\partial_{i,j}+\sL_{\upsilon,R}^g)u_t,\phi_t\rangle_{\mH_{-1,p}(\mR^d),\mH_{1,p^*}(\mR^d)} dt=\int_{0}^{1} \langle f_t,\phi_t\rangle_{\mH_{-1,p}(\mR^d),\mH_{1,p^*}(\mR^d)} dt.
	\ee 
	
	Likewise, a measurable function $v:[0,1]\times \mR^d\mapsto\mR$ is considered a solution to Eq. (\ref{EquationA.2+1}) if $v\in \mH_{1,p^*}^{q^*}([0,1])$, $\partial_sv\in \mH_{-1,p^*}^{q^*}([0,1])$, $u(1,\cdot)=0$ and Eq. (\ref{EquationA.2+1}) holds in $\mH_{-1,p}^{q}([0,1])$. Specifically, for every $\phi\in\mH_{1,p}^{q}([0,1])$, the following equality must be satisfied
	\be\label{Equ-Solu-22}
	\int_{0}^{1}\langle(\partial_t v_t-\partial_{i,j} (a^{i,j} v_t)+\sL_{\upsilon,R}^g v_t,\phi_t\rangle_{\mH_{-1,p^*}(\mR^d),\mH_{1,p}(\mR^d)} dt+\int_{0}^{1} \langle h_t,\phi_t\rangle_{\mH_{-1,p^*}(\mR^d),\mH_{1,p}(\mR^d)} dt=0.
	\ee 
\end{definition}

Additionally, we will utilize the following lemma from \cite[Lemma A.2]{Ling2022}.
\begin{lemma}\label{Multipliction-norm1}
	Let $p,p_1,p_2$ be real number in $(1,\infty)$ and let $\theta\in (0,1]$.
	
	(i) Assume that $p_1,p_2\geq p$ and the indices satisfy the condition $\frac{1}{p}\leq \frac{1}{p_1}+\frac{1}{p_2}<\frac{1}{p}+\frac{\theta}{d}$. Then the pointwise multiplication is a continuous bilinear map
	\ce 
	\mH_{\theta,p_1}(\mR^d)\times \mH_{\theta,p_2}(\mR^d)\to \mH_{\theta,p}(\mR^d).
	\de 
	
	(ii) Assume that $p_1\geq p,p_2\geq p_1^*$,  and the indices fulfill $\frac{1}{p}\leq \frac{1}{p_1}+\frac{1}{p_2}<\frac{1}{p}+\frac{\theta}{d}$. Then the pointwise multiplication is a continuous bilinear map 
	\ce 
	\mH_{-\theta,p_1}(\mR^d)\times \mH_{-\theta,p_2}(\mR^d)\to \mH_{-\theta,p}(\mR^d).
	\de 
	
	(iii) Let $g$ be a bounded measurable function such that its gradient $\nabla g$ belongs to $ L_{p_0}(\mR^d)$ for some $p_0\in (d,\infty)$. Let $f\in \mH_{1,p^*}(\mR^d)$, $h\in \mH_{-1,p}(\mR^d)$, and suppose  $\frac{1}{p}+\frac{1}{p_0}<1$. Then the product $fg$ belongs to $\mH_{1,p^*}(\mR^d)$, $gh$ belongs to $\mH_{-1,p}(\mR^d)$ and he following inequalities hold:
	\ce 
	\|fg\|_{\mH_{1,p^*}(\mR^d)}&\lesssim &(\|g\|_{\infty}+\|\nabla g\|_{p_0})\|f\|_{\mH_{1,p^*}(\mR^d)},\\
	\|gh\|_{\mH_{-1,p}(\mR^d)}&\lesssim &(\|g\|_{\infty}+\|\nabla g\|_{p_0})\|h\|_{\mH_{-1,p}(\mR^d)}.
	\de 
\end{lemma}

\begin{remark}\label{remarkA.2}
	Based on the definitions provided, it is understood that $a^{i,j}\partial_{i,j}u$ is a well-defined distrubition in $\mH_{-1,p}^q(\mR^d)$  and  $\partial_{i,j}(a^{i,j}v)$ in $\mH_{-1,p^*}^{q*}(\mR^d)$, as established by Lemma \ref{Multipliction-norm1}.
\end{remark}
\begin{theorem}\label{TheoremA.4} 
	Under  Condition ($\sH_a^f$), there exists a unique solution $u$ to Eq. (\ref{EquationA.1}) and a unique solution $v$ to Eq. (\ref{EquationA.2+1}). Furthermore, we have 
	\be \label{formula0+A.3}
	\|u\|_{\mH_{1,p}^q([0,1])}+\|\partial_tu\|_{\mH_{-1,p}^q([0,1])}\leq N\|f\|_{\mH_{-1,p}^q([0,1])},
	\ee 
	\be \label{formula1+A.3}
	\|v\|_{\mH_{1,p^*}^{q^*}([0,1])}+\|\partial_tv\|_{\mH_{-1,p^*}^{q^*}([0,1])}\leq N\|f\|_{\mH_{-1,p^*}^{q^*}([0,1])},
	\ee 
	where $N$ is a finite positive constant depending on $d, p, q, p_0, \beta, C_1, C_2$.
\end{theorem}

Before presenting the proof of the aforementioned theorem, we will first establish several auxiliary results to facilitate our understanding and derivation.

\begin{lemma}[\cite{XichengZhang2017} Lemma 4.1]\label{lemma0+A.5}
	Let $\phi$ be a nonzero smooth function with compact suport. Denote $\phi^z=\phi(x-z)$. For any $r\in \mR$ and $p>1$, there exists a constant $N>1$  only depending on $r,p,\phi$ such that for any $f\in\mH_{r,p}(\mR^d)$, we have
	\ce 
	N^{-1}\|f\|_{\mH_{r,p}(\mR^d)}\leq \left(\int \|f\phi^z\|_{\mH_{r,p}(\mR^d)}^pdz\right)^{1/p}\leq N\|f\|_{\mH_{r,p}(\mR^d)}.
	\de   
\end{lemma}

The following well-known results can be found in \cite{Xicheng-2011}  and \cite[Theorem 1.1.1]{Stein1970}.
\begin{lemma}\label{lemmaA.7+0}
	
	(i) Let $\mB$ be a Banach space and a locally integrable function $f:\mR^d\to \mB$  with $\nabla f\in \mL_{loc}^1(\mR^d,\mB^d)$. There exists a Lebesgue null set $E$ such that for all $x,y\in E$, 
		\be \label{formulaA.7+00}
		\|f(x)-f(y)\|_{\mB}\leq 2^d\int_{0}^{|x-y|}\fint_{B_r}\|\nabla f\|_{\mB}(x+w)+\|\nabla f\|_{\mB}(y+w)dwdr,
		\ee
		where  $\fint_{B_s}:=\frac{1}{|B_s|}\int_{B_s}$. Furthermore, if $\nabla f\in \mL^p(\mR^d;\mB^d)$ for some $p\in(d,\infty]$, then there exists a constant $C_{d,p}>0$ such that
		\be \label{formulaA.7+100}
		\|f(x)-f(y)\|_{\mB}\leq C_{d,p}|x-y|^{1-d/p}\|\nabla f\|_{\mL^p(\mR^d;\mB)}.
		\ee
		
		(ii) For $p\in(1,\infty]$, there exists a constant $N_{d,p}>0$ such that for all $f\in \mL^p(\mR^d)$, the maximal function $\sM f$ satisfies
		\ce 
		\|\sM f\|_p\leq N_{d,p}\|f\|_p.
		\de 

\end{lemma}

It is worth noting that (\ref{formulaA.7+100}) either is a direct consequence of \cite[Remark 12.49]{Giovanni2017} or follows easily from (\ref{formulaA.7+00}).

\begin{lemma}\label{lemma)+A.6}
	Let $\psi$ be a smooth function supported in the ball $B_1:=\{x\in\mR^d:|x|\leq 1\}$. For each $r>0$, $z\in\mR^d$, we define $B_r^z:=\{x\in\mR^d:|x-z|\leq r\}$ and the averaged function 
	\ce 
	a^z(t)=\frac{1}{|B_r^z|}\int_{B_r^z}a(t,y)dy. 
	\de  
	We then have the following results, for any  $r,t\in\mR_+$, 
	\be\label{formula0+A.4.} 
	\lim_{r\downarrow 0}\sup_{(t,z)\in[0,1]\times \mR^d}\|(a-a^z)\psi(\frac{\cdot-z}{r})\|_{\mH_{1,p_0}(\mR^d)}=0
	\ee 
	and 
	\be\label{formula0+0A.5.}
	\lim_{r\downarrow 0}\sup_{(t,z)\in[0,1]\times \mR^{d}}\|(a-a^z)\psi(\frac{\cdot-z}{r})\|_{\infty}=0.
	\ee
\end{lemma}
\begin{proof}
From  \cite[Lemma 8]{Ling2022},	 (\ref{formula0+A.4.}) holds. For (\ref{formula0+0A.5.}), we have 
	\ce 
	&&|a(t,x)-a^z(t)|\\
	&\leq& \fint_{B_r^z}|a(t,y)-a(t,x)|dy\\
	&\leq &\fint_{B_r^z}|y-x|\sup_{|w|\neq 0}|w|^{-1}|a(t,y+w)-a(t,y)\big|dy\\
	&\leq &(r+|z-x|)\fint_{B_r^z}\sup_{|w|\neq 0}|w|^{-1}\big|a(t,y+w)-a(t,y)\big|dy.
	\de 
	Applying H\"{o}lder's  inequality and (\ref{imparticular-estimate-1}), we have 
	\ce 
	&&|a(t,x)-a^z(t)|\\
	&\lesssim &(r+|z-x|)r^{-\frac{d}{p_0}}\left[\int_{B_r^z}\left(\sup_{|w|\neq 0}|w|^{-1}\big|a(t,y+w)-a(t,y)\big|\right)^{p_0}dy\right]^{\frac{1}{p_0}}\\
	&\leq &(r+|z-x|)r^{-\frac{d}{p_0}}\left[\int_{\mR^d}\left(\sup_{|w|\neq 0}|w|^{-1}\big|a(t,y+w)-a(t,y)\big|\right)^{p_0}dy\right]^{\frac{1}{p_0}}\\
	&\lesssim&(r+|z-x|)r^{-\frac{d}{p_0}}\|a_t\|_{\mH_{1,p_0}},
	\de 
where the implicit constant is independent of  $t,z,r$. Therefore, by the fact $|z-x|<r$ when $\frac{|x-z|}{r}<1$, we have 
	\ce 
	|(a-a^z)\psi(\frac{x-z}{r})|&\lesssim&r^{1-d/p_0}\|a_t\|_{\mH_{1,p_0}}|\psi(\frac{x-z}{r})|,
	\de 
which	implies   (\ref{formula0+0A.5.}).

\end{proof}
\begin{lemma}\label{lemma1+A.6}
	Let $\psi$ be a smooth, positive function supported in the ball $B_1:=\{x\in\mR^d:|x|\leq 1\}$. For each $r>0$, $z\in\mR^d$.	 For each $r>0$, $z\in\mR^d$, we define
	\ce 
	g^z(t,w)=\frac{1}{|B_r^z|}\int_{B_r^z}g(t,y,w)dy,
	\de  
	 where  $B_r^z:=\{x\in\mR^d:|x-z|\leq r\}$. Then  there exists a constant $N$, independent of $t,z,r$ such that for any $x,z\in\mR^d$, $r,t\in\mR_+$,
	\be\label{formula1+A.5.}
     (\varGamma_{0,2}^{0,R}(g^z(t,w)-g(t,x,w)))^{\frac{1}{2}}\psi(\frac{x-z}{r})\leq Nr\|\varGamma_{1,2}^{0,R}g_t\|_{\infty}\psi(\frac{x-z}{r}).
	\ee
	Moreover, we also have 
	\be\label{formula1+A.5.+1}
	\varGamma_{0,2}^{0,R}g_t^z&\leq&\|\varGamma_{0,2}^{0,R}g_t\|_{\infty}.
	\ee
\end{lemma}
\begin{proof}
	Let $\mB=\mL^2(B_R;\upsilon)$. By using (\ref{formulaA.7+00}), and Minkowski's inequality, we have 
	\ce 
	&&\varGamma_{0,2}^{0,R}(g^z(t,\cdot)-g(t,x,\cdot))\\
	&=&\int_{|w|<R}\bigg(\fint_{B_r^z}g(t,y,w)-g(t,x,w)dy\bigg)^2\upsilon(dw)\\
	&\leq &\bigg[\fint_{B_r^z}\bigg(\int_{|w|<R}(g(t,y,w)-g(t,x,w))^2\upsilon(dw)\bigg)^{\frac{1}{2}}dy\bigg]^2\\
	&\leq &\bigg[2^d\fint_{B_r^z}\int_{0}^{|y-x|}\fint_{B_s}[\varGamma_{1,2}^{0,R} g(t,u+y)]^{\frac{1}{2}}+[\varGamma_{1,2}^{0,R}g(t,u+x)]^{\frac{1}{2}}dudsdy\bigg]^2.
	\de
	Using the assumptions on $g$, we obtain 
	\ce 
	&&\fint_{B_r^z}\int_{0}^{|y-x|}\fint_{B_s}[\varGamma_{1,2}^{0,R}g(t,u+y)]^{\frac{1}{2}}dudsdy\\
	&\lesssim&(r+|z-x|)\|\varGamma_{1,2}^{0,R}g_t\|_{\infty}^{\frac{1}{2}},
	\de 
	where the implicit constant is independent of $t,z,r$. This yields (\ref{formula1+A.5.}) by the fact $|z-x|<r$.
	
	An analogous argument establishes (\ref{formula1+A.5.+1}). This completes the proof.
\end{proof}
\begin{lemma}\label{lemmaA.9+}\cite[Lemma 2.5]{Kim-2008}
	 Let $a_k:\mR^+\to \mR^d\times \mR^d$ for $k=1,\cdots,n$ be a measurable functions, satisfying condition (\ref{coefficientsA_1}).
	For fixded $\iota\in \mR$, $p\in (1,\infty)$.  Suppose that for each $k$, $f^k\in \mH_{\iota,p}^q([0,1])$ and let $u^k\in \mH_{\iota,p}^q([0,1])$ be the solution to the following partial differential equation 
	\ce 
	(\partial_s+a_k^{i,j}\partial_{ij})u^k=f^k,~u(1,\cdot)=0.
	\de 
	Then there exists a positive constant $C$ such that
	\ce 
	\int_{0}^{1}\Pi_{k=1}^n\|\nabla^2u_t^k\|_{\mH_{\iota,p}(\mR^d)}^pdt\leq C\sum_{k=1}^{n}\int_{0}^{1}\|f_t^k\|_{\mH_{\iota,p}(\mR^d)}^p\Pi_{j\neq k}^n\|\nabla^2u_t^j\|_{\mH_{\iota,p}(\mR^d)}^pdt.
	\de 
\end{lemma}
\begin{lemma}\label{LemmaA.10.}
	Under Condition  ($\sH_a^f$), there exists a unique solution $u$ to  Eq. (\ref{EquationA.1}) that satisfies Estimate (\ref{formula0+A.3}), provided that $q\geq p$.
\end{lemma}
\begin{proof}

	By employing the standard continuity method (as detailed in, for instance, \cite[p.14.,Theorem 4]{Krylov-2008}), it suffices to establish the a priori estimate (\ref{formula0+A.3}) for Eq. (\ref{EquationA.1}). Leveraging the Marcinkiewicz interpolation theorem, we need only consider the case when $q=n_0$ for any integer $n_0\geq 1$. Let $n_0\geq 1$ be a fixed integer. It is worth noting that if $u$ is a solution to Eq. (\ref{EquationA.1}), then, by Remark \ref{remarkA.2}, it suffices to demonstrate the existence of a positive constant  $N=N(d,p,q,p_0,C_1,C_2)$ such that 
	\ce 
	\|u\|_{\mH_{1,p}^{n_0p}([0,1])}+\|\mathscr{L}_{\upsilon,R}^{g}u\|_{\mH_{-1,p}^{n_0p}([0,1])}\leq N\|f\|_{\mH_{-1,p}^{n_0p}([0,1])},
	\de 
	since this estimate implies that
	\ce 
	\|\partial_tu\|_{\mH_{-1,p}^{n_0p}([0,1])}\leq N\|f\|_{\mH_{-1,p}^{n_0p}([0,1])}.
	\de 
	
	Let $\varrho>0$ be a fixed constant and $\phi$ be a nonnegative, smooth function supported in the ball $B_\varrho:=\{x\in\mR^d:|x|\leq \varrho\}$ with $\|\phi\|_{p}=1$. For each $z\in\mR^d$, $a^z(t)$ is defined as in Lemma \ref{lemma)+A.6}
	and we set 
	\ce 
	\phi^z(x):=\phi(x-z),~u^z(s,x):=u(s,x)\phi^z(x),~f^z(s,x):=f(s,x)\phi^z(x).
	\de 
	Then $u^z$ satisfies the equation
	\be\label{Equ-PDF-z}
	\partial_tu^z+a^{ij}(z)\partial_{ij}^2u^z=F^z,~~u^z(1,\cdot)=0,
	\ee
	where 
	\ce 
	F^z:=f\phi^z+2a^{ij}\partial_{i}u\partial_{j}\phi^z+a^{ij}u\partial_{ij}\phi^z+(a^{ij}(z)-a^{ij})\partial_{ij}u^z-(\mathscr{L}_{\upsilon,R}^{g}u)\phi^z.
	\de 
	The proof now proceeds through several steps.
	
	\textbf{Step 1.} We demonstrate that for each $t\in [0,1]$,
	\be \label{formula0+A.4}
	\bigg(\int_{\mR^d}\|F_t^z\|_{\mH_{-1,p}(\mR^d)}^pdz\bigg)^{\frac{1}{p}}\lesssim \|f_t\|_{\mH_{-1,p}(\mR^d)}+C_\varepsilon\|u_t\|_{p}+(o_\varepsilon(1)+o_\varrho(1))\|u_t\|_{\mH_{1,p}(\mR^d)},
	\ee 
	where $o_\varepsilon(1)$, $o_\varrho(1)$ are  constants independent of $t$ and $\lim_{\varepsilon\downarrow 0}o_\varepsilon(1)=0$, $\lim_{\varrho\downarrow 0}o_\varrho(1)=0$ and $C_\varepsilon$ is a constant depending on $\varepsilon$.  Applying Lemma \ref{lemma0+A.5} and Lemma \ref{Multipliction-norm1}, we can see that 
	\ce 
	&&\int_{\mR^d}\|f\phi^z+2a^{ij}\partial_{i}u\partial_{j}\phi^z+a^{i,j}u\partial_{ij}\phi^z\|_{\mH_{-1,p}(\mR^d)}^pdz\\
	&\lesssim&\|f\|_{\mH_{-1,p}(\mR^d)}^p+\|2a^{ij}\partial_{i}u\|_{\mH_{-1,p}(\mR^d)}^p+\|a^{ij}u\|_{\mH_{-1,p}(\mR^d)}^p\\
	&\lesssim&\|f\|_{\mH_{-1,p}(\mR^d)}^p+\|u\|_{p}^p.
	\de 
    Note that the implicit constant above depends on $\varrho$. Let $\varphi$ be a smooth function on $\mR^d$ such that $\varphi=1$ if $|x|\leq \varrho$ and $\varphi(x)=0$ if $|x|\geq 2\varrho$. Define $\varphi^z(x)=\varphi(x-z)$. Since $p_0> d $ and $\frac{1}{p}+\frac{1}{p_0}<1$, applying Lemmas \ref{lemma0+A.5} and Lemma \ref{Multipliction-norm1}, we have 
	\ce
	&&\|(a^{ij}(z)-a^{ij})\partial_{i,j}u^z\|_{\mH_{-1,p}(\mR^d)}\\
	&\leq& N\big[\|(a^{ij}(z)-a^{ij})\varphi^z\|_{\infty}+\|(a^{ij}(z)-a^{ij})\varphi^z\|_{\mH_{1,p_0}(\mR^d)}\big]\|\partial_{i,j}u^z\|_{\mH_{-1,p}(\mR^d)}.
	\de 
	By Lemma \ref{lemma)+A.6}, we have
	\ce 
	\sup_z\|(a^{ij}(z)-a^{ij})\varphi^z\|_{\infty}+\sup_z\|(a^{ij}(z)-a^{ij})\varphi^z\|_{\mH_{1,p_0}(\mR^d)}=o_\varrho(1).
	\de  
	Therefore 
	\ce 
	\|(a^{ij}(z)-a^{ij})\partial_{ij}u^z\|_{\mH_{-1,p_0}(\mR^d)}\leq o_\varrho(1)\|u^z\|_{\mH_{1,p}(\mR^d)}.
	\de 
	By Fubini's theorem and $\|\phi\|_{p}=1$, it is easy to see that for any $h\in L_{p}(\mR^d)$
	\ce 
	\int_{\mR^d}\|h\phi^z\|_{p}^pdz=\|h\|_{p}.
	\de 
	Hence, by Minkowski inequality and Lemma \ref{lemma0+A.5}, we get
	\be\label{formula-uz} 
	\int_{\mR^d}\|u^z\|_{\mH_{1,p}}^pdz&\leq& \int_{\mR^d}\|u^z\|_{p}^pdz+\int_{\mR^d}\|\nabla u^z\|_{p}^pdz\no\\
	&\leq& \int_{\mR^d}\|u^z\|_{p}^pdz+\int_{\mR^d}\|\nabla u\phi^z\|_{p}^pdz+\int_{\mR^d}\| u\nabla\phi^z\|_{p}^pdz\no\\
	&\leq&\|\nabla u\|_{p}^p+C_p\|u\|_{p}^p.
	\ee 
	This shows that 
	\ce 
	\bigg(\int_{\mR^d}\|(a^{ij}(z)-a^{ij})\partial_{ij}u^z\|_{\mH_{-1,p}(\mR^d)}dz\bigg)^{\frac{1}{p}}
	\leq o_\varrho(1)(\|\nabla u\|_{p}+C_p\|u\|_{p}).
	\de 
	
	Considering the term $\int_{\mR^d}\|\mathscr{L}_{\upsilon,R}^{g}u\phi^z\|_{\mH_{-1,p}(\mR^d)}^pdz$, we will apply the assumption about $g$ and (\ref{formula1+5}) of Lemma \ref{lemma3.2} to estimate it as follows.
	
	Applying triangle inequality,  Lemmas \ref{lemma0+A.5} and  \ref{Multipliction-norm1}, we get 
	\ce 
	&&\int_{\mR^d}\|(\mathscr{L}_{\upsilon,R}^{g}u)\phi^z\|_{\mH_{-1,p}(\mR^d)}dz\\
	&\leq&\|(\mathscr{L}_{\upsilon,R}^{g(z)}-\mathscr{L}_{\upsilon,R}^{g})u\|_{\mH_{-1,p}(\mR^d)}^p+\|(\mathscr{L}_{\upsilon,R}^{g(z)}-\mathscr{L}_{\upsilon,\varepsilon}^{g(z)})u\|_{\mH_{-1,p}(\mR^d)}^p
	+\|\mathscr{L}_{\upsilon,\varepsilon}^{g(z)}u\|_{\mH_{-1,p}(\mR^d)}^p\\
	&:=&B_1+B_2+B_3,
	\de 
	where $g^z(t,y)$ defined as in Lemma \ref{lemma1+A.6}. Denote 
	$$\Delta_z g(x,y):=g(x,y)-g^z(y),~\Delta^{ij} g(x,y,z):=(\Delta g(x,z))^i(g^z(y))^j,$$
	$$\varLambda_z u(x,y):=u(x+\Delta g(x,z))-u(x),$$
	$$\varLambda u(x,w_2):=u(x+w_2)-u(x),$$
	$$\Delta\varLambda u(x,w):=\varLambda u(x+w_1,w_2)-\varLambda u(x,w_2),$$
	where $w=(w_1,w_2)$.
	
	Below  in the netx setps, for the convenience of the function notation, we drop the time variable $t$, which will not cause confusion from the context.
	
	 Set  $\sup_\psi=\sup_{\psi\in C_c^{\infty}(\mR^d),\|\psi\|_{\mH_{1,p^*}(\mR^d)}\leq 1}.$  For the term $B_1$, using the definition of operator $\mathscr{L}_{\upsilon,R}^{g}$ and the triangle inequality, we have 
	\ce 
	(B_1)^{1/p}&=&\sup_\psi\int_{\mR^d}(\mathscr{L}_{\upsilon,R}^{g^z}u-\mathscr{L}_{\upsilon,R}^{g}u)(x)\psi(x) dx\\
	&=&\sup_\psi\int_{\mR^d}\int_{|y|<R} \bigg[u(x+g^z(y))-u(x+g(x,y))+\Delta_z g(x,y)\cdot\nabla u(x)\bigg]\upsilon(dy)\psi(x)dx\\
	&\leq &\sup_\psi\int_{\mR^d}\int_{|y|<R}\bigg(\varLambda_z u(x,y)(x+g^z(y),y)-\varLambda_z u(x,y)(x,y)\bigg)\upsilon(dy)\psi(x)dx\\
	&&+\sup_\psi\int_{\mR^d}\int_{|y|<R} \bigg[u(x+\Delta_z g(x,y))-u(x)-\Delta_z g(x,y)\cdot\nabla u(x)\bigg]\upsilon(dy)\psi(x)dx\\
	&=:&B_{11}+B_{12}.
	\de 
	For the first term $B_{11}$, apply Fubini's theorem  and the integration by parts formula to get
	\ce 
	B_{11}&=&\sup_\psi\int_{\mR^d}\int_{|y|<R}\int_{[0,1]}(g^z(y))^i\partial_{i}\varLambda_z u(x,y)(x+\theta g^z(y),y)d\theta \upsilon(dy)\psi(x)dx\\
	&=&\sup_\psi\int_{|y|<R}\int_{[0,1]}\int_{\mR^d}(g^z(y))^i\partial_{i}\varLambda_z u(x,y)(x+\theta g^z(y),y)\psi(x)dxd\theta \upsilon(dy)\\
	&= &\sup_\psi\int_{|y|<R}\int_{[0,1]}\int_{\mR^d}(g^z(y))^i\varLambda_z u(x,y)(x+\theta g^z(y),y)\partial_{i}\psi(x)dxd\theta \upsilon(dy)\\
	&\leq &\sup_\psi\int_{|y|<R}\int_{[0,1]}\int_{\mR^d}(g^z(y))^i(\varLambda_z u(x,y)(x+\theta g^z(y),y)-\varLambda_z u(x,y)(x))\partial_{i}\psi(x)dxd\theta \upsilon(dy)\\
	&&+\sup_\psi\int_{|y|<R}\int_{[0,1]}\int_{\mR^d}(g^z(y))^i\varLambda_z u(x,y)(x)\partial_{i}\psi(x)dxd\theta \upsilon(dy).
	\de 
	Then  applying the assumption about $g$, Lemma \ref{lemma3.2} and Fubini's theorem to estemte the last two terms, we get 
	\ce 
	B_{11}&\leq&\|\int_{|y|<R}g^z(y)\left (\Delta_z g(x,y)g^z(y)\right)^{\frac{1}{2}}\sup_{(w_1,w_2)\neq 0}(|w_1||w_2|)^{-\frac{1}{2}}|\Delta\varLambda u(x,w)|\upsilon(dy)\|_p\\
	&&+\|\int_{|y|<R}|\Delta_z g(x,y),y)g^z(y)|\sup_{(w_1)\neq 0}|w_1|^{-1}|u(x+w_1)-u(x)|\upsilon(dy)\|_p\\
	&\leq&(\varGamma_{0,2}^{0,R}g^z)^{\frac{3}{4}}\|(\varGamma_{0,2}^{0,R}|\Delta_z g(x,y)|)^{\frac{1}{4}}\sup_{(w_1,w_2)\neq 0}(|w_1||w_2|)^{-\frac{1}{2}}|\Delta\varLambda u(x,w)|\|_{p}\\
	&&+(\varGamma_{0,2}^{0,R}g^z)^{\frac{1}{2}}\|(\varGamma_{0,2}^{0,R}|\Delta_z g(x,y)|)^{\frac{1}{2}}\sup_{(w_1)\neq 0}|w_1|^{-1}|u(x+w_1)-u(x)|\|_p.
	\de 
	Using  Lemma \ref{lemma0+A.5} and Lemma \ref{lemma1+A.6} with $r=\varrho$, from above inequality we obtain 
	\ce 
	B_{11}&\lesssim& (\varGamma_{0,2}^{0,R}g^z)^{\frac{3}{4}}\|(\varGamma_{0,2}^{0,R}|\Delta_z g(x,y)|)^{\frac{1}{4}}\varphi^z\sup_{(w_1,w_2)\neq 0}(|w_1||w_2|)^{-\frac{1}{2}}|\Delta\varLambda u(x,w)|\|_{p}\\
	&&+(\varGamma_{0,2}^{0,R}g^z)^{\frac{1}{2}}\|(\varGamma_{0,2}^{0,R}|\Delta_z g(x,y)|)^{\frac{1}{2}}\varphi^z\sup_{(w_1)\neq 0}|w_1|^{-1}|u(x+w_1)-u(x)|\|_p\\ &\lesssim&\varrho^{\frac{1}{4}}\|\varGamma_{0,2}^{0,R}g\|_{\mL^{\infty}([0,1])}^{\frac{3}{4}}\|\varGamma_{1,2}^{0,R}g\|_{\mL^{\infty}([0,1])}^{\frac{1}{4}}\|\sup_{(w_1,w_2)\neq 0}(|w_1||w_2|)^{-\frac{1}{2}}|\Delta\varLambda u(x,w)|\|_{p}\\
	&&+\varrho^{\frac{1}{2}}\|\varGamma_{0,2}^{0,R}g\|_{\mL^{\infty}([0,1])}^{\frac{1}{2}}\|\varGamma_{1,2}^{0,R}g\|_{\mL^{\infty}([0,1])}^{\frac{1}{2}}\|\sup_{(w_1)\neq 0}|w_1|^{-1}|u(x+w_1)-u(x)|\|_p.
	\de
	Using (\ref{analysis-operater-l}) and (\ref{imparticular-estimate-1}) in Lemma \ref{lemma3.2}, we have
	\ce 
	B_{11}\lesssim (\varrho^{\frac{1}{4}}+\varrho^{\frac{1}{2}})\|u\|_{\mH_{1,p}(\mR^d)}.
	\de 
	For the term $B_{12}$, noting that  
	\ce 
(\partial_i u)(x+\vartheta\Delta_z g(x,y))
	=\partial_i u(x+\vartheta\Delta_z g(x,y))-\vartheta(\partial_i u)(x+\vartheta\Delta_z g(x,y))\partial_i \Delta_z g(x,y)),
	\de 
	apply Fubini's theorem  and then use the integration by parts formula to get
	\ce 
	&&B_{12}\\
	&=&\sup_\psi\int_{\mR^d}\int_{|y|<R}\int_{[0,1]}\bigg[(\Delta_z g(x,y))^i \bigg((\partial_{i}u)(x+\vartheta\Delta_z g(x,y))-\partial_{i}u(x)\bigg)\bigg]d\vartheta\upsilon(dy)\psi(x)dx\\
	&=&\sup_\psi\int_{|y|<R}\int_{[0,1]}\int_{\mR^d}\bigg[(\Delta_z g(x,y))^i \bigg((\partial_{i}u)(x+\vartheta\Delta_z g(x,y))-\partial_{i}u(x)\bigg)\psi(x)\bigg]dxd\vartheta\upsilon(dy)\\
	&\leq&\sup_\psi\int_{|y|<R}\int_{[0,1]}\int_{\mR^d}\bigg[\partial_i(\Delta_z g(x,y))^i \bigg(u(x+\vartheta\Delta_z g(x,y))-u(x)\bigg)\psi(x)\bigg]dxd\vartheta\upsilon(dy)\\
	&&+\sup_\psi\int_{|y|<R}\int_{[0,1]}\int_{\mR^d}\bigg[\vartheta(\Delta_z g(x,y))^i\partial_i(\Delta_z g(x,y))^i (\partial_{i}u)(x+\vartheta\Delta_z g(x,y))\psi(x)\bigg]dxd\vartheta\upsilon(dy)\\
	&&+\sup_\psi\int_{|y|<R}\int_{[0,1]}\int_{\mR^d}\bigg[(\Delta_z g(x,y))^i \bigg(u(x+\vartheta\Delta_z g(x,y))-u(x)\bigg)\partial_{i}\psi(x)\bigg]dxd\vartheta\upsilon(dy)\\
	&:=&B_{121}+B_{122}+B_{123}.
	\de 
	Applying Fubini's theorem, Lemma \ref{lemma0+A.5} and Lemma \ref{lemma1+A.6} with $r=\varrho$ to the term $B_{121}$, we obtain 
	\ce 
	B_{121}
	&\leq&\sup_\psi\int_{\mR^d}\int_{|y|<R}|\partial_i(\Delta g(x,z))^i\Delta g(x,z))| \sup_{w\neq 0}|w|^{-1}|u(x+w)-u(x)|\upsilon(dy)\psi(x)dx\\
	&\leq &\|\varGamma_{1,2}^{0,R}g\|_{\mL^{\infty}([0,1])}^{\frac{1}{2}}\|(\varGamma_{0,2}^{0,R}\Delta_z g(x,y))^{\frac{1}{2}} \sup_{w\neq 0}|w|^{-1}|u(x+w)-u(x)|\|_{p}\\
	&\lesssim &\|\varGamma_{1,2}^{0,R}g\|_{\mL^{\infty}([0,1])}^{\frac{1}{2}}\|(\varGamma_{0,2}^{0,R}\Delta_z g(x,y))^{\frac{1}{2}}\varphi^z \sup_{w\neq 0}|w|^{-1}|u(x+w)-u(x)|\|_{p}\\
	&\lesssim &\varrho^{\frac{1}{2}}(\|\varGamma_{0,2}^{0,R}g\|_{\mL^{\infty}([0,1])}\|\varGamma_{1,2}^{0,R}g\|_{\mL^{\infty}([0,1])})^{\frac{1}{2}}\|\varphi^z \sup_{w\neq 0}|w|^{-1}|u(x+w)-u(x)|\|_{p}\\
	&\lesssim &\varrho^{\frac{1}{2}}(\|\varGamma_{0,2}^{0,R}g\|_{\mL^{\infty}([0,1])}\|\varGamma_{1,2}^{0,R}g\|_{\mL^{\infty}([0,1])})^{\frac{1}{2}}\|\sup_{w\neq 0}|w|^{-1}|u(x+w)-u(x)|\|_{p}.
	\de 
	From  Lemma \ref{lemma3.2}, we get  
	\ce 
    B_{121}	&\lesssim &\varrho^{\frac{1}{2}}\|u\|_{\mH_{1,p}(\mR^d)}.
	\de 
	For the term $B_{122}$, by direct calculation we obtain
	\ce 
	B_{122}&=&\sup_\psi\int_{\mR^d}\int_{|y|<R}\partial_i(\Delta_z g(x,y))^i\int_{[0,1]}\vartheta(\Delta_z g(x,y))^i \partial_{i}u(x+\vartheta\Delta_z g(x,y))d\vartheta\upsilon(dy)\psi(x)dx\\
	&=&\sup_\psi\int_{\mR^d}\int_{|y|<R}\partial_i(\Delta_z g(x,y))^i\bigg(u(x+\Delta_z g(x,y))-\int_{[0,1]} u(x+\vartheta\Delta_z g(x,y))d\vartheta\bigg)\upsilon(dy)\psi(x)dx\\
	&\leq &\sup_\psi\int_{\mR^d}\int_{|y|<R}\partial_i(\Delta_z g(x,y))^i(u(x+\Delta_z g(x,y))-u(x))\upsilon(dy)\psi(x)dx\\
	&&+\sup_\psi\int_{\mR^d}\int_{|y|<R}\partial_i(\Delta_z g(x,y))^i\int_{[0,1]} (u(x+\vartheta\Delta_z g(x,y))-u(x))d\vartheta\upsilon(dy)\psi(x)dx\\
	&\leq &\sup_\psi\int_{\mR^d}\int_{|y|<R}|\partial_i(\Delta_z g(x,y))^i\Delta_z g(x,y)| \sup_{w\neq 0}|w|^{-1}|u(x+w)-u(x)|\upsilon(dy)\psi(x)dx.
	\de 
	Then apply H\"{o}lder's inequality, Lemma \ref{lemma0+A.5}, Lemma  \ref{lemma1+A.6} and (\ref{imparticular-estimate-1}) in Lemma \ref{lemma3.2} to yield 
	\ce 
    B_{122}	&\lesssim &\|\varGamma_{1,2}^{0,R}g\|_{\mL^{\infty}([0,1])}^{\frac{1}{2}}\|(\varGamma_{0,2}^{0,R}\Delta_z g(x,y))^{\frac{1}{2}}\varphi^z \sup_{w\neq 0}|w|^{-1}|u(x+w)-u(x)|\|_{p}\\
    &\lesssim &\varrho^{\frac{1}{2}}\|\varGamma_{1,2}^{0,R}g\|_{\mL^{\infty}([0,1])}\|\sup_{w\neq 0}|w|^{-1}|u(x+w)-u(x)|\|_{p}\\
    &\lesssim &\varrho^{\frac{1}{2}}\|\varGamma_{1,2}^{0,R}g\|_{\mL^{\infty}([0,1])}\|u\|_{\mH_{1,p}(\mR^d)}.
	\de 
	For the term $B_{123}$, from Lemma \ref{lemma0+A.5} , Lemma Lemma  \ref{lemma1+A.6} and (\ref{imparticular-estimate-1}) in Lemma \ref{lemma3.2}, we get 
	\ce 
	B_{123}
	&= &\|\int_{|y|<R}\int_{[0,1]}(\Delta g(x,z))^i (u(x+\vartheta\Delta g(x,z))-u(x))d\vartheta\upsilon(dy)\|_{p}\\
	&\leq&\|\varGamma_{0,2}^{0,R}\Delta_z g(x,y)\sup_{w\neq 0}|w|^{-1}|u(x+w)-u(x)|\|_{p}\\
	&\lesssim&\|\varGamma_{0,2}^{0,R}\Delta_z g(x,y)\varphi^z\sup_{w\neq 0}|w|^{-1}|u(x+w)-u(x)|\|_{p}\\
	&\lesssim&\varrho\|\varGamma_{1,2}^{0,R}g\|_{\mL^{\infty}([0,1])}\|u\|_{\mH_{1,p}(\mR^d)}.
	\de 
	The term $B_2$, we obtain by direct calculation 
	\ce 
	&&(B_2)^{1/p}\\
	&=&\|\int_{\varepsilon<|y|<R}[u(\cdot+g^z(y))-u(x)-g^z(y)\cdot\nabla u]\upsilon(dy)\|_{\mH_{-1,p}(\mR^d)}^p\\
	&=&\sup_\psi\int_{\mR^d}\int_{\varepsilon<|y|<R}[u(x+g^z(y))-u(x)-g^z(y)\cdot\nabla u(x)]\upsilon(dy)\psi(x) dx\\
	&=&\sup_\psi\int_{\mR^d}\int_{\varepsilon<|y|<R}\int_{[0,1]}( g^z(y))^i[\partial_iu(x+\vartheta g^z(y))-\partial_i u(x)]d \vartheta\upsilon(dy)\psi(x) dx.
	\de 
	Then apply Fubini's theorem and the integration by parts formula to get  
	\ce 
	(B_2)^{1/p}&=&\sup_\psi\int_{\varepsilon<|y|<R}\int_{[0,1]}\int_{\mR^d}( g^z(y))^i[\partial_iu(x+\vartheta g^z(y))-\partial_i u(x)]\psi(x) dxd \vartheta\upsilon(dy)\\
	&=&\sup_\psi\int_{\varepsilon<|y|<R}\int_{[0,1]}\int_{\mR^d}( g^z(y))^i[u(x+\vartheta g^z(y))- u(x)]\partial_i\psi(x) dxd \vartheta\upsilon(dy)\\
	&\leq&\int_{\varepsilon<|y|<R}|g^z(y))^i|\int_{[0,1]}\sup_\psi\int_{\mR^d}( [u(x+\vartheta g^z(y))- u(x)]\partial_i\psi(x) dxd \vartheta\upsilon(dy)\\
	&\leq&\int_{\varepsilon<|y|<R}|g^z(y))|\upsilon(dy)\|u\|_p.
	\de 
	From H\"{o}lder's inquality, (\ref{formula1+A.5.+1}) and the above inequality, we have 
	\ce 
	(B_2)^{1/p}\leq (\|\varGamma_{0,2}^{0,R}g\|_{\mL^{\infty}([0,1])}^{\frac{1}{2}}\upsilon(\{y:\varepsilon<|y|<R\}))^{\frac{1}{2}}\|u\|_p.
	\de 
	For the term $B_3$, applying Fubini's theorem and the integration by parts formula, we obtain
	\ce 
   &&(B_3)^{1/p}\\
    &=&\sup_\psi\int_{\mR^d}\int_{|y|<\varepsilon}[u(x+g^z(y))-u(x)-g^z(y)\cdot\nabla u(x)]\upsilon(dy)\psi(x) dx\\
    &=&\sup_\psi\int_{\mR^d}\int_{|y|<\varepsilon}\int_{[0,1]}\vartheta(g^z(y))^i(g^z(y))^j\partial_{ij}u(x+\vartheta g^z(y))d\vartheta \upsilon(dy)\psi(x) dx\\
    &=&\sup_\psi\int_{|y|<\varepsilon}\int_{[0,1]}\int_{[0,1]}\int_{\mR^d}\vartheta(g^z(y))^i(g^z(y))^j\partial_{ij}u(x+\theta g^z(y))\psi(x)dxd\theta d\vartheta \upsilon(dy) \\
    &=&\sup_\psi\int_{|y|<\varepsilon}\int_{[0,1]}\int_{[0,1]}\int_{\mR^d}\vartheta(g^z(y))^i(g^z(y))^j\partial_{ij}u(x+\theta g^z(y))\psi(x)dx d\theta d\vartheta \upsilon(dy)\\
    &\leq &\sup_\psi\int_{|y|<\varepsilon}|g^z(y)|^2\int_{[0,1]} \int_{[0,1]} \vartheta\sup_\psi\int_{\mR^d}\partial_{j}u(x+\theta g^z(y))\partial_{i}\psi(x)dx d\theta d\vartheta \upsilon(dy)\\
    &\leq &\varGamma_{0,2}^{0,\varepsilon}g^z\|u\|_{\mH_{1,p}(\mR^d)}.
    \de 
 From (\ref{formula1+A.5.+1}), the above inequality implies 
 \ce
(B_3)^{1/p}\leq \|\varGamma_{0,2}^{0,\varepsilon}g\|_{\mL^{\infty}([0,1])}\|u\|_{\mH_{1,p}(\mR^d)}.
 \de 

By combining the estimates for $B_1$, $B_2$  and $B_3$, we obtain the following result for any $\varepsilon\in (0,R)$ and $\varrho>0$ 
	\be\label{formula_estimate_l} 
	&&\bigg(\int_{\mR^d}\|\mathscr{L}_{\upsilon,R}^{g}u\phi^z\|_{\mH_{-1,p}(\mR^d)}^pdz\bigg)^{\frac{1}{p}}\no\\
	&\lesssim& (\varrho^{\frac{1}{4}}+\varrho^{\frac{1}{2}}+\varrho+\|\varGamma_{0,2}^{0,\varepsilon}g\|_{\mL^{\infty}([0,1])})\|u\|_{\mH_{1,p}(\mR^d)}+(\upsilon(\{y:\varepsilon<|y|<R\}))^{\frac{1}{2}}\|u\|_p\no\\
	&\lesssim&(o_\varepsilon(1)+o_\varrho(1))\|u\|_{\mH_{1,p}(\mR^d)}+C_\varepsilon\|u\|_p,
	\ee 
	where $C_\varepsilon\approx (\upsilon(\{\varepsilon<|y|<R\}))^{\frac{1}{2}}$, $o_\varepsilon(1)\approx\|\varGamma_{0,2}^{0,\varepsilon}g\|_{\mL^{\infty}([0,1])}$ and $o_\varrho(1)\approx \varrho^{\frac{1}{4}}+\varrho^{\frac{1}{2}}+\varrho$. Combining the above estimates, we obtain (\ref{formula0+A.4}).

	\textbf{Step 2.} We show that for every integer $1\leq n\leq n_0$ and every $s\in [0,1]$,  the following estimate holds
	\be\label{formulaA.10+02} 
	\|u\|_{\mH_{1,p}^{np}([s,1])}\lesssim \|f\|_{\mH_{-1,p}^{np}([s,1])}+\|u\|_{\mH_{-1,p}^{np}([s,1])}.
	\ee 
	Since 
	\be\label{formulaA.10+01}
	\|u\|_{\mH_{1,p}^{np}([s,1])}\approx \|u\|_{\mH_{-1,p}^{np}([s,1])}+\|\nabla^2u\|_{\mH_{-1,p}^{np}([s,1])},
	\ee
	so, it suffices to estimate the term $\|\nabla^2u\|_{\mH_{-1,p}^{np}([s,1])}$. Apply Lemma \ref{lemma0+A.5} to get, 
	\be\label{formulaA.10+00}
	&&\|\nabla^2u\|_{\mH_{-1,p}^{np}([s,1])}^{np}\no\\
	&\lesssim&\int_{s}^1\bigg(\int_{\mR^d}\bigg(\|\nabla^2u_t^z\|_{\mH_{-1,p}^{p}([s,1])}^p+\|u_t\nabla^2\phi^z\|_{\mH_{-1,p}^{p}([s,1])}^p+\|\nabla u_t\cdot \nabla\phi^z \|_{\mH_{-1,p}^{p}([s,1])}^p\bigg)dz\bigg)^{n}dt\no\\
	 &\lesssim&\int_{s}^1\bigg(\int_{\mR^d}\|\nabla^2u_t^z\|_{\mH_{-1,p}^{p}([s,1])}^pdz\bigg)^{n}dt+\int_{s}^1\|u_t\|_{p}^{np}dt.
	\ee 
	By Tonelli's theorem,  we have 
	\be \label{Tonelli_1}
	&&\int_{s}^1\bigg(\int_{\mR^d}\|\nabla^2u_t^z\|_{\mH_{-1,p}^{p}([s,1])}^pdz\bigg)^ndt\no\\
	&=&\int_{s}^1\int_{\mR^{nd}}\Pi_{k=1}^{n}\|\nabla^2u_t^{z_k}\|_{\mH_{-1,p}^{p}([s,1])}^pdz_1\cdots dz_ndt\no\\
	&=&\int_{\mR^{nd}}\int_{s}^1\Pi_{k=1}^{n}\|\nabla^2u_t^{z_k}\|_{\mH_{-1,p}^{p}([s,1])}^pdtdz_1\cdots dz_n.
	\ee 
   Given $z_1,\cdots,z_n\in\mR^d$ and Lemma \ref{lemmaA.9+} and formula (\ref{formula0+A.4}), we have
	\ce 
    &&\int_{\mR^{nd}}\int_{s}^1\Pi_{k=1}^{n}\|\nabla^2u_t^{z_k}\|_{\mH_{-1,p}^{p}([s,1])}^pdtdz_1\cdots dz_n\\
    &\lesssim& \int_{\mR^{nd}}\sum_{i=1}^n\int_{s}^1\|F_t^{z_i}\|_{\mH_{-1,p}^{p}([s,1])}^p\Pi_{k\neq i }^{n}\|\nabla^2u_t^{z_k}\|_{\mH_{-1,p}^{p}([s,1])}^pdtdz_1\cdots dz_n\\
    &=& n\int_{s}^1\int_{\mR^{d}}\|F_t^{z}\|_{\mH_{-1,p}^{p}([s,1])}^pdz\bigg(\int_{\mR^{d}}\|\nabla^2u_t^{z}\|_{\mH_{-1,p}^{p}([s,1])}^pdz\bigg)^{n-1}dt\\
    &\lesssim&\int_{s}^1\bigg(\int_{\mR^{d}}\|\nabla^2u_t^{z}\|_{\mH_{-1,p}^{p}([s,1])}^pdz\bigg)^{n-1}\bigg(\|f_t\|_{\mH_{-1,p}(\mR^d)}+C_\varepsilon\|u_t\|_{p}+(o_\varepsilon(1)+o_\varrho(1))\|u_t\|_{\mH_{1,p}(\mR^d)}\bigg)^pdt
	\de 
  Substituting the above estimate into (\ref{Tonelli_1}) and then applying H\"{o}lder's inequality, we have
	\ce 
	&&\bigg[\int_{s}^1\bigg(\int_{\mR^d}\|\nabla^2u_t^2\|_{\mH_{-1,p}^{p}([s,1])}^pdz\bigg)^ndt\\
	&\lesssim&\int_{s}^1\bigg(\int_{\mR^d}\|\nabla^2u_t^2\|_{\mH_{-1,p}^{p}([s,1])}^pdz\bigg)^{n}dt
	\bigg]^{1-\frac{1}{n}}\\
	&&\times\bigg[\int_{s}^1\big(\|f_t\|_{\mH_{-1,p}(\mR^d)}+C_\varepsilon\|u_t\|_{p}+(o_\varepsilon(1)+o_\varrho(1))\|u_t\|_{\mH_{1,p}(\mR^d)}\big)^{np}dt\bigg]^{\frac{1}{n}},
	\de 
	which yields 
	\ce 
	&&\int_{s}^1\bigg(\int_{\mR^d}\|\nabla^2u_t^2\|_{\mH_{-1,p}^{p}([s,1])}^pdz\bigg)^ndt\\
	&\lesssim& \int_{s}^1\big(\|f_t\|_{\mH_{-1,p}(\mR^d)}^{np}+C_\varepsilon\|u_t\|_{p}^{np}+(o_\varepsilon(1)+o_\varrho(1))\|u_t\|_{\mH_{1,p}(\mR^d)}^{np}\big)dt.
	\de 
	Substituting the above inequality into (\ref{formulaA.10+00}), we have 
	\ce 
	\|\nabla^2u\|_{\mH_{-1,p}^{np}([s,1])}^{np}\leq \|f_t\|_{\mH_{-1,p}^{np}([s,1])}^{np}+C_\varepsilon\|u_t\|_{\mH_{p}^{np}([s,1])}^{np}+(o_\varepsilon(1)+o_\varrho(1))\|u_t\|_{\mH_{1,p}^{np}([s,1])}^{np}.
	\de 
	By interpolation inequallity 
	\be\label{formula-interp-1} 
	\|u\|_{\mH_{p}^{np}([s,1])}\leq C_\varrho\|u\|_{\mH_{-1,p}^{np}([s,1])}+o_\varrho(1)\|u\|_{\mH_{1,p}^{np}([s,1])},
	\ee 
	we obtain
	\ce 
	\|\nabla^2u\|_{\mH_{-1,p}^{np}([s,1])}^{np}\lesssim \|f_t\|_{\mH_{-1,p}^{np}([s,1])}^{np}+(C_\varepsilon+C_\varrho)\|u_t\|_{\mH_{-1,p}^{np}([s,1])}^{np}+(o_\varepsilon(1)+o_\varrho(1))\|u_t\|_{\mH_{1,p}^{np}([s,1])}^{np}.
	\de 
	Substituting the above inequality into (\ref{formulaA.10+01}), we get 
	\ce 
	\|u\|_{\mH_{1,p}^{np}([s,1])}\lesssim \|f_t\|_{\mH_{-1,p}^{np}([s,1])}^{np}+(C_\varepsilon+C_\varrho)\|u_t\|_{\mH_{-1,p}^{np}([s,1])}^{np}+(o_\varepsilon(1)+o_\varrho(1))\|u_t\|_{\mH_{1,p}^{np}([s,1])}^{np}.
	\de 
	By selecting suitably small values for $\rho, \varepsilon$, we can deduce (\ref{formulaA.10+02}) from the aforementioned estimate.
	
	\textbf{Step 3}  In this step, we aim to demonstrate that
	\ce 
	\|u\|_{\mH_{-1,p}^{\infty}([0,1])}\lesssim \|f\|_{\mH_{-1,p}^{p}([0,1])}.
	\de 
	Given that $u^z$ satisfies (\ref{Equ-PDF-z}), we can express $u^z$ as
	\ce 
	u_s^z=\int_{s}^1p_{A_{s,t}(z)}F_t^zdt,~~\mbox{where}~A_{s,t}(z)=2\int_{s}^{t}a(r,z)dr.
	\de 
	Applying Minkowski's inequality alongside \cite[Theorem 5.30]{Triebel2013}, we obtain 
	\ce 
	\|u_s^z\|_{\mH_{-1,p}(\mR^d)}\lesssim \int_{s}^{1}\|F_t^z\|_{\mH_{-1,p}(\mR^d)}dt.
	\de 
    Next, by utilizing Hölder's inequality, Lemma \ref{lemma0+A.5}, (\ref{formula0+A.4}), and the interpolation inequality (\ref{formula-interp-1}), we derive
	\ce 
	\|u_s\|_{\mH_{-1,p}(\mR^d)}^p&\lesssim& \int_{\mR^d}\|u_s^z\|_{\mH_{-1,p}(\mR^d)}^pdz\\
	&\lesssim&\int_{s}^{1}\int_{\mR^d}\|F_t^z\|_{\mH_{-1,p}(\mR^d)}^pdzdt\\
	&\lesssim&\int_{s}^{1}\big[\|f_t\|_{\mH_{-1,p}(\mR^d)}^p+\|u_t\|_{\mH_{-1,p}(\mR^d)}^p+\|u_t\|_{\mH_{1,p}(\mR^d)}^p\big]dt.
	\de
   Using (\ref{formulaA.10+02}) with $n=1$ , we further get
	\ce 
	\|u_s\|_{\mH_{-1,p}(\mR^d)}^p\lesssim \|f\|_{\mH_{-1,p}^p([0,1])}^p+\int_{s}^{1}\|u_t\|_{\mH_{-1,p}(\mR^d)}^pdt,
	\de 
	which, by Gronwall's inequality, implies
	\be\label{formulaA.15.+0}
	\|u\|_{\mH_{-1,p}^{\infty}([0,1])}\lesssim \|f\|_{\mH_{-1,p}^{p}([0,1])}.
	\ee

	\textbf{Step 4} Substituting (\ref{formulaA.15.+0}) into (\ref{formulaA.10+02}), we find 
	\ce 
	\|u\|_{\mH_{1,p}^{np}([s,1])}&\lesssim& \|f\|_{\mH_{-1,p}^{np}([0,1])}+\|f\|_{\mH_{-1,p}^{p}([0,1])}\\
	&\lesssim& \|f\|_{\mH_{-1,p}^{np}([0,1])}.
	\de 
  This further implies
	\be \label{last_estimate_1}
	\|u\|_{\mH_{1,p}^{q}([s,1])}
	\lesssim \|f\|_{\mH_{-1,p}^{p}([0,1])}
	\ee 
	when $q\geq p$. 
  From (\ref{formula_estimate_l}) and (\ref{last_estimate_1}), we can immediately conclude	(\ref{formula0+A.3}) for the case when $q\geq p$.
\end{proof}
\begin{lemma}\label{LemmaA.11.+0}
	Suppose Condition  ($\sH_a^f$) holds and $q^*\geq p^*$. Then there exists a unique solution $v$ to (\ref{EquationA.2+1}) which also fulfills  (\ref{formula1+A.3}).
\end{lemma}
\begin{proof}
	The proof follows similarly to that of Lemma \ref{LemmaA.10.}, with key differences in the estimates within \textbf{Step 1}.  Let $v$ be a solution to Eq. (\ref{EquationA.2+1}). The notations $\phi^z$, $a(z)$ is the same as in the proof of Lemma \ref{LemmaA.10.}. In addition, denote $h^z(t,x)=h(t,x)\phi^z(x)$ and $v^z(t,x)=v(t,x)\phi^z(x)$. It is evident that $v^z$ satisfies the relation
	\ce 
	\partial_t v^z-a^{ij}(z)\partial_{i,j}v^z+H^z=0,~~v^z(0,\cdot)=0,
	\de 
	where 
	\ce 
	H^z=h^z+a^{ij}(z)\partial_{ij}v^z-\partial_{ij}(a^{ij}v)\phi^z+\sL_{\upsilon,R}^gv\phi^z.
	\de 
   We continue using the functions $\varphi$ and $\varphi^z$ defined in the proof of  Lemma \ref{LemmaA.10.}.  Our goal is to demonstrate that for any $0<\varepsilon<R$ and $\varrho>0$,
	\be \label{formulaA.21.+}
	\int_{\mR^d}\|H^z\|_{\mH_{-1,p^*}(\mR^d)}^{p^*}dz\lesssim \|h_t\|_{\mH_{-1,p^*}(\mR^d)}^{p^*}+\|v_t\|_{\mH_{-1,p^*}}^{p^*}+(o_\varepsilon(1)+o_\varrho(1))\|v_t\|_{\mH_{1,p^*}(\mR^d)}^{p^*}.
	\ee 
	
    Through direct computation, we obtain
	\ce 
	&&\|H^z\|_{\mH_{-1,p^*}(\mR^d)}\\
	&\leq &\|\partial_{ij}[(a^{ij}-a^{ij}(z))v^z]\|_{\mH_{-1,p^*}(\mR^d)}+\|h^z+2\partial_{i}(a^{ij}v)\partial_j\phi^z\|_{\mH_{-1,p^*}(\mR^d)}\\
	&&+\|a^{ij}v\partial_{ij}\phi^z\|_{\mH_{-1,p^*}(\mR^d)}+\|\sL_{\upsilon,R}^gv\phi^z\|_{\mH_{-1,p^*}(\mR^d)}\\
	&\lesssim&\|(a^{ij}-a^{ij}(z))v^z\|_{\mH_{1,p^*}(\mR^d)}+\|h^z+2\partial_{i}(a^{ij}v)\partial_j\phi^z\|_{\mH_{-1,p^*}(\mR^d)}\\
	&&+\|a^{ij}v\partial_{ij}\phi^z\|_{\mH_{-1,p^*}(\mR^d)}+\|\sL_{\upsilon,R}^gv\phi^z\|_{\mH_{-1,p^*}(\mR^d)}.
	\de 
	Applying Lemma \ref{lemma0+A.5}, part (III) in Lemma \ref{Multipliction-norm1} and Lemma \ref{lemma)+A.6}, we have 
	\ce
	&&\|(a^{ij}-a^{ij}(z))v^z\|_{\mH_{1,p^*}(\mR^d)}\\
	&\lesssim& \bigg(\|(a^{ij}(z)-a^{ij})\varphi^z\|_{\infty}+\|(a^{ij}(z)-a^{ij})\varphi^z\|_{\mH_{1,p_0}(\mR^d)}\bigg)\|v^z\|_{\mH_{1,p^*}(\mR^d)}\\
	&\lesssim& o_\varrho(1)\|v^z\|_{\mH_{1,p^*}(\mR^d)}.
	\de 
	Applying Lemma \ref{lemma0+A.5} and  part (III) in Lemma \ref{Multipliction-norm1}, we have 
	\ce 
	\int_{\mR^d}\|a^{ij}v\partial_{ij}\phi^z\|_{\mH_{-1,p^*}(\mR^d)}^{p^*}dz&\lesssim& \int_{\mR^d}(\|a^{ij}\|_{\infty}+\|\nabla a^{ij}\|_{p_0})^{p^*}\|v\partial_{ij}\phi^z\|_{\mH_{-1,p^*}(\mR^d)}^{p^*}dz\\
	&\lesssim& \int_{\mR^d}\|v\partial_{ij}\phi^z\|_{\mH_{-1,p^*}(\mR^d)}^{p^*}\\
	&\lesssim& \|v\|_{p^*}^{p^*}.
	\de 
	Applying Lemma \ref{lemma0+A.5}, we find
	\ce 
	&&\int_{\mR^d}\|h^z+2\partial_{i}(a^{ij}v)\partial_j\phi^z\|_{\mH_{-1,p^*}(\mR^d)}^{p^*}dz\\
	&\lesssim& \|h\|_{\mH_{-1,p^*}(\mR^d)}^{p^*}+\|\partial_i(a^{ij}v)\|_{\mH_{-1,p^*}(\mR^d)}^{p^*}\\
	&\lesssim& \|h\|_{\mH_{-1,p^*}(\mR^d)}^{p^*}+\|v\|_{p^*}^{p^*}.
	\de 
	Similar to (\ref{formula-uz}), we have 
	\ce 
	\int_{\mR^d}\|v^z\|_{\mH_{1,p^*}(\mR^d)}^{p^*}dz\lesssim \|v\|_{\mH_{1,p^*}(\mR^d)}^{p^*}+C_\varrho\|v\|_{p^*}^{p^*}.
	\de 
	Similar to	(\ref{formula_estimate_l}), we have 
	\ce 
	\int_{\mR^d}\|\mathscr{L}_{\upsilon,R}^{g}v\phi^z\|_{\mH_{-1,p^*}(\mR^d)}^{p^*}dz\lesssim \bigg(o_\varepsilon(1)+o_\varrho(1)\bigg)\|v\|_{\mH_{1,p^*}(\mR^d)}^{p^*}+C_\varepsilon\|v\|_{p^*}^{p^*}.
	\de 
	Using the interpolation inequality 
	\ce 
	\|v\|_{p^*}\leq C_\varrho\|v\|_{\mH_{-1,p^*}(\mR^d)}+\varrho \|v\|_{\mH_{1,p^*}},
	\de 
	we derive (\ref{formulaA.21.+}) from the previous estimates by selecting $\varepsilon$  and $\varrho$  sufficiently small.
  Finally, follow Steps 2,3,4 of the proof of Lemma \ref{LemmaA.10.} to obtain (\ref{formula1+A.3}) when $q^*\geq p^*$.
\end{proof}

\begin{proof}[Proof of Theorem \ref{TheoremA.4}.] For Eq. (\ref{EquationA.1}), by the method of continuity, it suffices to prove (\ref{formula0+A.3}) when $u$ is a solution to Eq. (\ref{EquationA.1}). The case $q\geq p$ has been proved in Lemma (\ref{LemmaA.10.}). Consider the case $q<p$, which is equivalent to $q^*>p^*$. For each $h\in \mH_{-1,p^*}^{q^*}([0,1])$, let $v$ be the unique solution to (\ref{EquationA.2+1}), as guaranteed by Lemma \ref{LemmaA.11.+0}. By setting $\phi=v$ in (\ref{Equ-Solu-1}) and using Eq. (\ref{EquationA.2+1}) for $v$ to get 
	\be\label{equivalenceR1}
	&&\int_{0}^{1}\langle f_t-\sL_{\upsilon,R}^gu,v_t \rangle_{\mH_{-1,p}(\mR^d)\times\mH_{1,p^*}(\mR^d)}dt\no\\
	&=&\int_{0}^{1}\langle (\partial_t+a^{ij}\partial_{ij})u,v_t \rangle_{\mH_{-1,p}(\mR^d)\times \mH_{1,p^*}(\mR^d)}dt\no\\
	&=&\int_{0}^{1}\langle u_t,h_t+\sL_{\upsilon,R}^gv \rangle_{\mH_{1,p}(\mR^d)\times \mH_{-1,p^*}(\mR^d)}dt.
	\ee 
First, we prove that
	\be \label{equivalenceNorm_1}
	\sup_{\|h\|_{\mH_{-1,p^*}^{q^*}([0,1])}\leq 1}\int_{0}^{1}\langle u_t,h_t+\sL_{\upsilon,R}^gv \rangle_{\mH_{1,p}(\mR^d)\times \mH_{-1,p^*}(\mR^d)}dt=\sup_{\|h\|_{\mH_{-1,p^*}^{q^*}([0,1])}\leq 1}\int_{0}^{1}\langle u_t,h_t \rangle_{\mH_{1,p}(\mR^d)\times \mH_{-1,p^*}(\mR^d)}dt.
	\ee 
	To achieve this, it suffices to show that for any given $\forall h\in \mH_{-1,p^*}^{q^*}(\mR^d)$, the solution $v_1$ to the equation 
	\be\label{commonPDE1} 
	\partial_t v_1-\partial_{ij}(a^{ij}v_1)+h=0,~v(0,\cdot)=0
	\ee 
   is also a solution to the equation
	\be\label{jumPDE1} 
	\partial_t v-\partial_{ij}(a^{ij}v)+\sL_{\upsilon,R}^gv+h_0=0,~v(0,\cdot)=0
	\ee 
	with another function $h_0\in \mH_{-1,p^*}^{q^*}(\mR^d)$.
	
	Taking $\forall h\in \mH_{-1,p^*}^{q^*}(\mR^d)$, from Theorem A.4 \cite{Ling2022}, we know that Eq. (\ref{commonPDE1}) has a unique solution satisfying 
	\be\label{no_jump_PDE1}
	\|v_1\|_{\mH_{1,p^*}^{q^*}(\mR^d)}+\|\partial_t v_1\|_{\mH_{-1,p^*}^{q^*}(\mR^d)}\lesssim \|h\|_{\mH_{-1,p^*}^{q^*}(\mR^d)}.
	\ee 
	Let $h_0=h-\sL_{\upsilon,R}^gv_1$. By proving similarly to (\ref{formula_estimate_l}), we get 
	\ce 
	\|\sL_{\upsilon,R}^gv_1\|_{\mH_{-1,p^*}^{q^*}(\mR^d)}\lesssim \|v_1\|_{\mH_{1,p^*}^{q^*}(\mR^d)}.
	\de 
    From (\ref{no_jump_PDE1}), we see that
	\ce 
	\|\sL_{\upsilon,R}^gv_1\|_{\mH_{-1,p^*}^{q^*}(\mR^d)}\lesssim  \|h\|_{\mH_{-1,p^*}^{q^*}(\mR^d)}.
	\de 
	Then we get the estimate of $h_0$ as follows
	\ce 
	\|h_0\|_{\mH_{-1,p^*}^{q^*}(\mR^d)}&=&\|h-\sL_{\upsilon,R}^gv_1\|_{\mH_{-1,p^*}^{q^*}(\mR^d)}\\
	&\leq &\|h\|_{\mH_{-1,p^*}^{q^*}(\mR^d)}+\|\sL_{\upsilon,R}^gv_1\|_{\mH_{-1,p^*}^{q^*}(\mR^d)}\\
	&\lesssim &\|h\|_{\mH_{-1,p^*}^{q^*}(\mR^d)}+\|v_1\|_{\mH_{1,p^*}^{q^*}(\mR^d)}\\
	&\lesssim &\|h\|_{\mH_{-1,p^*}^{q^*}(\mR^d)}.
	\de 
	Now we can see that $v_1$ satisfies Eq. (\ref{jumPDE1}) since $h_0\in \mH_{-1,p^*}^{q^*}(\mR^d)$ and $h=h_0+\sL_{\upsilon,R}^gv$.
	
	Similarly, we can prove that
	\be \label{equivalenceNorm_2}
	\sup_{\|f\|_{\mH_{-1,p^*}^{q^*}(\mR^d)\leq 1}}\int_{0}^{1}\langle f_t-\sL_{\upsilon,R}^gu,v_t \rangle_{\mH_{-1,p}(\mR^d)\times\mH_{1,p^*}(\mR^d)}dt=\sup_{\|f\|_{\mH_{-1,p^*}^{q^*}(\mR^d)\leq 1}}\int_{0}^{1}\langle f_t,v_t \rangle_{\mH_{-1,p}(\mR^d)\times\mH_{1,p^*}(\mR^d)}dt.
	\ee 
	
	Next, let us prove (\ref{formula0+A.3}). It suffices to show that for any $h\in \mH_{-1,p^*}^{q^*}(\mR^d)$, we have
		\ce 
	\sup_{\|h\|_{ \mH_{-1,p^*}^{q^*}(\mR^d)\leq 1}}\int_{0}^{1}\langle u_t,h_t\rangle_{\mH_{1,p}(\mR^d)\times \mH_{-1,p^*}(\mR^d)}dt\leq\|f\|_{\mH_{-1,p}^q(\mR^d)}\|h\|_{\mH_{-1,p^*}^{q^*}(\mR^d)}.
	\de 
	Applying H\"{o}lder inequality and Lemma \ref{LemmaA.11.+0},  we  see that 
	\ce 
	&&|\int_{0}^{1}\langle f_t,v_t \rangle_{\mH_{-1,p}(\mR^d)\times\mH_{1,p^*}(\mR^d)}dt|\\
	&\leq &\|f\|_{\mH_{-1,p}^q(\mR^d)}\|v\|_{\mH_{1,p^*}^{q^*}(\mR^d)}\\
	&\lesssim&\|f\|_{\mH_{-1,p}^q(\mR^d)}\|h\|_{\mH_{-1,p^*}^{q^*}(\mR^d)}.
	\de 	
  From (\ref{equivalenceR1}) and (\ref{equivalenceNorm_2}),	we have 
	\ce 
	&&\int_{0}^{1}\langle u_t,h_t+\sL_{\upsilon,R}^gv \rangle_{\mH_{1,p}(\mR^d)\times \mH_{-1,p^*}(\mR^d)}dt\\
	&\leq& \sup_{\|f\|_{\in \mH_{-1,p^*}^{q^*}(\mR^d)}\leq 1}\int_{0}^{1}\langle f_t-\sL_{\upsilon,R}^gu,v_t \rangle_{\mH_{-1,p}(\mR^d)\times\mH_{1,p^*}(\mR^d)}dt\\
	&=&\sup_{\|f\|_{\in \mH_{-1,p^*}^{q^*}(\mR^d)}\leq 1}\int_{0}^{1}\langle f_t,v_t \rangle_{\mH_{-1,p}(\mR^d)\times\mH_{1,p^*}(\mR^d)}dt\\
	&\lesssim&\|f\|_{\mH_{-1,p}^q(\mR^d)}\|h\|_{\mH_{-1,p^*}^{q^*}(\mR^d)}.
	\de 
	Then again apply (\ref{equivalenceR1}), (\ref{equivalenceNorm_1}) and the above estimate to get
	\ce 
	&&\sup_{\|h\|_{\in \mH_{-1,p^*}^{q^*}(\mR^d)}\leq 1}\int_{0}^{1}\langle u_t,h_t\rangle_{\mH_{1,p}(\mR^d)\times \mH_{-1,p^*}(\mR^d)}dt\\
	&=&\sup_{\|h\|_{\in \mH_{-1,p^*}^{q^*}(\mR^d)}\leq 1}\int_{0}^{1}\langle u_t,h_t+\sL_{\upsilon,R}^gv \rangle_{\mH_{1,p}(\mR^d)\times \mH_{-1,p^*}(\mR^d)}dt\\
	&\lesssim&\|f\|_{\mH_{-1,p}^q(\mR^d)}\|h\|_{\mH_{-1,p^*}^{q^*}(\mR^d)}.
	\de 
    This implies (\ref{formula0+A.3}). The result for Eq. (\ref{EquationA.2+1}) follows using the similar arguments.
\end{proof}

	\section{Increment Estimates for the Euler Paths}
In order to prove Theorem \ref{M-theorem2.2}  we need some moment estimates for the following functionals of the solutions to  Eq. (\ref{EulerAPP3.81})
	\ce 
	\int_s^th(r,X_r^n)[f(r,X_r^n)-f(r,X_{r_n}^{n})]dr,
	\de 
	where $f,h$ are measurable functions in $\mL_p^q([0,1])$. In typical applications herein, we have the result as follows. 
	\begin{theorem}\label{convergeestimate01}
		Assume that Conditions ($\sH_\sigma^g$) and ($\sH_b$) hold. Let  $\gamma_n(f):=\sup_{r\in D_n}\|f\|_{\mL_\infty^q([r,r+1/n])}$. Let $X^n$ be the solution to  Eq. (\ref{EulerAPP3.81}) and let  $f\in \mL_{p}^q([0,1])$ and $h\in \mH_{1,p}^q([0,1])\cap \mL^{\infty}([0,1])$ for some $q\in [2,\infty)$ and $p\in(2(d/\beta\vee 1),\infty]$ satisfying $\frac{d}{p}+\frac{2}{q}<1$. Then for any $\bar{p}\in(0,p)$, there exists a constant $N=N(d,p,q,\bar{p})$ such that 
		\be \label{formula0+4.34+1}
		&&\|\sup_{t\in[0,1]}|\int_{0}^{t}h(r,\widetilde{X}_r^{n})(f(\widetilde{X}_r^{n})-f(\widetilde{X}_{r_n}^{n}))dr|\|_{L_{\bar{p}(\Omega)}}\no\\
		&\leq &N\bigg\{\|h\|_{\mL^{\infty}([0,1])}\gamma_n(f)(1/n)^{1-\frac{1}{q}}+[(1/n)^{\alpha/2}+(1/n)^{ 1-\beta/2}+(1/n)^{1/2}\log(n)]T(f,h)\bigg\},
		\ee
		where $$T(f,h)=(\|h\|_{\mL^{\infty}([0,1])}+\|h\|_{\mH_{1,p}^{q}([0,1])})\|f\|_{\mL_{p}^{q}([0,1])}.$$
		Particularly 
		\be \label{formula0+4.34+2}
		&&\|\sup_{t\in[0,1]}|\int_{0}^{t}(f(r,\widetilde{X}_r^{n})-f(r,\widetilde{X}_{r_n}^{n}))dr|\|_{L_{\bar{p}(\Omega)}}\no\\
		&\leq &N\bigg\{\gamma_n(f)(1/n)^{1-\frac{1}{q}}+[(1/n)^{\alpha/2}+(1/n)^{ 1-\beta/2}+(1/n)^{1/2}\log(n)]\|f\|_{\mL_{p}^{q}([0,1])}\bigg\}.
		\ee	
	\end{theorem}
	The rest of the current section is devoted for the proof of Theorem \ref{convergeestimate01}. First we derive some analytic estimates on the transition operators associated the discrete Euler-Maruyama scheme without drift. By means of the stochastic Davie-Gronwall lemma as Lemma \ref{DGlemma1} and Girsanov theorem, these analytic estimates are utilized to get the desired moment bound. 
	
	\subsection{Analysis of the discrete paths}
	This section is a continuation of Section 5, and we will adhere to the notations introduced in Section 5 unless otherwise specified.
   \begin{lemma} 
	Let Condition ($\sH_\sigma^g$) holds and let $f\in L_{p}(\mR^d)$.
	
	(i) Suppose $p\in [1,\infty]$ and  $r<t_n$. Then 
	\be\label{formula1+11}
	\|T_{r,t}f-T_{r,t_n}f\|_{p}\leq N\frac{1}{n}(t-r)^{-1}\|f\|_{p},
	\ee
	where the constant $N$ depends only on $d,p,c_0,C_1$.
	
	(ii) Suppose $p\in (2d/\beta,\infty]$. Then 
	\be\label{formula1+12+}
	\|F_{r,t}^nf-F_{r,t_n}^nf\|_{p}\leq N\frac{1}{n}\bigg[(t-r)^{\frac{\alpha}{2}-2}+(t-r)^{-\frac{1+\beta}{2}}\bigg]\|f\|_{p},
	\ee
	where the constant $N$ depends only on $d,p,c_0,C_1$.
\end{lemma}

\begin{proof}
	We have established the existence of a constant $C>0$ such that for every $x,y\in\mR^d$, the following inequalities hold:
	\ce 
	C^{-1}(t-r_n)\leq \int_{r}^{t}a_u(y)du+\int_{r_n}^{r}a_u(x)du\leq C(t-r_n).
	\de 
	Given the condition $1/2\leq (t_n-r_n)/(t-r_n)\leq 1$, it becomes straightforward to confirm the hypothesis of Lemma \ref{lemma1+4.2} for the choices  $B=A_{r,t}(y):=\frac{1}{2}\int_r^t a_\tau(y)d\tau$ with $a_\tau(y)=(\sigma\sigma^T)_t(y)$, $\widetilde{B}=A_{r,t_n}(y)$ and
	\ce 
	\|I-B\widetilde{B}^{-1}\|\lesssim (t-t_n)(t-r_n)^{-1}\leq \frac{1}{n}(t-r_n)^{-1}.
	\de 
	
	(i) By applying Lemma \ref{lemma1+4.2}, we obtain
	\ce 
	|T_{r,t}f(x)-T_{r,r_n}f(x)|\lesssim \frac{1}{n}(t-r)^{-1}\int_{\mR^d}p_{M(t-r)}(y)|f(y-x)|dy
	\de 
	for some universal constant $M$. We then utilize the Minkowski inequality to derive (\ref{formula1+11}). 
	
	(ii) Using (\ref{formula1+1+7}) from Lemma \ref{lemma1+4.1}, the second inequality (\ref{formula1+12+}) of this lemma  can be obtained in a similar manner as (\ref{formula1+11}).
\end{proof}
\begin{corollary}\label{corollary1+4.7}
	Assuming that Condition ($\sH_\sigma^g$)  holds, let $f$ be a function $f\in L_{p}(\mR^d)$. Then for any $s\in D_n$, $t\in (s+2/n,1]$ and  $p\in (2d/\beta,\infty)$, we have
	\be \label{differenceQ1}
	\|Q_{s,t}^nf-Q_{s,t_n}^nf\|_{p}\leq N\bigg(\frac{1}{n}(t-s)^{-1}+(1/n)^{\frac{\alpha}{2}}+(1/n)^{1-\frac{\beta}{2}}\bigg)\|f\|_{p},
	\ee 
	where the constant $N$ depends only on $d,p,c_0,c_1$. 
\end{corollary}
\begin{proof}
	By approximation, without loss of generality, let's assume that $f$ is bounded and uniformly continuous. From Lemma \ref{lemma4.4+},
	\ce 
	Q_{s,t}^nf-Q_{s,t_n}^nf=I_1+I_2+I_3,
	\de 
	where 
	\ce 
	I_1&=&T_{s,t}f(x)-T_{s,t_n}f(x),\\
	I_2&=&\int_{s}^{t_n}Q_{s,r_n}^n(F_{r,t}^nf)(x)-Q_{s,t_n}^n(F_{r,t_n}^nf)(x)dr,\\
	I_3&=&\int_{t_n}^{t}Q_{s,r_n}^n(F_{r,t}^nf)(x)dr.
	\de 
	By the fact $s<t-1/n$ implying $s<t_n$, apply (\ref{formula1+11}) to get 
	\ce 
	\|I_1\|_{p}\lesssim \frac{1}{n}(t-s)^{-1}\|f\|_{p}.
	\de 
	From Theorem \ref{theorem4.5}, we have 
	\ce 
	\|Q_{s,r_n}^nf\|_{p}\lesssim \|f\|_{p}
	\de 
	for every $r>s$. It follows that 
	\ce 
	\|I_2\|_{p}\lesssim \int_{s}^{t_n}\|(F_{r,t}^nf)(x)-(F_{r,t_n}^nf)(x)\|_{p}dr.
	\de 
		From Lemma \ref{integrale-Lemma3.6} and (\ref{formula1+12+}), by noting that $u-s\geq 1/n$  we have 
	\ce 
	\|I_2\|_{p}&\lesssim& \int_{s}^{t_n}\frac{1}{n}\bigg((t-r_n)^{\frac{\alpha}{2}-2}+(t-r_n)^{-\frac{1+\beta}{2}}\bigg)dr\|f\|_{p}\\
	&\lesssim& \bigg((1/n)^{\frac{\alpha}{2}}+(1/n)^{\frac{3-\beta}{2}}\bigg)\|f\|_{p}\\
	&\lesssim&\bigg((1/n)^{\frac{\alpha}{2}}+(1/n)^{1-\frac{\beta}{2}}\|f\|_{p}.
	\de
	From Theorem \ref{theorem4.5} and (\ref{lemma4.3+2}), we have 
	\ce 
	\|I_3\|_{p}&\lesssim& \int_{t_n}^{t}\|F_{r,t}^nf\|_{p}dr\\
	&\lesssim&\int_{t_n}^{t}\bigg((t-r_n)^{\frac{\alpha}{2}-1}+(t-r_n)^{-\frac{\beta}{2}}\bigg)\|f\|_{p}dr\\
	&\lesssim&\bigg((1/n)^{\frac{\alpha}{2}}+(1/n)^{1-\frac{\beta}{2}}\bigg)\|f\|_{p}.
	\de 
	Combining the previous estimates $I_1, I_2, I_3$, we obtain the result (\ref{differenceQ1}).
\end{proof}
\begin{corollary}\label{corollary4.8}
	Let $h$ be a function in $\mH_{1,p}(\mR^d)$ for some $p\in (d/\beta\vee 1,\infty)$. Then we have
	\be\label{differenceQ_h1}
	\|Q_{s,t}^nh(x)-h\|_{p}\lesssim \|h\|_{\mH_{1,p}(\mR^d)}[(t-s)^{\frac{\alpha}{2}}+(t-s)^{1-\frac{\beta}{2}}] .
	\ee
\end{corollary}
\begin{proof}
	By approximation, without loss of generality, let's assume that $h$ is continuously differentiable and has bounded derivatives. We rewrite $Q_{s,t}^nh-h$ as 
	\ce 
	Q_{s,t}^nh-h=I_1+I_2,
	\de 
	where 
	\ce 
	I_1&=&T_{s,t}h-h,\\
	I_2&=&\int_{s}^{t}Q_{s,r_n}^n(F_{r,t}^nh)(x)dr.
	\de 
	From the defination of the operator $T$, we have 
	\ce 
	I_1&=&\int_{\mR^d}[p_{A_{s,t}(y)}(y-x)h(y)-p_{A_{s,t}(x)}(y-x)h(x))]dy\\
	&=:&I_{11}+I_{12},
	\de 
	where 
	\ce 
	I_{11}&=&\int_{\mR^d}[(p_{A_{s,t}(y)}(y-x)-p_{A_{s,t}(x)}(y-x))h(y))]dy,\\
	I_{12}&=&\int_{\mR^d}p_{A_{s,t}(x)}(y-x)(h(y)-h(x))dy,
	\de 
	and	$A_{r,t}(y):=\frac{1}{2}\int_r^t a_\tau(y)d\tau$, $a_\tau(y)=(\sigma\sigma^T)_t(y)$. Applying Condition ($\sH_\sigma^g$) (i),  it is straightforward to verify that 
	\ce 
	\lambda^{-1}I\leq A_{s,t}(y)(A_{s,t}(y))^{-1}\leq \lambda I
	\de 
	and 
	\ce 
	\|I- A_{s,t}(x)(A_{s,t}(y))^{-1}\|\leq \lambda |x-y|^{\alpha}
	\de 
	for some finite constant $\lambda$. Then, by Lemma \ref{lemma1+4.2}, we obtain 
	\ce 
	|I_{11}|\lesssim \int_{\mR^d}|x-y|^{\alpha}p_{\lambda(t-s)}(x-y)|h(y)|dy.
	\de 
	Using the Minkowshi inequality, we have 
	\ce 
	\|I_{11}\|_{p}\lesssim (t-s)^{\frac{\alpha}{2}}\|h\|_{p}.
	\de 
	From the Hardy-Littlewood maximal inequality, there is a non-negative function $k\in L_{p}(\mR^d)$ such that $\|k\|_{p}\leq \|\nabla h\|_{p}$ and 
	\ce 
	|h(y)-h(x)|\leq |x-y|(k(x)+k(y)),~~a.e.~x,y\in\mR^d.
	\de  
	Using ellipticity of $\sigma$ and above estimate, we deduce that 
	\ce 
	|I_{12}|&\leq &\int_{\mR^d}p_{A_{s,t}(x)}(y-x)|x-y|(k(x)+k(y))dy\\
	&\lesssim& k(x)(t-s)^{\frac{1}{2}}+\int_{\mR^d}p_{\lambda(t-s)}(y-x)|x-y|k(y)dy.
	\de 
	Again applying the Minkowshi inequality, we have 
	\ce 
	\|I_{12}\|_p\lesssim (t-s)^{\frac{1}{2}}\|k\|_{p}\lesssim (t-s)^{\frac{1}{2}}\|\nabla h\|_{p}.
	\de 
	Using Theorem \ref{theorem4.5} and Lemma \ref{lemma4.3+}, we get 
	\ce 
	\|I_2\|_p&\lesssim& \int_{s}^{t}\|F_{r,t}^nf\|_pdr\\
	&\lesssim& \int_{s}^{t}[(t-r)^{\frac{\alpha}{2}-1}+(t-r)^{-\frac{\beta}{2}}]\|f\|_pdr\\
	&\lesssim& [(t-s)^{\frac{\alpha}{2}}+(t-s)^{1-\frac{\beta}{2}}]\|f\|_p.
	\de 
	From the estimates $I_{1}$, $I_{12}$ and $I_2$, the result (\ref{differenceQ_h1}) is obtained. 
\end{proof}
\subsection{Moment estimates} 
By observing the conditional expectation
\ce 
\mE[f(\widetilde{X}_t^{n})|\sF_s]=Q_{s,t}^nf(\widetilde{X}_s^{n}),
\de 
for $s\in D_n$ and  $f$ being bounded measurable, we can derive the following estimates.
\begin{proposition}\label{proposition4.9}
	Let  $\widetilde{X}^n$ satisfy (\ref{formula1+12+1}).
	 Let $f$ be a function in $L_{p}(\mR^d)$, and let $h$ be a function in $\mH_{1,p}(\mR^d)\cap L_{\infty}(\mR^d)$ for some $p\in(2d/\beta\vee 2,\infty)$. Then, for every $r,s,u\in [0,1]$ such that $r-u\geq 2/n$, $r-s_n>3/n$ and $s-u\geq 2/n$, we have 
	 	\be\label{formula1+14}
	 && \|\mE_s(f(\widetilde{X}_r^{n})-f(\widetilde{X}_{r_n}^{n}))\|_{L_p(\Omega|\sF_u)}\no\\
	 &\leq& N(s-u)^{-\frac{d}{2p}}\bigg(\frac{1}{n}(r-s-2/n)^{-1}+(1/n)^{\frac{\alpha}{2}}+(1/n)^{1-\frac{\beta}{2}}\bigg)\|f\|_{p}
	 \ee  
	 and  
		\be \label{formula1+15}
		&&\|\mE_sh(\widetilde{X}_r^{n})(f(\widetilde{X}_r^{n})-f(\widetilde{X}_{r_n}^{n}))\|_{L_p(\Omega|\sF_u)}\no\\
		&\leq &N \|f\|_{p}(s-u)^{-\frac{d}{2p\bar{p}_0}}\bigg[(r-s-2/n)^{-\frac{d}{2p}}\bigg((1/n)^{\frac{\alpha}{2}}+(1/n)^{1-\frac{\beta}{2}}\bigg)\|h\|_{\mH_{1,p}(\mR^d)}\no\\
		&& +\bigg(\frac{1}{n}(r-s-2/n)^{-1}+(1/n)^{\frac{\alpha}{2}}+(1/n)^{1-\frac{\beta}{2}}\bigg)\|h\|_{\infty}\bigg].
		\ee 
\end{proposition}
\begin{proof}
 Let $\widetilde{s}=s_n+1/n$. Since $\widetilde{s}-u\geq 2/n$ and $r-\widetilde{s}\geq 2/n$, we can apply (\ref{formula1+130}) with $\rho=p$ and Corollary \ref{corollary1+4.7} to conclude that 
	\be \label{conditionExpextion1}
	&&\|\mE_{\widetilde{s}}[f(\widetilde{X}_r^{n})-f(\widetilde{X}_{r_n}^{n})]\|_{L_p(\Omega|\sF_u)}\no\\
	&=&\|Q_{\widetilde{s},r}^nf(\widetilde{X}_{\widetilde{s}}^{n})-Q_{\widetilde{s},r_n}^nf(\widetilde{X}_{\widetilde{s}}^{n})\|_{L_p(\Omega|\sF_u)}\no\\
	&\lesssim&(\widetilde{s}-u)^{-\frac{d}{2p}}\|Q_{\widetilde{s},r}^nf(\widetilde{X}_{\widetilde{s}}^{n})-Q_{\widetilde{s},r_n}^nf(\widetilde{X}_{\widetilde{s}}^{n})\|_{p}\no\\
	&\lesssim&(\widetilde{s}-u)^{-\frac{d}{2p}}\bigg(\frac{1}{n}(r-\widetilde{s})^{-1}+(1/n)^{\frac{\alpha}{2}}+(1/n)^{1-\frac{\beta}{2}}\bigg)\|f\|_{p}.
	\ee 
	Since $\widetilde{s}-u\geq s-u$ and $r-\widetilde{s}\geq r-s-2/n$, we apply the following estimate
	\ce 
	&&\|\mE_{s}[f(\widetilde{X}_r^{n})-f(\widetilde{X}_{r_n}^{n})]\|_{L_p(\Omega|\sF_u)}\\
	&=&\|\mE_{s}[\mE_{\widetilde{s}}[f(\widetilde{X}_r^{n})-f(\widetilde{X}_{r_n}^{n})]]\|_{L_p(\Omega|\sF_u)}\\
	&\leq&\mE_{s}[\|\mE_{\widetilde{s}}[f(\widetilde{X}_r^{n})-f(\widetilde{X}_{r_n}^{n})]\|_{L_p(\Omega|\sF_u)}]\\ &=&\|\mE_{\widetilde{s}}[f(\widetilde{X}_r^{n})-f(\widetilde{X}_{r_n}^{n})]\|_{L_p(\Omega|\sF_u)},
	\de 
	and previous estimates (\ref{conditionExpextion1}) to obtain (\ref{formula1+14}).
	
	Now, we show (\ref{formula1+15}). By the triangle inequality, we have 
	\ce 
	&&\|\mE_s[h(\widetilde{X}_r^{n})(f(\widetilde{X}_r^{n})-f(\widetilde{X}_{r_n}^{n}))]\|_{L_p(\Omega|\sF_u)}\\
	&\leq&\|\mE_s[(hf)(\widetilde{X}_r^{n})-(hf)(\widetilde{X}_{r_n}^{n})\|_{L_p(\Omega|\sF_u)}+\|\mE_s[h(\widetilde{X}_{r_n}^{n})-h(\widetilde{X}_r^{n})]f(\widetilde{X}_{r_n}^{n})\|_{L_p(\Omega|\sF_u)}.
	\de 
	It is clear that $\|fh\|_{p}\leq\|f\|_{p}\|h\|_{\infty} $. Applying (\ref{formula1+14}) to the first part of right hand side of above inequality  we see that
	\be\label{part_one_1}
	&&\|\mE_s[(hf)(\widetilde{X}_r^{n})-(hf)(\widetilde{X}_{r_n}^{n})]\|_{L_p(\Omega|\sF_u)}\no\\
	&\lesssim& \|f\|_{p}\|h\|_{\infty}(s-u)^{-\frac{d}{2p}}\bigg(\frac{1}{n}(r-s-2/n)^{-1}+(1/n)^{\frac{\alpha}{2}}+(1/n)^{1-\frac{\beta}{2}}\bigg).
	\ee 
  It suffices to estimate the $L_p(\Omega|\sF_u)$-norm of 
  	\ce 
  B:=\mE_{\widetilde{s}}[(h(\widetilde{X}_{r_n}^{n})-h(\widetilde{X}_r^{n}))f(\widetilde{X}_{r_n}^{n})].
  \de 
  	By conditioning on $\sF_{r_n}$, we have 
	\ce 
	B&=&Q_{\widetilde{s},r_n}^n[(h-Q_{r_n,r}^nh)f](\widetilde{X}_{\widetilde{s}}^{n}).
	\de 
	Using (\ref{formula1+130}), 
	\ce 
	\|B\|_{L_p(\Omega|\sF_u)}\lesssim  (\widetilde{s}-u)^{-\frac{d}{2p}}\|Q_{\widetilde{s},r_n}^n[(h-Q_{r_n,r}^nh)f]\|_{p}.
	\de 
	Then applying Theorem \ref{theorem4.5}, we get
	\ce 
	\|B\|_{L_p(\Omega|\sF_u)}\lesssim (\widetilde{s}-u)^{-\frac{d}{2p}}(r_n-\widetilde{s})^{-\frac{d}{2p}}\|[(h-Q_{r_n,r}^nh)f]\|_{\frac{p}{2}}.
	\de 
	By H\"{o}lder's inequality, we have 
	\ce 
	\|[(h-Q_{r_n,r}^nh)f]\|_{\frac{p}{2}}\leq \|h-Q_{r_n,r}^nh\|_{p}\|f\|_{p}, 
	\de 
	and then using Corollary \ref{corollary4.8}, we see that 
		\be\label{part_one_2}
	\|B\|_{L_p(\Omega|\sF_u)}&\lesssim& \|f\|_{p}(\widetilde{s}-u)^{-\frac{d}{2p}}(r_n-\widetilde{s})^{-\frac{d}{2p}}((1/n)^{\frac{\alpha}{2}}+(1/n)^{1-\frac{\beta}{2}})\|h\|_{\mH_{1,p}(\mR^d)}\no\\
	&\lesssim&\|f\|_{p}(s-u)^{-\frac{d}{2p}}(r-s-2/n)^{-\frac{d}{2p}}((1/n)^{\frac{\alpha}{2}}+(1/n)^{1-\frac{\beta}{2}})\|h\|_{\mH_{1,p}(\mR^d)}.
	\ee
	Combining the previous estimates (\ref{part_one_1}) and (\ref{part_one_2}), we finish the proof.
\end{proof}
\begin{proposition}
	Let  $\widetilde{X}^n$ satisfy (\ref{formula1+12+1}). Suppose that $f\in \mL_{p}^q([0,1])$ and $h\in \mH_{1,p}^q([0,1])\cap \mL^{\infty}([0,1])$, where $q\in [2,\infty)$ and $ p\in(2d/\beta\vee 2,\infty]$ are such that  $\frac{d}{p}+\frac{2}{q}<1$. Let $u\in[0,1-4/n]$ be a fixed number, $n\geq 4$ is an integer. Then for every $u+4/n\leq S\leq T\leq 1$, ones has the bound
	\be\label{formula0+4.28+} 
	&&\|\int_{S}^{T}h(r,\widetilde{X}_r^{n})(f(r,\widetilde{X}_r^{n})-f(r,\widetilde{X}_{r_n}^{n}))dr\|_{L_p(
		\Omega|\sF_u)}\no\\
	&\leq &N[(1/n)^{\alpha/2}+(1/n)^{ 1-\beta/2}+(1/n)^{1/2}\log(n)]T(f,h),
	\ee
	where the constant $N$ only depends on $d,p,q,c_1,c_0$ and $$T(f,h)=(\|h\|_{\mL^{\infty}([S,T])}+\|h\|_{\mH_{1,p}^{q}([S,T])})\|f\|_{\mL_{p}^{q}([S,T])}.$$
    In particular,
	\be \label{formula0+4.29+}
	&&\|\int_{S}^{T}(f(r,\widetilde{X}_r^{n})-f(r,\widetilde{X}_{r_n}^{n}))dr\|_{L_p(
		\Omega|\sF_u)}\no\\
	&\leq &N[(1/n)^{\alpha/2}+(1/n)^{ 1-\beta/2}+(1/n)^{1/2}\log(n)]\|f\|_{\mL_{p}^{q}([S,T])}.
	\ee
	
\end{proposition}
\begin{proof}
	Let $S,T$ be such that $u+4/n\leq S\leq T\leq 1$. By linearity, let us assume that $\|f\|_{\mL_{p}^{q}([S,T])}=\|h\|_{\mL^{\infty}([S,T])}+\|h\|_{\mL_{1,p}^{q}([S,T])}=1$.
	
	For every $(s,t)\in \Delta_2([S,T])$, define 
	\ce 
	A_{s,t}:=\mE_s\int _s^th(r,\widetilde{X}_r^{n})(f(r,\widetilde{X}_r^{n})-f(r,\widetilde{X}_{r_n}^{n}))dr.
	\de 
	Next we split the proof into two cases $t\leq s_n+4/n$ and $t> s_n+4/n$.
	
	\emph{Case 1.} For $t\in (s,s_n+4/n]$, by (\ref{formula1+130})  and  noting that $r_n-u\geq s_n-u\geq s-u-1/n\geq 2/n$, we have 
	\ce 
	\|A_{s,t}\|_{L_p(\Omega|\sF_u)}&\leq& \|h\|_{\mL^{\infty}([S,T])}\int_s^t\|f(r,\widetilde{X}_r^{n})\|_{L_p(\Omega|\sF_u)}+\|f(r,\widetilde{X}_{r_n}^{n})\|_{L_p(\Omega|\sF_u)}dr\\
	&\lesssim&\int_{s}^{t}(r_n-u)^{-\frac{d}{2p}}\|f(r,\cdot)\|_{p}dr.
	\de 
	Note that $r_n-u\geq s_n-u\geq (s-u)/2$. Using H\"{o}lder's inequality and the fact that $t-s\leq 4/n$, we have 
	\be\label{formula1+16+} 
	\|A_{s,t}\|_{L_p(\Omega|\sF_u)}&\lesssim& (s-u)^{-\frac{d}{2p}}(t-s)^{1-\frac{1}{q}}\|f(r,\cdot)\|_{\mL_{p\bar{p}_0}^q([S,T])}\no\\
	&\lesssim& (1/n)^{\frac{1}{2}}(s-u)^{-\frac{d}{2p}}(t-s)^{\frac{1}{2}-\frac{1}{q}}\|f\|_{\mL_{p}^q([S,T])}.
	\ee 
	
	\emph{Case 2.} When $t\in(s_n+4/n,1]$, we have 
	\ce 
	\|A_{s,t}\|_{L_p(\Omega|\sF_u)}\leq \|A_{s,s_n+4/n}\|_{L_p(\Omega|\sF_u)}+\int _{s_n+4/n}^t\|E_s[h(r,\widetilde{X}_r^{n})(f(r,\widetilde{X}_r^{n})-f(r,\widetilde{X}_{r_n}^{n}))]\|_{L_p(\Omega|\sF_u)}dr.
	\de 
	For the term $ \|A_{s,s_n+4/n}\|_{L_p(\Omega|\sF_u)}$, from  (\ref{formula1+16+}) we know
	\ce 
	\|A_{s,s_n+4/n}\|_{L_p(\Omega|\sF_u)}&\lesssim& (1/n)^{\frac{1}{2}}(s-u)^{-\frac{d}{2p}}(s_n+4/n-s)^{\frac{1}{2}-\frac{1}{q}}\|f\|_{\mL_{p}^q([S,T])}\\
	&\lesssim& (1/n)^{\frac{1}{2}}(s-u)^{-\frac{d}{2p}}(t-s)^{\frac{1}{2}-\frac{1}{q}}\|f\|_{\mL_{p}^q([S,T])}.
	\de 
	 The last inequality used the fact that $s_n+4/n-s\leq t-s$.  For the term $\int _{s_n+4/n}^t\|E_s[h(r,\widetilde{X}_r^{n})(f(r,\widetilde{X}_r^{n})-f(r,\widetilde{X}_{r_n}^{n}))]\|_{L_p(\Omega|\sF_u)}dr$, using H\"{o}lder's inequality  and (\ref{formula1+15}), we have 
	\ce 
	&&\int _{s_n+4/n}^t\|E_s[h(r,\widetilde{X}_r^{n})(f(\widetilde{X}_r^{n})-f(\widetilde{X}_{r_n}^{n}))]\|_{L_p(\Omega|\sF_u)}dr\\
	&\lesssim&(s-u)^{-\frac{d}{2p}}\int_{s_n+\frac{4}{n}}^{t}\|f(r,\cdot)\|_{p}\bigg[(r-s-2/n)^{-\frac{d}{2p}}\bigg((1/n)^{\frac{\alpha}{2}}+(1/n)^{1-\frac{\beta}{2}}\bigg)\|h(r,\cdot)\|_{\mH_{1,p\bar{p}_0}([S,T])}\\
	&& +\bigg(\frac{1}{n}(r-s-2/n)^{-1}+(1/n)^{\frac{\alpha}{2}}+(1/n)^{1-\frac{\beta}{2}}\bigg)\|h\|_{\mL_\infty([S,T])}\bigg]dr\\
	&\lesssim&(s-u)^{-\frac{d}{2p}}\|f\|_{\mL_{p}^q([s,t])}\bigg\{(t-s)^{1-\frac{d}{2p}-\frac{2}{p}}\bigg((1/n)^{\frac{\alpha}{2}}+(1/n)^{1-\frac{\beta}{2}}\bigg)\|h\|_{\mH_{1,p}^q([s,t])}\\
	&& +\bigg[\frac{1}{n}(t-s)^{-\frac{1}{q}}+\bigg((1/n)^{\frac{\alpha}{2}}+(1/n)^{1-\frac{\beta}{2}}\bigg)(t-s)^{1-\frac{1}{q}}\bigg]\|h\|_{\mL_\infty([S,T])}\bigg\}\\
	&\lesssim&(s-u)^{-\frac{d}{2p}}\|f\|_{\mL_{p}^q([s,t])}\bigg[(t-s)^{1-\frac{d}{2p}-\frac{2}{p}}\bigg((1/n)^{\frac{\alpha}{2}}+(1/n)^{1-\frac{\beta}{2}}\bigg)\|h\|_{\mH_{1,p}^q([s,t])}\\
	&& +(t-s)^{1-\frac{\alpha}{2}-\frac{1}{q}}\bigg((1/n)^{\frac{\alpha}{2}}+(1/n)^{1-\frac{\beta}{2}}\bigg)\|h\|_{\mL_\infty([S,T])}\bigg].
	\de 
	The last inequality used the fact that $(t-s)^{-1/q}(1/n)<(t-s)^{1-\alpha/2-1/q}(1/n)^{\alpha/2}$ when $t\in(s_n+4/n,1]$ and $(t-s)^{1-\frac{1}{q}}\leq (t-s)^{1-\frac{\alpha}{2}-\frac{1}{q}}$.	From now on, we obtain for $u+4/n\leq s\leq t\leq 1$,
	\be\label{formula1+17}
	&&\|A_{s,t}\|_{L_p(\Omega|\sF_u)}\no\\
	&\lesssim&(1/n)^{\frac{1}{2}}(s-u)^{-\frac{d}{2p}}(t-s)^{\frac{1}{2}-\frac{1}{q}}\|f\|_{\mL_{p}^q([S,T])}\no\\
	&&+(s-u)^{-\frac{d}{2p}}\|f\|_{\mL_{p}^q([s,t])}\bigg[(t-s)^{1-\frac{d}{2p}-\frac{2}{p}}\bigg((1/n)^{\frac{\alpha}{2}}+(1/n)^{1-\frac{\beta}{2}}\bigg)\|h\|_{\mH_{1,p}^q([s,t])}\no\\
	&& +(t-s)^{1-\frac{\alpha}{2}-\frac{1}{q}}\bigg((1/n)^{\frac{\alpha}{2}}+(1/n)^{1-\frac{\beta}{2}}\bigg)\bigg].
	\ee
	Furthermore, for any $v\in (s,t)$, we have $E_s\delta A_{s,v,t}=0.$ Let $\sW$ be the continuous control on $\Delta ([u+4/n,1])$ defined by 
	\ce 
	\sW(s,t)&=&\bigg[(s-u)^{-\frac{d}{2p}}(t-s)^{\frac{1}{2}-\frac{1}{q}}\|f\|_{\mL_{p}^q}([s,t])\bigg]^2\\
	&&+\bigg[(s-u)^{-\frac{d}{2p}}\|f\|_{\mL_{p}^q([s,t])}\bigg(t-s\bigg)^{1-\frac{\alpha}{2}-\frac{1}{p}}\bigg]^{\frac{2}{2-\alpha}}\\
	&&+\bigg[(s-u)^{-\frac{d}{2p}}\|f\|_{\mL_{p}^q([s,t])}\|h\|_{\mH_{1,p}^q([s,t])}\bigg(t-s\bigg)^{1-\frac{d}{2p}-\frac{2}{q}}\bigg]^{\frac{2p}{2p-d}}\\
	&&+(s-u)^{-\frac{d}{2p}}\|f\|_{\mL_{p}^q}\bigg(t-s\bigg)^{1-\frac{1}{p}}.
	\de 
	Denote 
	\ce 
	\sA_t:=\int_{0}^{t}h(r,\widetilde{X}_r^{n})(f(r,\widetilde{X}_r^{n})-f(r,\widetilde{X}_{r_n}^{n}))dr
	\de 
	and 
	\ce 
	\cJ_{s,t}:=\delta \sA_{s,t}-A_{s,t}.
	\de 
Applying similar estimates as those leading to (\ref{formula1+16+}), we have 
	\ce 
	\|\cJ_{s,t}\|_{L_p(\Omega|\sF_u)}\lesssim (s-u)^{-\frac{d}{2p}}(t-s)^{1-\frac{1}{q}}\|f\|_{\mL_{p}^q([s,t])}\lesssim\sW(s,t).
	\de 
	Furthermore, $\delta\cJ_{s,v,t}=-\delta A_{s,v,t}$ and from (\ref{formula1+17}), we derive that 
	\ce 
	\|\delta\cJ_{s,v,t}\|_{L_p(\Omega|\sF_u)}\lesssim (1/n)^{\frac{1}{2}}\sW(s,t)^{\frac{1}{2}}+\bigg[(1/n)^{\frac{\alpha}{2}}+(1/n)^{1-\frac{\beta}{2}}\bigg](\sW(s,t)^{1-\frac{\alpha}{2}}+\sW(s,t)^{1-\frac{d}{2p\bar{p}_0}}).
	\de 
	It is evident that $E_s\cJ_{s,t}=0$ and hence  $E_s\delta \cJ_{s,v,t}=0$. Applying Lemma \ref{DGlemma1} with chossing $m$ chosen such that $2^{-2m}\approx (1/n)^{1/2}$, we have 
	\ce 
	&&\|\cJ_{s,t}\|_{L_p(\Omega|\sF_u)}\\
	&\lesssim&[(1/n)^{\alpha/2}+(1/n)^{ 1-\beta/2}+(1/n)^{1/2}\log(n)][\sW(s,t)^{\frac{1}{2}}+\sW(s,t)^{1-\frac{\alpha}{2}}+\sW(s,t)^{1-\frac{d}{2p}}+\sW(s,t)].
	\de 
	By (\ref{formula1+17}) and the triangle inequality, this implies 
	\ce 
	&&\|\delta \sA_{s,t} \|_{L_p(\Omega|\sF_u)}\\
	&\lesssim&[(1/n)^{\alpha/2}+(1/n)^{ 1-\beta/2}+(1/n)^{1/2}\log(n)][\sW(s,t)^{\frac{1}{2}}+\sW(s,t)^{1-\frac{\alpha}{2}}+\sW(s,t)^{1-\frac{d}{2p}}+\sW(s,t)].
	\de 
	Since $\|f\|_{\mL^{\infty}([S,T])}=\|h\|_{\mL^{\infty}([S,T])}+\|h\|_{\mL_{1,p}^{q}([S,T])}=1$, we have 
	\ce 
	\sW(s,t)&\leq& \bigg[(s-u)^{-\frac{d}{2p}}(t-s)^{\frac{1}{2}-\frac{1}{q}}\bigg]^2+\bigg[(s-u)^{-\frac{d}{2p}}\bigg(t-s\bigg)^{1-\frac{\alpha}{2}-\frac{1}{p}}\bigg]^{\frac{2}{2-\alpha}}\\
	&&+\bigg[(s-u)^{-\frac{d}{2p}}\bigg(t-s\bigg)^{1-\frac{d}{2p}-\frac{2}{q}}\bigg]^{\frac{2p}{2p-d}}+(s-u)^{-\frac{d}{2p}}\bigg(t-s\bigg)^{1-\frac{1}{p}}.
	\de 
Hence,
	\ce 
	\sW(s,t)^{\frac{1}{2}}+\sW(s,t)^{1-\frac{\alpha}{2}}+\sW(s,t)^{1-\frac{d}{2p}}+\sW(s,t)\lesssim\sum_{i=1}^{16}(s-u)^{-\zeta_i}(t-s)^{\kappa_i},
	\de 
	where for each $i\in [1,16]\cap\mN$, $\zeta_i,\kappa_i\in[0,1]$ are some constants such that $\kappa_i-\zeta_i>0$. From the above estimate, we conclude that 
	\ce 
	\|\sA_{S,T}\|_{L_p(\Omega|\sF_u)}&\leq& \|\cJ_{s,t}\|_{L_p(\Omega|\sF_u)}+
	\|\delta \sA_{S,T}\|_{L_p(\Omega|\sF_u)}\\
	&\lesssim& [(1/n)^{\alpha/2}+(1/n)^{ 1-\beta/2}+(1/n)^{1/2}\log(n)]\sum_{i=1}^{16}(S-u)^{-\zeta_i}(T-S)^{\kappa_i},
	\de 
	which holds for every $u+4/n\leq S\leq T\leq 1$. We then apply Lemma \ref{sum-Lemma3.5} to get (\ref{formula0+4.28+}), and thus we complete the proof.
	
\end{proof}
\begin{proposition}\label{prop_difference_estimate_1}
	Let  $\widetilde{X}^n$ satisfies Eq. (\ref{formula1+12+1}). Suppose that $f\in \mL_{p}^{q}([0,1])\cap \mL_{\infty}^q([0,1])$ and $h\in \mH_{1,p}^q([0,1])\cap \mL^{\infty}([0,1])$, where $q\in [2,\infty)$ and $ p\in (2d/\beta\vee 2,\infty)$ are such that $\frac{d}{p}+\frac{2}{q}<1$. Define  $\gamma_n(f):=\sup_{r\in D_n}\|f\|_{\mL_\infty^q([r,r+1/n])}$. Then for any $p'\in(0,p)$ there is a constant $N=N(d,p,q,p')$ such that 
	\ce 
	&&\|\sup_{t\in[0,1]}|\int_{0}^{t}h(r,\widetilde{X}_r^{n})(f(\widetilde{X}_r^{n})-f(\widetilde{X}_{r_n}^{n}))dr|\|_{L_{p'}(\Omega)}\\
	&\leq &N\bigg\{\|h\|_{\mL^{\infty}([0,1])}\gamma_n(f)(1/n)^{1-\frac{1}{q}}+[(1/n)^{\alpha/2}+(1/n)^{ 1-\beta/2}+(1/n)^{1/2}\log(n)]T(f,h)\bigg\},
	\de 
	where $$T(f,h)=(\|h\|_{\mL^{\infty}([0,1])}+\|h\|_{\mH_{1,p}^{q}([0,1])})\|f\|_{\mL_{p}^{q}([0,1])}.$$
	Particularly 
	\be \label{formula0+4.34+}
	&&\|\sup_{t\in[0,1]}|\int_{0}^{t}(f(r,\widetilde{X}_r^{n})-f(r,\widetilde{X}_{r_n}^{n}))dr|\|_{L_{p'}(
		\Omega)}\no\\
	&\leq &N\bigg\{\gamma_n(f)(1/n)^{1-\frac{1}{q}}+[(1/n)^{\alpha/2}+(1/n)^{ 1-\beta/2}+(1/n)^{1/2}\log(n)]\|f\|_{\mL_{p}^{q}([0,1])}\bigg\}.
	\ee	
\end{proposition}

	\begin{proof}
		Let $$\sA_t=\int_{0}^{t}h(r,\widetilde{X}_r^{n})(f(\widetilde{X}_r^{n})-f(\widetilde{X}_{r_n}^{n}))dr.$$
		From the assumptions stated in the theorem and Hölder's inequality, we have
		\be \label{continuous_deltaA1}
		|\delta \sA_{s,t}|\leq 2\|h\|_{\mL^{\infty}([0,1])}\int_{s}^{t}\|f_r\|_{\infty}dr\leq 2\|h\|_{\mL^{\infty}([0,1])}\|f\|_{\mL_{\infty}^q([s,t])}(t-s)^{1-\frac{1}{q}},
		\ee 	
		for every $(s,t)\in [0,1]\times [0,1]$. Then $\sA_t$ has continuous sample paths. In view of Lemma	\ref{lemma_sup_1}, we only need to show that there exists a constant $N=N(d,p,q)$ such that 
		\be \label{estimate-bound1}
		\|\delta \sA_{s,t}\|_{L_{p}(\Omega|\sF_s)}&\leq& N\bigg\{\|h\|_{\mL^{\infty}([0,1])}\gamma_n(f)(1/n)^{1-\frac{1}{q}}\no\\
		&&+[(1/n)^{\alpha/2}+(1/n)^{ 1-\beta/2}+(1/n)^{1/2}\log(n)]T(f,h)\bigg\}
		\ee 
		for every $(s,t)\in \Delta$. 
		
		Case 1: For every $(s,t)\in \Delta$ satisfying $t\geq s+2/n$, we obtain from 
		(\ref{continuous_deltaA1}) and (\ref{formula0+4.28+}) that
		\ce 
		&&\|\delta \sA_{s,t}\|_{L_{p}(\Omega|\sF_s)}\\
		&\leq& \|\delta \sA_{s,s+2/n}\|_{L_{p}(\Omega|\sF_s)}+\|\delta \sA_{s+2/n,t}\|_{L_{p}(\Omega|\sF_s)}\\
		&\lesssim& \|h\|_{\mL^{\infty}([0,1])}\|f\|_{\mL_{\infty}^q([s,s+2/n])}(1/n)^{1-\frac{1}{q}}+[(1/n)^{\alpha/2}+(1/n)^{ 1-\beta/2}+(1/n)^{1/2}\log(n)]T(f,h).
		\de 
		
		Case 2: For  every $(s,t)\in \Delta$ satisfying $t\leq s+2/n$, (\ref{continuous_deltaA1}) implies that 
		\ce 
		\|\delta \sA_{s,t}\|_{L_{p}(\Omega|\sF_s)}\lesssim \|f\|_{\mL_{\infty}^q([s,s+2/n])}(1/n)^{1-\frac{1}{q}}.
		\de 
	
		By noting that $\|f\|_{\mL_{\infty}^q([s,s+2/n])}\lesssim \gamma_n(f)$, in both cases, we have obtained (\ref{estimate-bound1}). We complete the proof.
	\end{proof}

\begin{proof}[Proof of Theorem \ref{convergeestimate01}]
	We may assume without loss of generality that $f$ is nonnegative.  Define   
	\ce 
	\cR(X)=\sup_{t\in[0,1]}|\int h(r,X_r^n)[f(r,X_r^n)-f(r,X_{r_n}^n)]|.
	\de 
	Let $\widetilde{X}^n$ be the solution to Eq. (\ref{formula1+12+1}) and let $\rho$, $	\widetilde{B}_r$, $\widetilde{P}$ as defined in the proof of Theorem \ref{proposition0+4.13}.
	By noting 
	\ce 
	dX_r^n=\sigma(r,X_{r_n}^n)d\widetilde{B}_r+\int_{|z|\leq R}\beta(r,X_{{r_n-}}^{n}(x),z)\widetilde{N}(dr\,dz),
	\de 
 it follows from Girsanov's theorem and H\"{o}lder's inequality for $\theta>1$ close enough to 1 such that $1< \theta \bar{p}\leq p$ 
	\be\label{formula1+22} 
	\mE|\cR(X)|^{\bar{p}}&=&\widetilde{E}\big[\rho^{-1}|\cR(\widetilde{X})|^{\bar{p}}\big]
	\leq \widetilde{E}[\rho^{-\frac{\theta}{\theta-1}}]^{1-\frac{1}{\theta}} \widetilde{E}\big[|\cR(\widetilde{X})|^{\bar{p}\theta}\big]^{\frac{1}{\theta}},
	\ee
	where $\cR(\widetilde{X})=\sup_{t\in[0,1]}|\int h(r,\widetilde{X}_r^n)[f(r,\widetilde{X}_r^n)-f(r,\widetilde{X}_{r_n}^n)]|$.
	From Proposition (\ref{prop_difference_estimate_1}), we immediately get that 
   \be\label{cRformula1+22} 
	&&\widetilde{E}\big[|\cR(\widetilde{X})|^{\bar{p}\theta}\big]^{\frac{1}{\theta}}\no\\
	&\leq&	 N\bigg\{\|h\|_{\mL^{\infty}([0,1])}\gamma_n(f)(1/n)^{1-\frac{1}{q}}+[(1/n)^{\alpha/2}+(1/n)^{ 1-\beta/2}+(1/n)^{1/2}\log(n)]T(f,h)\bigg\},
	\ee
	where $$T(f,h)=(\|h\|_{\mL^{\infty}([0,1])}+\|h\|_{\mH_{1,p}^{q}([0,1])})\|f\|_{\mL_{p}^{q}([0,1])}.$$
	Set $w_r=(\sigma^{-1}b_n)(r,\widetilde{X}_{r_n}^n)$ and 
	\ce 
	\cW=\exp\bigg(\frac{1}{2}\bigg(\frac{\theta}{\theta-1}+(\frac{\theta}{\theta-1})^2\bigg)\int_{0}^1|w_r|^2dr\bigg).
	\de 
From the proof in Theorem \ref{proposition0+4.13},  we see that $\widetilde{E}[|\cW|^2]$ is bounded uniformly in $n$ and hence $\widetilde{E}[\rho^{-\frac{\theta}{\theta-1}}]$ is also bounded. Combining  (\ref{formula1+22}), (\ref{cRformula1+22} ) and  the above estimates, we complete the proof.
\end{proof}

	\section{Pathwise Estimates for the Exact Solution}
	
		\begin{theorem}\label{theorem6.1+}
		Let $p\in (2(d/\beta\vee1),\infty)$, $q\in (2,\infty)$. 
		
		(i) Assuming Condition  ($\sH_\sigma^g$) and ($\sH_b$), let $f\in \mL_{p_1}^{q_1}([0,1])$ for some $p_1\in (2(d/\beta\vee1), q_1\in [2,\infty]$ satisfying $\frac{d}{p_1}+\frac{2}{q_1}<2$. Then for every $m\geq 1$, there exists a constant $N=N(d,p_1,q_1,m)$ such that 
		\ce 
		\|\int_{0}^{t}f(r,\bar{X}_r)dr\|_{L_m(\Omega)}\leq N\|f\|_{\mL_{p_1}^{q_1}([0,1])}.
		\de 
		
		(ii) Assuming Condition  ($\sH_\sigma^g$) and ($\sH_b$) with $q_0=\infty$ and $p>1$, let  $\theta\in [0,1)$ and suppose $f\in \mL_p^q([0,1])\cap\mH_{-\theta,p}^q([0,1])$, where the indices satisfy $\frac{d}{p}+\frac{2}{q}+\theta<2$. Then for any $\bar{p}\in (0,p)$, there exists a constant $N=N(\theta,d,p,q,\bar{p})$ such that 
		\ce 
		\|\sup_{t\in[0,1]}\int_{0}^{t}f(r,\bar{X}_r)dr\|_{\mL_{\bar{p}}(\Omega)}\leq N\|f\|_{\mH_{-\theta,p}^{q}([0,1])}.
		\de 
		
		(iii) Assuming Condition  ($\sH_\sigma^g$) and ($\sH_b$) with $q_0=\infty$, $p>1$, and $\frac{d}{p}+\frac{2}{q}<2$, let $f\in \mL_p^q([0,1])$. Let $\varLambda>0$ be a constant and $\sW_1$ be a continuous control on $\Delta$. We assume that for every $(s,t)\in \Delta$,
		\ce 
		\|f\|_{\mH_{-\theta,p}^{q}([0,1])}\leq \varLambda \sW_1(s,t)^{\frac{1}{q}}~\mbox{and}~\|f\|_{\mL_{p}^{q}([0,1])}\leq \sW_1(s,t)^{\frac{1}{q}}.
		\de  
		Then for any $\bar{p}\in(0,p)$, there exists a constant $N=N(\theta,d,p,q,\bar{p})$ such that 
		\ce 
		\|\sup_{t\in[0,1]}\int_{0}^{t}f(r,\bar{X}_r)dr\|_{\mL_{\bar{p}}(\Omega)}\leq N(1+|\log(\varLambda)|)\sW_1(0,1)^{\frac{1}{q}}.
		\de 
	\end{theorem}

	\subsection{Analytic estimates}
	For each $t\in(0,1]$, we consider the following backward second oder parabolic integral-differential equation: 
	\be\label{Equation0+6.1}
	\partial_ru+\mathscr{L}_2^au+\mathscr{L}_{\upsilon,R}^gu=f,~~u(t,\cdot)=0,
	\ee 
	where $f\in\mH_{-1,p}^q$. When the dependence on $t$ is relevant, we denote by $u_t(s,x)$  the solution to Eq. (\ref{Equation0+6.1}) evaluated at $(s,x),~s\leq t,~x\in\mR^d$.
	\begin{theorem}\label{Theorem_6.1}
		Assuming Condition ($\sH_\sigma^g$), ($\widehat{\sH}_{\sigma}^g$) with $q_0=\infty$, let $q\in(2,\infty)$ and $p\in(d/\beta\vee1,\infty)$ be such that $\frac{1}{p}+\frac{1}{p_0}<1$. Then for every $\theta\in[0,1]$, every $0\leq s\leq t\leq 1$ and every $f\in\mH_{-\theta,p}^p([0,1])$, we have 
		\ce 
		\|u_t(s,\cdot)\|_{p}\leq N(t-s)^{1-\frac{\theta}{2}-\frac{1}{q}}\|f\|_{\mH_{-\theta,p}^p([s,t])}.
		\de 
	\end{theorem}
	\begin{proof}
		By interpolation, it suffices to show that 
		\be \label{formula6.4-2}
		\|u_t(s,\cdot)\|_{p}\leq N(t-s)^{\frac{1}{2}-\frac{1}{q}}\|f\|_{\mH_{-1,p}^p([0,1])},~~\mbox{for}~f\in\mH_{-1,p}^q([0,1]),
		\ee 
		\be \label{formula6.4-1-0}
		\|u_t(s,\cdot)\|_{p}\leq N(t-s)^{1-\frac{1}{q}}\|f\|_{\mL_{p}^p([0,1])},~~\mbox{for}~f\in\mL_{p}^q([0,1]).
		\ee 
	We proceed with the proof in two steps.
		
		Step 1. Let $u$ be a function in $\mH_{2,p}^q([0,1])$ or $\mH_{1,p}^q([0,1])$ with $u(r,x)=0$ for $r\in[t,1]$.  We may assume that $u$ is a smooth function on $[0,1]\times \mR^d$ with compact support, by approximation. By Duhamel's fomula, we have 
		\ce 
		u(s,x)=\int_{s}^{t}p_{s,r}(\partial_su+\frac{1}{2}\Delta u)(r,x)dr.
		\de 
		Applying the Minkowski iequality and \cite[Theorem 5.30]{Triebel2013}, we can get, if $u\in \mH_{2,p}^q([0,1])$,
		\ce 
		\|u_s\|_{p}&\leq& \int_{s}^{t}\|p_{s,r}(\partial_su+\frac{1}{2}\Delta u)(r,\cdot)\|_{p}dr\\
		&\lesssim& \int_{s}^{t}\|(\partial_su+\frac{1}{2}\Delta u)(r,\cdot)\|_{p}dr,
		\de 
		and if $u\in \mH_{1,p}^q([0,1])$,
			\be\label{formula-heat-est6.4} 
		\|u_s\|_{p}&\leq& \int_{s}^{t}\|p_{s,r}(\partial_su+\frac{1}{2}\Delta u)(r,\cdot)\|_{p}dr\no\\
		&\lesssim& \int_{s}^{t}\|(\partial_su+\frac{1}{2}\Delta u)(r,\cdot)\|_{\mH_{-1,p}(\mR^d)}dr.
		\ee 
		Using H\"{o}lder's inequality, we get for $r\in [0,t]$
		\be\label{formula-heat-est6.2} 
		\|u_s\|_{p}\lesssim (r-s)^{1-\frac{1}{p}}\|(\partial_su+\frac{1}{2}\Delta u)(r,\cdot)\|_{\mL_{p}^q(\mR^d)}.
		\ee  
     	If $u\in \mH_{1,p}^q([0,1])$, applying (\ref{formula-heat-est6.4} ) and proceeding similarly to (\ref{formula-heat-est6.2}), we get for $r\in [0,t]$
		\be\label {formula-heat-est6.3} 
		\|u_s\|_{p}\lesssim (t-s)^{\frac{1}{2}-\frac{1}{q}}\|(\partial_su+\frac{1}{2}\Delta u)(r,\cdot)\|_{\mH_{-1,p}^q(\mR^d)}. 
		\ee 
		
		Step 2. From (\ref{Equation0+6.1}), we have 
		\ce 
		\partial_su+\frac{1}{2}\nabla u=f+(\frac{1}{2}\nabla-\mathscr{L}_2^a)u-\mathscr{L}_{\upsilon,R}^gu.
		\de 
		It follows from (\ref{formula-heat-est6.3}) that for every $0\leq s\leq t\leq 1$,
			\be\label{formula6.4-1} 
		\|u^t(s)\|_{p}&\lesssim& (t-s)^{1-\frac{1}{q}}\|f+(\frac{1}{2}\nabla-\mathscr{L}_2^a)u-\mathscr{L}_{\upsilon,R}^gu\|_{\mH_{-1,p}^q(\mR^d)}\no\\
		&\lesssim& (t-s)^{1-\frac{1}{q}}(\|f\|_{\mH_{-1,p}^q(\mR^d)}+\|(\frac{1}{2}\nabla-\mathscr{L}_2^a)u^t\|_{\mH_{-1,p}^q(\mR^d)}+\|\mathscr{L}_{\upsilon,R}^gu^t\||_{\mH_{-1,p}^q(\mR^d)}).
		\ee 
	From  part (III) in Lemma \ref{Multipliction-norm1} and  (\ref{formula0+A.3}) in Theorem \ref{TheoremA.4}, we have 
		\ce 
	&&\|(\frac{1}{2}\nabla-\mathscr{L}_2^a)u^t\|_{\mH_{-1,p}^q(\mR^d)}+\|\mathscr{L}_{\upsilon,R}^gu^t\|_{\mH_{-1,p}^q(\mR^d)}\\
	&\lesssim&(\|a\|_{\infty}+\|\nabla a\|_{p_0})\|\nabla^2u^t\|_{\mH_{-1,p}(\mR^d)}+\|\mathscr{L}_{\upsilon,R}^gu^t\|_{\mH_{-1,p}^q(\mR^d)}\\
	&\lesssim& \|f\|_{\mH_{-1,p}^q(\mR^d)}.
		\de 
Combining the previous estimates, we obtain (\ref{formula6.4-2}).

	Analogously, by using (\ref{formula-heat-est6.3}), we get
		\be\label{formula6.4+0} 
		\|u_t(s)\|_{p}&\lesssim& (t-s)^{1-\frac{1}{q}}\|f+(\frac{1}{2}\nabla-\mathscr{L}_2^a)u-\mathscr{L}_{\upsilon,R}^gu\|_{\mL_{p}^q(\mR^d)}\no\\
		&\lesssim& (t-s)^{1-\frac{1}{q}}\|f\|_{\mL_{p}^q(\mR^d)}+\|\nabla^2u\|_{\mL_{p}^q(\mR^d)}+\|\mathscr{L}_{\upsilon,R}^gu\|_{\mL_{p}^q(\mR^d)}.
		\ee 
		By simply computing, we have 
		\ce 
		\|\mathscr{L}_{\upsilon,R}^gu\|_{\mL_{p}^q(\mR^d)}\leq \|\varGamma_{0,2}^{0,R}g\|_{\mL^{\infty}([0,1])}\|\nabla^2 u\|_{\mL_{p}^q(\mR^d)}.
		\de 
		From \cite[Theorem 4.3]{Xie-Zhang_2020}, it is known that 
		\ce 
		\|\nabla^2u\|_{\mL_{p}^q(\mR^d)}\lesssim \|f\|_{\mL_p^q([0,1])}.
		\de 
		By substituting these two estimates into  (\ref{formula6.4+0}), we arrive at (\ref{formula6.4-1-0}).
	\end{proof}
	
		\subsection{Generalized It\^{o} formula}
		\begin{proposition}\label{strong-solution1}
		Under Conditions $(\sH_\sigma^g)$  and $(\sH_b)$, 
		
		(i) there  exists unique strong solution to Eq. (\ref{EQU1});
		
		(ii) for any measurable function $f\in\mL_{p}^q([0,1])$ with $q\in [2,\infty]$,  $ p\in(2(d/\beta\vee 1),\infty)$ satisfying $\frac{d}{p}+\frac{2}{q}<1$, we obtain that
		\be\label{formula0+24}
		\mE_s\left [\int_{s}^tf(r ,X_r)dr\right ]\leq N\|f\|_{\mL_{p}^q([s,t])},
		\ee
		where $N$ depends only on $d,p,q,c_0$.
		\end{proposition}
	\begin{proof}
	Under our conditions, we apply Theorem 2.2 from \cite{Xie-Zhang_2020} to establish the existence and uniqueness of the strong solution to Eq. (\ref{EQU1}), and Theorem 5.6 from \cite{Xie-Zhang_2020} to derive Eq. (\ref{formula0+24}).
	\end{proof}

	\begin{lemma}[Generalized It\^{o} formula]\label{Generalized Ito formula}
	Under the assumptions that Conditions ($\sH_\sigma^g$) and ($\sH_b$) hold with $p\in(d/2\vee1,\infty), q\in[2,\infty)$ such that $\frac{d}{p}+\frac{2}{q}<1$, let $X_t$ be a solution of Eq. (\ref{EQU1}).  Furthermore, let $p'\in(d/2\vee1,\infty), q'\in[2,\infty)$ with $\frac{d}{p'}+\frac{2}{q'}<1$. For any $U\in\mH_{2,p'}^{q'}([0,1])$ with $\partial_tU\in\mL_{p'}^{q'}([0,1])$, we have the following generalized Itô formula:
		\be\label{Generalized-Ito-formula}
		&&U(t,X_t)\no\\
		&=&U(0,x)+\int_{0}^{t}(\partial_rU+\frac{1}{2}a^{i,j}\partial_{i,j}U+b^i\partial_iU)(r,X_r)dr+\int_{0}^{t}(\sigma^{i,j}\partial_iU)(r,X_r)dB_r^j\no\\
		&&+\int_{0}^{t}\int_{|z|<R}[U(r,X_{r-}+g(r,X_{r-},z))-U(r,X_{r-})]\widetilde{N}(dz\,dr)\no\\
		&&+\int_{0}^{t}\int_{|z|<R}[U(r,X_{r}+g(r,X_{r},z))-U(r,X_{r})-g^i(r,X_{r},z)\partial_iU(r,X_{r})]\upsilon(dz)dr.
		\ee
	\end{lemma}
	\begin{proof}
		Let $U_n:=U*\eta_n$ be the mollifying approximation of $U$ in $\mR^{d+1}$. By applying the  It\^{o} formula, we have 
		\be\label{Ito-formula-1} 
		&&U_n(t,X_t)\no\\
		&=&U_n(0,x)+\int_{0}^{t}(\partial_rU_n+\frac{1}{2}a^{ij}\partial_{ij}U_n+b^i\partial_iU_n)(r,X_r)dr+\int_{0}^{t}(\sigma^{ij}\partial_iU_n)(r,X_r)dB_r^j\no\\
		&&+\int_{0}^{t}\int_{|z|<R}[U_n(r,X_{r-}+g(r,X_{r-},z))-U_n(r,X_{r-})-g^i(r,X_{r-},z)\partial_iU_n(r,X_{r-})]\upsilon(dz)dr\no\\
		&&+\int_{0}^{t}\int_{|z|<R}[U_n(r,X_{r-}+g(r,X_{r-},z))-U_n(r,X_{r-})]\widetilde{N}(dz\,dr).
		\ee
		Let $\widehat{U}:=U_n-U$.	By the It\^{o} isometric formula and (\ref{formula0+24}), we have 
		\ce 
		\mE\left[\big|\int_{0}^{t}(\sigma^{ij}\partial_i\widehat{U})(r,X_r)dB_r^j\big|^2\right]
		&\leq &\|\sigma\|_{\mL^{\infty}([0,1])}\mE[\int_{0}^{t}|\nabla\widehat{U}(r,X_r)|^2dr]\\
		&\lesssim&\|(\nabla\widehat{U})^2\|_{\mL_{p'/2}^{q'/2}}=\|\nabla\widehat{U}\|_{\mL_{p'}^{q'}([0,1])}^2,
		\de 
		which converges to zero as $n\to\infty$. Similarly 
		\ce 
		\mE\left[\int_{0}^{t}|(\partial_r+a^{ij}\partial_{ij})\widehat{U}|(r,X_r)dr\right]=0,~as~n\to \infty.
		\de 
		Let $\frac{1}{p_1}:=\frac{1}{p}+\frac{1}{p'}$, $\frac{1}{q_1}:=\frac{1}{q}+\frac{1}{q'}$. Since $\frac{d}{p_1}+\frac{2}{q_1}<2$,  by (\ref{formula0+24}) and H\"{o}lder's inequality we have 
		\ce 
		\mE\left[\int_{0}^{t}|(b^i\partial_i\widehat{U})(r,X_r)|dr\right]&\lesssim& \|b^i\partial_i\widehat{U}\|_{\mL_{p_1}^{q_1}([0,1])}\\
		&\leq&\|b\|_{\mL_{p}^{q}([0,1])}\|\nabla(U_n-U)\|_{\mL_{p'}^{q'}([0,1])}=0,~\mbox{as}~n\to \infty.
		\de 
		By the isometric formula, we obtain 
		\ce 
		&&A:=\mE\left[\left|\int_{0}^{t}\int_{|z|<R}\widehat{U}(r,X_{r-}+g(r,X_{r-},z))-\widehat{U}(r,X_{r-})\widetilde{N}(dz\,dr)\right|^2\right]\\
		&=&\mE\left[\int_{0}^{t}\int_{|z|<R}|\widehat{U}(r,X_{r-}+g(r,X_{r-},z))-\widehat{U}(r,X_{r-})|^2\upsilon(dz)dr\right]\\
		&\leq&\mE\left[\int_{0}^{t}\int_{|z|<R}|\widehat{U}(r,X_{r-}+g(r,X_{r-},z))-\widehat{U}(r,X_{r-})-g(r,X_{r-},z)\cdot\nabla\widehat{U}(r,X_{r-})|^2\upsilon(dz)dr\right]\\
		&&+\mE\left[\int_{0}^{t}\int_{|z|<R}|g(r,X_{r-},z)\cdot\nabla\widehat{U}(r,X_{r-})|^2\upsilon(dz)dr\right]\\
		&\leq&\|\varGamma_{0,2}^{0,R}g\|_{\mL^{\infty}([0,1])}\bigg(\mE\left[\int_{0}^{t}\sup_{y\neq0}|y|^{-2}|\widehat{U}(r,X_{r-}+y)-\widehat{U}(r,X_{r-})-y\cdot\nabla \widehat{U}(r,X_{r-})|^2dr\right]\\
		&&+\mE\left[\int_{0}^{t}|\nabla \widehat{U}(r,X_{r-})|^2dr\right]\bigg).
			\de 
	Since $\frac{2d}{p'}+\frac{2}{q'}<2$, applying (ii) of Proposition \ref{strong-solution1} to both terms
	 $$\mE\left[\int_{0}^{t}\sup_{y\neq0}|y|^{-2}|\widehat{U}(r,X_{r-}+y)-\widehat{U}(r,X_{r-})-y\cdot\nabla \widehat{U}(r,X_{r-})|^2dr\right]$$ 
	 and 
	 $$\mE\left[\int_{0}^{t}|\nabla \widehat{U}(r,X_{r-})|^2dr\right],$$
		 we get 
		\ce 
		A&\lesssim&\|\varGamma_{0,2}^{0,R}g\|_{\mL^{\infty}([0,1])}\big(\|\sup_{y\neq0}|y|^{-2}|\widehat{U}(\cdot,\cdot+y)-\widehat{U}-y\cdot\nabla\widehat{U}|^2\|_{\mL_{p'/2}^{q'/2}([0,1])}+\|\nabla \widehat{U}\|_{\mL_{p'/2}^{q'/2}([0,1])}\big).
		\de 
		Using (\ref{formula1+5}) with $\alpha=1$, we have the following estimte
		\ce 
		A&\lesssim&\|\varGamma_{0,2}^{0,R}g\|_{\mL^{\infty}([0,1])}\big(\|\sup_{y\neq0}|y|^{-1}|\widehat{U}(\cdot,\cdot+y)-\widehat{U}-y\cdot\nabla \widehat{U}|\|_{\mL_{p'}^{q'}([0,1])}^2+\|\nabla \widehat{U}\|_{\mL_{p'}^{q'}([0,1])}^2\big)\\
		&\leq&\|\varGamma_{0,2}^{0,R}g\|_{\mL^{\infty}([0,1])}\|\widehat{U}\|_{\mH_{1,p'}^{q'}([0,1])}^2=0,~\mbox{as}~n\to \infty.
		\de 
	Similarly, by first applying (ii) of Proposition \ref{strong-solution1} and then applying (\ref{formula1+5}), it is easy to get 
		\ce 
		&&\mE\left[\left|\int_{0}^{t}\int_{|z|<R}[\widehat{U}(r,X_{r-}+g(r,X_{r-},z))-\widehat{U}(r,X_{r-})-g(r,X_{r-},z)\cdot \nabla\widehat{U}(r,X_{r-})]\upsilon(dz)dr\right|\right]\\
		&\leq&\|\varGamma_{0,2}^{0,R}g\|_{\mL^{\infty}([0,1])}^2\mE\left[\int_{0}^{t}\sup_{y\neq0}|y|^{-2}|\widehat{U}(r,X_{r-}+y)-\widehat{U}(r,X_{r-})-y\cdot\nabla\widehat{U}(r,X_{r-})|dr\right]\\
		&\lesssim&\|\varGamma_{0,2}^{0,R}g\|_{\mL^{\infty}([0,1])}^2\|\sup_{y\neq0}|y|^{-2}|\widehat{U}(\cdot,\cdot+y)-\widehat{U}-y\cdot\nabla\widehat{U}\|_{\mL_{p'}^{q'}([0,1])}\\
		&\lesssim&\|\varGamma_{0,2}^{0,R}g\|_{\mL^{\infty}([0,1])}^2\|\widehat{U}\|_{\mH_{2,p'}^{q'}([0,1])}	=0,~\mbox{as}~n\to \infty.
		\de 
	
	By taking limits $n$ approaches infinity for both sides of (\ref{Ito-formula-1}), we derive the desired formula.
	\end{proof}

	\subsection{Moment estimates}
	Let $\bar{X}$ be a solution to SDE with jumps given by
	\be\label{Equation6.8+0} 
	d\bar{X}_t=\sigma(t,\bar{X}_t)dB_t+\int_{|z|<R}g(t,\bar{X}_t,z)\widetilde{N}(ds\,dz),~\bar{X}_0=x\in\mR^d.
	\ee 
	By Proposition \ref{strong-solution1},  it is known that Eq. (\ref{Equation6.8+0}) has a unique srong solution, which is also a Markov process. Let $Q_{s,t}$  be the transtion operator associated to $\bar{X}$. In particular, we have 
	\ce 
	E[f(\bar{X}_t)|\sF_s]=Q_{s,t}f(\bar{X}_s)
	\de  
	for any bounded measurable function $f$.

	\begin{lemma}
	Assume that the condtion ($\sH_\sigma^g$)  holds. Let $p,p'\in (d/\beta\vee d,\infty]$, $p\leq p'$, $p<\infty$, and let $f\in L_{p}(\mR^d)$. There exists a constant $N(d,p,p',\alpha,c_0,c_1)>0$ such that for  $t\in[s,1]$, we have 
		\be\label{transition6.9}
		\|Q_{s,t}f\|_{p'}\leq N(t-s)^{\frac{d}{2p'}-\frac{d}{2p}}\|f\|_{p}.
		\ee 	
	\end{lemma}
\begin{proof}
	Let $X^n$ be the solution to the Euler-Maruyama scheme (\ref{EulerAPP3.8}). It  suffices  to show that the laws of $X^n$ converge to the law of $\bar{X}$. Once this is established,  (\ref{transition6.9})  can be derived from Theorem  \ref{theorem4.5}. 	
	
	Let us first show that  laws of  $X^n$  converge to the law of $X$. Let $P^n$ be the probability law of  $X^n$ on the  Skorokhod space $\sD[0,1]$ with the Skorokhod topology, the Borel $\sigma$-algebra and the filtration $t\mapsto \sG_t=\sigma\{\omega_s:\s\in[0,t]\}$. Let $\psi$ be a smooth function with bounded derivatives. By It\^{o} formula, we see that
	\ce 
	M_t^n(\omega)=\psi(\omega_t)-\psi(x)-\int_0^t\sL_2^{\sigma(r,\omega_{r_n})} \psi(\omega_r)+\sL_\upsilon^{g(r,\omega_{r_n})} \psi(\omega_r)dr
	\de 
	is a martingale under $P^n(d\omega)$. Define 
	\ce 
	M_t(\omega)=\psi(\omega_t)-\psi(x)-\int_0^t\sL_2^{\sigma(r,\omega_{r})} \psi(\omega_r)+\sL_\upsilon^{g(r,\omega_{r})} \psi(\omega_r)dr.
	\de 
	
	Let $\tau\le 1$ be an arbitrary stopping time, let $0\le \theta$, and set
	\ce
	\tau_\theta = (\tau+\theta)\wedge 1.
	\de
	By the It\^o isometry, we obtain
	\ce
	\mathbb{E}\bigl|\widetilde{X}_{\tau_\theta}^{n}-\widetilde{X}_{\tau}^{n}\bigr|^2
	\le 2\mathbb{E}\left[\int_{\tau}^{\tau_\theta}\|\sigma(r,\widetilde{X}_{r_n}^{n})\|^2dr
	+\int_{\tau}^{\tau_\theta}\int_{|z|\le R}|g(r,\widetilde{X}_{r_n-}^{n},z)|^2\,\upsilon(dz)dr\right].
	\de
	By the boundedness assumption, there exists a constant $N$ independent of $n,\tau,\theta$ such that 
	\be\label{formula0+26+}
		\mathbb{E}\bigl|\widetilde{X}_{\tau_\theta}^{n}-\widetilde{X}_{\tau}^{n}\bigr|^2
		\le N\,\theta.
	\ee
	
	By Aldous's tightness criterion, we see that  (\ref{formula0+26+}) implies that the probability laws $\{P^n\}$ are tight. Let $P$ be a probability measure such that $P_n$ converges to $P$ through a subsequence, which we still denote
	by $P_n$. Let $s\leq t$ be fixed and $G\in \sG_s$, we have
	\ce  
	&&\int \delta M_{s,t} \textbf{1}_{G}d P\\
	&=&\int \delta M_{s,t} \textbf{1}_{G}(d P-dP^n)+\int \left(\delta M_{s,t} \textbf{1}_{G}-\delta M_{s,t}^n \textbf{1}_{G}\right)dP_n+\int \delta M_{s,t}^n \textbf{1}_{G}dP^n\\
	&=:& A_1+A_2+A_3.
	\de 
	It is evident that $\lim_{n\to \infty }A_1=0$ and $A_3=0$ because of $M^n$ being a martingale under $P^n$. For the term $A_2$, we have 
	\ce 
	A_2&=&\int \left(\delta M_{s,t} \textbf{1}_{G}-\delta M_{s,t}^n \textbf{1}_{G}\right)dP_n(\omega)\\
	&=&\int\int_s^t \left[\sL_2^{\sigma(r,\omega_{r_n})}-\sL_2^{\sigma(r,\omega_{r})}\right] \psi(\omega_r)dr\textbf{1}_{G}dP_n(\omega)\\
	&&+\int\int_s^t \left[\sL_{\upsilon,R}^{g(r,\omega_{r_n})} -\sL_{\upsilon,R}^{g(r,\omega_{r})} \right]\psi(\omega_r)dr\textbf{1}_{G}dP_n(\omega)\\
	&=:&A_{21}+A_{22}.
	\de 
	For the term $A_{21}$, using H\"{o}lder continuity of $\sigma$ and (\ref{formula0+26+}) we have 
	\ce 
	|A_{21}|&\lesssim& \int \left[\int_s^t|\omega_r-\omega_{r_n}|^\alpha\right] dP^n(\omega)\\
	&\lesssim&\int_s^t\mE|\bar{X}_r^n-\bar{X}_{r_n}^n|^\alpha dr\lesssim \int_s^t|r-r_n|^{\frac{\alpha}{2}} dr.
	\de 
	This implies that $\lim_{n\to\infty} A_{21}=0$. As for the term $A_{22}$, we have
	\ce 
	|A_{22}|&=&\int\int_s^t \left|\sL_{\upsilon,R}^{g(r,\omega_{r_n})} -\sL_{\upsilon,R}^{g(r,\omega_{r})} \right|\psi(\omega_r)dr\textbf{1}_{G}dP_n(\omega)\\
	&=&\int\int_s^t\int_{|z|<R} \left|\psi(\omega_r+g(r,\omega_{r_n},z))-\psi(\omega_r+g(r,\omega_{r}),z)\right.\\
	&&\left.-(g(r,\omega_{r_n},z)-g(r,\omega_{r},z))\cdot\nabla \psi(\omega_r)\right|\upsilon(dz) dr \textbf{1}_{G} dP_n(\omega)\\
	&=&\int\int_s^t\int_{|z|<R} \int_0^1\left|(\nabla\psi)(\omega_r+(1-\theta)g(r,\omega_{r_n},z)-\theta g(r,\omega_{r},z))-\nabla \psi(\omega_r)\right|d\theta\\
	&&\cdot |g(r,\omega_{r_n},z)-g(r,\omega_{r},z)|\upsilon(dz) dr \textbf{1}_{G} dP_n(\omega)\\
		&\lesssim &\int\int_s^t\int_{|z|<R} \int_0^1\left[\theta|g(r,\omega_{r_n},z)|+(1-\theta)|g(r,\omega_{r},z)|\right]d\theta\\
	&&\cdot |g(r,\omega_{r_n},z)-g(r,\omega_{r},z)|\upsilon(dz) dr \textbf{1}_{G} dP_n(\omega)\\
		&\lesssim &\int\int_s^t\int_0^1\int_{|z|<R} \left[|g(r,\omega_{r_n},z)|+|g(r,\omega_{r},z)|\right]\\
	&&\cdot |(\nabla g)(r,(1-\theta)\omega_{r}+\theta\omega_{r_n},z)|\upsilon(dz) d\theta |\omega_{r_n}-\omega_{r}|dr \textbf{1}_{G} dP_n(\omega).
	\de 
	Using Cauchy-Schwarz inequality, we obtain 
	\ce 
	|A_{22}|&\lesssim& \left[\|\varGamma_{0,2}^{0,R} g\|_{\mL^{\infty}[0,1]}+\|\varGamma_{1,2}^{0,R} g\|_{\mL^{\infty}[0,1]}\right] \int\int_s^t |\omega_{r_n}-\omega_{r}|dr \textbf{1}_{G}. dP_n(\omega)\\
	&\lesssim& \int_s^t\mE|\bar{X}_r^n-\bar{X}_{r_n}^n|dr\lesssim \int_s^t|r-r_n|^{\frac{1}{2}}dr.
	\de 
	This implies that $\lim_{n\to\infty} A_{22}=0$. Now we know that $M$ is a Martingale under $P$. In other words, $P$ is a solution to the martingale problem associated to Eq. (\ref{Equation6.8+0} ), which is unique from (i) of Proposition \ref{strong-solution1}.  We have shown that $\{P_n\}_n$ has exactly one accumulating point, which is the law of solution to (\ref{Equation6.8+0}). This also means that $\bar{X}_r^n$ converges weakly to $\bar{X}_r$.
	
\end{proof}

By utilizing (\ref{transition6.9}) and employing an argument analogous to that used in Proposition (\ref{proposition4.90}), we arrive at the following lemma.
\begin{lemma}\label{P-Krelovy-e1}
	Given that Condition ($\sH_\sigma^g$) holds, let $f\in L_{\rho\bar{p}_0}(\mR^d)$  be a measurable function on $\mR^d$  for some $\rho\in (0,\infty]$ and $\rho\bar{p}_0\in (d/\beta\vee 1, \infty]$. Under these conditions, there exists a finite constant $N(\alpha,d,\rho,c_0,c_1)>0$ such that for every $r,u\in [0,1]$
	\be \label{formula_6.10}
	\|f(\bar{X}_r)\|_{L_{\rho}(\Omega|\sF_u)}\leq N(r-u)^{-\frac{d}{2\rho\bar{p}_0}}\|f\|_{\rho\bar{p}_0}.
	\ee 
\end{lemma}

	\begin{proposition}\label{s_Krelovy-esitimate1}
	Let $p\in  (d/\beta\vee 1, \infty)$, $q\in [2,\infty)$. 
		
		(i) Under the assumption of Condition $(\sH_\sigma^g)$, let $f\in \mL_{p_1}^{q_1}([0,1])$ for some $p_1\in (d/\beta\vee 1, \infty], q_1\in [1,\infty]$ satisfying $\frac{d}{p_1}+\frac{2}{q_1}<2$. Then for every $m\geq 1$, there exists a positive constant $N=N(d,p_1,q_1,m)$ such that 
		\ce 
		\|\int_{0}^{t}f(r,\bar{X}_r)dr\|_{L_m(\Omega)}\leq N\|f\|_{\mL_{p_1}^{q_1}([0,1])}.
		\de 
		
		(ii) Assuming Conditions $(\sH_\sigma^g)$ and  $(\widehat{\sH}_{\sigma}^g)$  with $q_0=\infty$ and $\frac{1}{p}+\frac{1}{p_0}<1$ and $p\geq 2$, let $\theta\in [0,1)$ and  $h\in \mL_{p}^{q}([0,1])\cap\mH_{-\theta,p}^q([0,1])$  such that $\frac{d}{p}+\frac{2}{q}+\theta<2$. Then for any $\bar{p}\in (0,p)$, there exists a positive constant $N=N(\theta,d,p,q,\bar{p})$ such that 
		\ce 
		\|\sup_{t\in[0,1]}\int_{0}^{t}h(r,\bar{X}_r)dr\|_{\mL_{\bar{p}}(\Omega)}\leq N\|h\|_{\mH_{-\theta,p}^{q}([0,1])}.
		\de 
		
		(iii)Assuming Conditions $(\sH_\sigma^g)$ and  $(\widehat{\sH}_{\sigma}^g)$  with $q_0=\infty$, $\frac{1}{p}+\frac{1}{p_0}<1$ and $\frac{d}{p}+\frac{2}{q}<2$, let $h\in \mL_p^q([0,1])$ and let $\varLambda>0$ be a constant. Let $\sW_1$ be a continuous control on $\varDelta$. We assume that for every $(s,t)\in \Delta$,
		\ce 
		\|h\|_{\mH_{-1,p}^{q}([0,1])}\leq \varLambda \sW_1(s,t)^{\frac{1}{q}}~\mbox{and}~\|h\|_{\mL_{p}^{q}([0,1])}\leq \sW_1(s,t)^{\frac{1}{q}}.
		\de  
		Then for any $\bar{p}\in(0,p)$, there exists a positive constant $N=N(\theta,d,p,q,\bar{p})$ such that 
		\ce 
		\|\sup_{t\in[0,1]}\int_{0}^{t}h(r,\bar{X}_r)dr\|_{L_{\bar{p}}(\Omega)}\leq N(1+|\log(\varLambda)|)\sW_1(0,1)^{\frac{1}{q}}.
		\de 
	\end{proposition} 
	\begin{proof}
		(i) From Lemma \ref{Khasminskii's lemma}, it suffices to show that 
		\be \label{continousestimate01}
		\mE_s\int_s^t|f(r,\bar{X}_s)|dr\lesssim (t-s)^{1-\frac{d}{2p_1}-\frac{1}{q_1}}\|f\|_{\mL_{p_1}^{q_1}([s,t])}.
		\ee 
		Using the Minkowski inequality, (\ref{formula_6.10}) with $\rho=1, \bar{p}_0=p_1$ and H\"{o}lder's inquality, we have for $(s,t)\in \Delta$ and $\tau<s$,
		\ce 
		\|\int_s^t|f(r,\bar{X}_r)|dr\|_{L_{1}(\Omega|\sF_{\tau})}&\leq& \int_s^t\|f(r,\bar{X}_r)\|_{L_{1}(\Omega|\sF_{\tau})}dr\\
		&\lesssim&\int_s^t(r-\tau)^{-\frac{d}{2p}}\|f_r\|_{p_1}dr\\
		&\lesssim&(s-\tau)^{-\frac{d}{2p_1}}(t-s)^{1-\frac{1}{q_1}}\|f\|_{\mL_{p_1}^{q_1}([s,t])}.
		\de 
		From Lemma \ref{sum-Lemma3.5}, we have 
		\ce 
		\|\int_s^tf(r,\bar{X}_r)dr\|_{L_{p}(\Omega|\sF_{\tau})}\lesssim (t-s)^{1-\frac{d}{2p_1}-\frac{1}{q_1}}\|f\|_{\mL_{p_1}^{q_1}([s,t])}.
		\de 
		From (\ref{continousestimate01}) and the Kolmogorov continuity theorem, the process 
		\ce 
		t\mapsto \int_0^tf(r,\bar{x}_r)dr
		\de
		has a continuous version. For this version, we see that 
			\ce 
		\|\int_s^tf(r,\bar{X}_r)dr\|_{L_{p}(\Omega|\sF_{s})}\leq (t-s)^{1-\frac{d}{2p_1}-\frac{1}{q_1}}\|f\|_{\mL_{p_1}^{q_1}([s,t])},
		\de 
		and then we get (\ref{continousestimate01}), completing the proof of part (i).
		
		(ii) In view of Lemma \ref{lemma_sup_1}, it suffices to show that 
		\be\label{estimateii-01}
		\sup_{(s,t)\in\Delta}\|\int_s^th(r,\bar{X}_r)dr\|_{L_{p}(\Omega)|\sF_{\tau}}\lesssim \|h\|_{\mH_{-\theta,p}^{q}([0,1])}.
		\ee 
		
  	Let $u^t\in\mH_{2,p}^q([0,t])$ be the solution of following partial differential equation 
	\be\label{equality-PDE1}
		\partial_ru+\mathscr{L}_2^au+\mathscr{L}_{\upsilon,R}^gu=-h,~~u(t,\cdot)=0.
	\ee 
		Using Generalized It\^{o} formula (\ref{Generalized-Ito-formula}) for $u^t(s,\bar{X}_s)$, we see that
		\be \label{ITOformula01}
		&&u^t(s,\bar{X}_s)\no\\
		&=&u^t(t,\bar{X}_t)+\int_{s}^{t}(\partial_ru^t+\mathscr{L}_2^au^t+\mathscr{L}_{\upsilon,R}^gu^t)(r,\bar{X}_r)dr+\int_{s}^{t}(\sigma^{ij}\partial_iu^t)(r,\bar{X}_r)dB_r^j\no\\
		&&+\int_{s}^{t}\int_{|z|<R}\bigg(u^t(r,\bar{X}_{r-}+g(r,\bar{X}_{r-},z))-u^t(r,\bar{X}_{r-})\bigg)\widetilde{N}(dz\,dr).
		\ee
		From (\ref{equality-PDE1}) and $u^t(t,\bar{X}_t)=0$, we get 
		\ce 
		\int_s^th(r,\bar{X}_r)dr&=&-u^t(s,\bar{X}_s)+\int_{s}^{t}(\sigma^{i,j}\partial_iu)(r,X_r)dB_r^j\\
		&&+\int_{s}^{t}\int_{|z|<R}[u(r,X_{r-}+g(r,X_{r-},z))-u(r,X_{r-})]\widetilde{N}(dz\,dr).
		\de 
		Let $0<\tau<s$. Taking norm $\|\cdot\|_{L^p(\Omega|\sF_\tau)}$ on  both sides of the above equation, we have
			\ce 
		&&\|\int_s^th(r,\bar{X}_r)dr\|_{L^p(\Omega|\sF_\tau)}\\
		&=&\|u^t(s,\bar{X}_s)\|_{L^p(\Omega|\sF_\tau)}+\|\int_{s}^{t}(\sigma^{i,j}\partial_iu)(r,X_r)dB_r^j\|_{L^p(\Omega|\sF_\tau)}\\
		&&+\|\int_{s}^{t}\int_{|z|<R}[u(r,X_{r-}+g(r,X_{r-},z))-u(r,X_{r-})]\widetilde{N}(dz\,dr)\|_{L^p(\Omega|\sF_\tau)}\\
		&=:&A_1+A_2+A_3.
		\de 
		We now estimate $A_1$. Applying (\ref{formula_6.10}) and Theorem \ref{Theorem_6.1}, we have 
		\ce 
		A_1&\lesssim& (s-\tau)^{-\frac{d}{2p}}\|u_s^t\|_p\\
		&\lesssim&(s-\tau)^{-\frac{d}{2p}}(t-s)^{1-\frac{\theta}{2}-\frac{1}{q}}\|h\|_{\mH_{-\theta,p}^{q}([s,t])}.
		\de 
		From Burkholder–Davis–Gundy's inequality, use (\ref{formula_6.10}) and H\"{o}lder's inequality to get 
		\ce 
		A_2&\leq& \int_{s}^{t}\|(\sigma\nabla u)(r,X_r)\|_{L^p(\Omega|\sF_\tau)}dr\\
		&\lesssim &\int_{s}^{t}(r-\tau)^{-\frac{d}{2p}}\|\nabla u(r,\cdot)\|_{p}dr\\
		&\lesssim &(s-\tau)^{-\frac{d}{2p}}(t-s)^{1-\frac{1}{q}}\|\nabla u(r,\cdot)\|_{\mL_p^q([s,t])}.
		\de  
	Applying the	Burkholder–Davis–Gundy inequality and (\ref{formula_6.10}) with $\bar{p}_0=1$ to $A_3$, we obtain
		\ce 
		A_3&\lesssim &\int_{s}^{t}\|\bigg(\int_{|z|<R}|u(r,X_{r}+g(r,X_{r},z))-u(r,X_{r})|^p\upsilon(dz)\bigg)^{\frac{1}{p}}\|_{L^p(\Omega|\sF_\tau)}dr\\
		&&+\int_{s}^{t}\|\bigg(\int_{|z|<R}|u(r,X_{r}+g(r,X_{r},z))-u(r,X_{r})|^2\upsilon(dz)\bigg)^{\frac{1}{2}}\|_{L^{p}(\Omega|\sF_\tau)}dr\\
		&\lesssim&\int_{s}^{t}(r-\tau)^{-\frac{d}{2p}}\|\bigg(\int_{|z|<R}|u(r,\cdot+g(r,\cdot,z))-u(r,\cdot)|^p\upsilon(dz)\bigg)^{\frac{1}{p}}\|_{L^p}dr\\
		&&+\int_{s}^{t}(r-\tau)^{-\frac{d}{2p}}\|\bigg(\int_{|z|<R}|u(r,\cdot+g(r,\cdot,z))-u(r,\cdot)|^2\upsilon(dz)\bigg)^{\frac{1}{2}}\|_{L^{p}(\Omega|\sF_\tau)}dr.
		\de 
	   Using Lemma \ref{lemma3.2} and  H\"{o}lder's inequality,  we have 
		\ce 
		A_3	&\leq& \int_{s}^{t}(r-\tau)^{-\frac{d}{2p}}\|\varGamma_{0,2}^{0,R}g(r,\cdot)\|_{\infty}^{\frac{1}{p}}\|\sup_{y\neq 0}|y|^{-\frac{2}{p}}|u(r,\cdot+y)-u(r,\cdot)|\|_{L^p}dr\\
			&&+\int_{s}^{t}(r-\tau)^{-\frac{d}{2p}}\|\varGamma_{0,2}^{0,R}g(r,\cdot)\|_{\infty}^{\frac{1}{2}}\|\sup_{y\neq 0}|y|^{-1}|u(r,\cdot+y)-u(r,\cdot)|\|_{L^p}dr\\
			&\lesssim &\int_{s}^{t}(r-\tau)^{-\frac{d}{2p}}\bigg(\| u(r,\cdot)\|_{\mH_{1,p}(\mR^d)}+\|u\|_{\mH_{\frac{2}{q},p}(\mR^d)}\bigg)dr\\
			&\lesssim &(s-\tau)^{-\frac{d}{2p}}(t-s)^{1-\frac{1}{q}}(\| u(r,\cdot)\|_{\mH_{1,p}^q([0,1])}+\|u(r,\cdot)\|_{\mH_{\frac{2}{q},p}^q([0,1])})\\
			&\lesssim &(s-\tau)^{-\frac{d}{2p}}(t-s)^{1-\frac{1}{q}}\| u(r,\cdot)\|_{\mH_{1,p}^q([0,1])}\\
			&\lesssim &(s-\tau)^{-\frac{d}{2p}}(t-s)^{1-\frac{1}{q}}\|h\|_{\mH_{-\theta,p}^q([s,t])}.
		\de 
	We have utilized the following facts in deriving the last inequality. Firstly, if $h\in \mL_p^q([s,t])$, then according to \cite[Theorem 4.3.]{Xie-Zhang_2020}, we have 
		\ce 
		\|u\|_{\mH_{1,p}^q([s,t])}\lesssim \| u\|_{\mH_{2,p}^q([s,t])}\lesssim \|h\|_{\mL_p^q([s,t])}.
		\de 
	Secondly, if  $h\in \mH_{-1,p}^q([s,t])$, from Theorem \ref{TheoremA.4}, we have
		\ce 
		\|u\|_{\mH_{1,p}^q([s,t])}\lesssim \|h\|_{\mH_{-1,p}^q([s,t])}.
		\de 
By applying interpolation for $\theta\in[0,1)$,  we can deduce that
		\ce
		\|u\|_{\mH_{1,p}^q([s,t])}\lesssim  \|h\|_{\mH_{-\theta,p}^q([s,t])}.
		\de 
Finally, an application of Lemma \ref{sum-Lemma3.5} leads to the establishment of (\ref{estimateii-01}).
	
	(iii) For each $(s,t)\in \Delta$ with $s>\tau$, define 
	\ce 
	\sA_{s,t}=E_s\int_s^th(r,\bar{X}_r)dr, \mbox{and}~\cJ_{s,t}=\int_s^th(r,\bar{X}_r)dr-\sA_{s,t}.
	\de  	 
	Dfine the control $\sW_2$ by 
	\ce 
	\sW_2=[(s-\tau)^{-\frac{d}{2p}}(t-s)^{\frac{1}{2}-\frac{1}{q}}\sW_1(s,t)^{\frac{1}{q}}]^{2}+(s-\tau)^{-\frac{d}{2p}}(t-s)^{1-\frac{1}{q}}\sW_1(s,t)^{\frac{1}{q}}.
	\de 
	From (\ref{ITOformula01}),  it is evident that  $\sA_{s,t}=u^t(s,\bar{X}_s)$. By applying (\ref{formula_6.10}) and Theorem (\ref{Theorem_6.1}), we have for every $\tau<s\leq u\leq t$ that $E_s\delta\sA_{s,u,t}=0$, and
	\be\label{formula_estimate_A} 
	\|\sA_{s,t}\|_{L_{p}(\Omega|\sF_\tau)}&\lesssim& (s-\tau)^{-\frac{d}{2p}}\|u^t(s,\cdot)\|_{p}\no\\
	&\lesssim&(s-\tau)^{-\frac{d}{2p}}(t-s)^{\frac{1}{2}-\frac{1}{q}}\|h\|_{\mH_{-1,p}^q([s,t])}.
	\ee 
Then by (\ref{formula_estimate_A} ) we have 
	\ce 
	\|\delta\cJ_{s,u,t}\|_{L_p(\Omega|\sF_\tau)}&=&\|\delta\sA_{s,u,t}\|_{\mL_p(\Omega|\sF_\tau)}\lesssim\sW_2^{\frac{1}{2}}.
	\de 
	Applying Minkowski's inequality, (\ref{formula_6.10}) and H\"{o}lder's inequality, we get 
	\ce 
	\|\cJ_{s,t}\|_{L_{p}(\Omega|\sF_\tau)}&\leq& 2\int_s^t\|h(r,\bar{X}_r)\|_{L_{p}(\Omega|\sF_\tau)}dr\\
	&\lesssim&\int_s^t(r-\tau)^{-\frac{d}{2p}}\|h(r,\cdot)\|_pdr\\
	&\lesssim&(s-\tau)^{-\frac{d}{2p}}(t-s)^{1-\frac{1}{q}}\|h\|_{\mL_{p}^q([s,t])}.
	\de 
	It is clear that $\|\cJ_{s,t}\|_{L_{p}(\Omega|\sF_\tau)}\leq \sW_2$, $E_s\cJ_{s,t}=0$, and hence $E_s\delta\cJ_{s,u,t}=0$.  This verifes the conditions (3.1-3.3) of Lemma \ref{DGlemma1}. Therefore, an application of Lemma  \ref{DGlemma1} yields that for every $\tau<s\leq u\leq t\leq 1$,
	\ce 
		\|\cJ_{s,t}\|_{L_{p}(\Omega|\sF_\tau)}\lesssim \varLambda(1+\log(\varLambda))\sW_2(s,t)^{\frac{1}{2}}+\varLambda\sW_2(s,t).
	\de 
	By the triangle inequality,  
	\ce 
	\|\int_s^t h(r,\bar{x}_r)dr\|_{L_{p}(\Omega|\sF_\tau)}\lesssim \varLambda(1+\log(\varLambda))\sW_2(s,t)^{\frac{1}{2}}+\varLambda\sW_2(s,t).
	\de 
	An application of Lemma \ref{sum-Lemma3.5} gives 
	\ce 
	\|\int_s^t h(r,\bar{x}_r)dr\|_{L_{p}(\Omega|\sF_\tau)}\lesssim \varLambda(1+\log(\varLambda))\sW_1(0,1).
	\de 
From (\ref{continousestimate01}) and Kolmogorov continuity theorem, the process 
\ce 
t\mapsto \int_0^th(r,\bar{x}_r)dr
\de
 has a continuous version. For this version, we see that the previous estimate holds for every $(s,t)\in \Delta$ and $\tau=s$. This shows that (iii) holds by using Lemma\ref{lemma_sup_1}. 
	\end{proof}
	
	\begin{proof}[Proof of Theorem \ref{theorem6.1+}]
    Let $\bar{X}$ be the solution to Eq. (\ref{Equation6.8+0} ). To derive moment estimates for $X$ from those obtained in Proposition \ref{s_Krelovy-esitimate1}, we will employ Girsanov's transformation, similar to the proof of Theorem \ref{proposition0+4.13}.  First of all, we define
	\ce 
	\rho:=\exp\bigg(\int_{0}^{t}(\sigma^{-1}b)(r,\bar{X}_{r})dB_r-\frac{1}{2}\int_{0}^{t}|(\sigma^{-1}b)(r,\bar{X}_{r})|^2dr\bigg).
	\de 
	 Let $\bar{P}$ denote the probability measure defined by $d\bar{P}=\rho dP$ and let $\bar{E}$ represent the expectation under  $\bar{P}$. We recall Conditions $\sH_\sigma^g$, $\sH_b$  and the uniform ellipticity of $\sigma$. These conditions imply that the function $f:=|\sigma^{-1}b|^2$ belongs to $\mL_{p/2}^{q/2}([0,1])$ and satisfies (\ref{continousestimate01}). By applying Lemma \ref{Khasminskii's lemma} and Novikov's criterion, we see that $E[\rho^{r}]$ is bounded for any $r\in\mR$.
	 Consequently, Girsanov's theorem implies that
	\ce 
	dX_r=\sigma(r,X_{r})d\bar{B}_r+\int_{|z|\leq R}\beta(r,X_{r-},u)\widetilde{N}(dr\,du),
	\de 
	where 
	\ce 
	\bar{B}_r:=\int_{0}^{r}(\sigma^{-1}b)(s,X_{s})ds+B_r,~t\in[0,1].
	\de 
	 Here, $\bar{B}_r$ is a Brownain motion under  $\bar{P}$, and $N(dr\,dz)$ remains a Poisson random measure with the same compensator  $dt\upsilon(dz)$. Using similar computations as in the proof of Theorem \ref{proposition0+4.13}, we can now deduce Theorem \ref{theorem6.1+} from Proposition \ref{s_Krelovy-esitimate1}. This completes the proof.
	\end{proof}

\section{Proof of Theorem \ref{M-theorem2.2} and Theorem\ref{M-theorem2.4}}

In this section, we prove Theorem \ref{M-theorem2.2} and Theorem \ref{M-theorem2.4}. 
The proof of the strong convergence estimate is based on the Zvonkin transform associated with the approximating drift $b^n$. 
We first derive a transformed error identity which separates the Brownian-type errors from the jump-induced errors. 
The estimates of these terms will be given in the subsequent subsections.

\begin{lemma}\label{first-lemma-10.1}
	Assume that Conditions ($\sH_\sigma^g$) and  ($\sH_b$) hold. Let $p\in(d/2\vee1, \infty)$, $q\in(1,\infty)$. suppose $f\in\mL_p^q([0,1])$ and $A>0$ be such that 
	\ce 
	\|f\|_{\mL_p^q([0,1])}+\|b\|_{\mL_p^q([0,1])}+\|\varGamma_{0,2}^{0,R}\|_{\mL^{\infty}([0,1])}\leq A.
	\de 
	Then there exists constant $\lambda_0=\lambda_0(A,a)>0$ such that for all $\lambda\geq \lambda_0$, Eq. (\ref{PDE_b}) has a unique solution $u$ in $\mL_{2,p}^q([0,1])$. Moreover, for any $\vartheta\in(0,2)$ for some $p_1\in [p,\infty]$ and $q_1\in[q,\infty]$ satisfying
	\ce 
	\frac{d}{p}+\frac{2}{q}<2-\vartheta+\frac{d}{p_1}+\frac{2}{q_1},
	\de 
	there exists a constant $N=N(d,p,q,p_1,q_1,A,\vartheta)>0$ such that 
	\ce 
	(\lambda\vee 1)^{\frac{1}{2}(2-\vartheta+\frac{d}{p_1}+\frac{2}{q_1}-\frac{d}{p}-\frac{2}{q})}\|u\|_{\mH_{\vartheta,p_1}^{q_1}(0,1)}+\|\partial_t u\|_{\mL_p^q([0,1])}+\|u\|_{\mH_{2,p}^q([0,1])}\leq C\|f\|_{\mL_p^q([0,1])}.
	\de 
\end{lemma}
\begin{proof}
	This is a direct consequences of Theorem 4.3 in \cite{Xie-Zhang_2020} and the fact 
	\ce 
	&&\|\sL_{\upsilon,R}^gu\|_{\mL_p^q([0,1])}\\
	&\leq& \|\int_{|z|<R}|u_t(x+g_t(x,z))-u_t(x)-g_t(x,z)\cdot \nabla u(x)|\upsilon(dz)\|_{\mL_p^q([0,1])}\\
	&\leq&\|\varGamma_{0,2}^{0,R}g\|_{\mL^{\infty}([0,1])}\|\sup|y|^{-2}|u(x+y)-u(y)-y\cdot \nabla u(x)||\varGamma_{0,2}^{0,R}g_t(x)|\|_{\mL_p^q([0,1])}\\
	&\leq&\|\sup|y|^{-2}|u(x+y)-u(y)-y\cdot \nabla u(x)|\|_{\mL_p^q([0,1])}\\
	&\leq&\|\varGamma_{0,2}^{0,R}g\|_{\mL^{\infty}([0,1])}\|\sup|y|^{-2}|u(x+y)-u(y)-y\cdot \nabla u(x)||\varGamma_{0,2}^{0,R}g_t(x)|\|_{\mL_p^q([0,1])}\\
	&\lesssim&\|u\|_{\mH_{2,p}^q([0,1])}.
	\de 
\end{proof}
\begin{remark}
	Recall that $u=(u^1,\cdots,u^d)$, where for each $k=1,\cdots,d$, $u^k$ is the solution to Eq. (\ref{PDE_b}) with $b=b^{n,k}$.	In light of Conditions ($\sH_\sigma^g$) and ($\sH_b$), there exists $\lambda_0>0$ such that for every $\lambda\geq\lambda_0$ and $\vartheta\in(0,2)$, we have
	\be \label{shauderestimate01}
	\sup_n\|\partial_t u\|_{\mL_p^q([0,1])}+\sup_n\|u\|_{\mH_{2,p}^q([0,1])}+\sup_n\|u\|_{\mH_{\vartheta,\infty}^\infty([0,1])}<\infty 
	\ee 
	and 
	\be\label{shauderestimate02}
	\sup_n\|\nabla u\|_{\mL^{\infty}([0,1])}=o_\lambda(1),
	\ee 
	where $o_\lambda(1)$ denote any constant such that $o_\lambda(1)\to 0$ as $\lambda\to \infty$.
\end{remark}

\begin{lemma}\label{middle-estimate2}
	Let $q>0$ and consider $u$ as  the solution to Eq. (\ref{PDE_b}). There exists a  lebesgue zero set $E$ such that for all $x,y\notin E$, the following property holds
	\be \label{general_estmate0}
	&&\int_{0}^{t}\int_{|z|<R}|[u(r,x+g(r,x,z))-u(r,x)]-[u(r,y+g(r,y,z))-u(r,y)]|^{q}\upsilon(dz)dr\no \\
	&&+\int_{0}^{t}\int_{|z|<R}|g(r,x,z)-g(r,y,z)|^q\upsilon(dz)dr\no\\
	&\leq&N\int_{0}^{t}|x-y|^{ q}[\|\varGamma^{0,R}_{1,q}|g_r|\|_{\infty}+\varGamma^{0,R}_{0,q}|g(r,y)|\sup_{z\neq 0,w\neq 0}|z|^{-1}|w|^{-1}|u(r,x+w+z)\no\\
	&&-u(r,x+z)-(u(r,x+w)-u(r,x))|]^qdr
	\ee 
	and 
	\be \label{general_estmate00}
	&&\int_{0}^{t}\bigg(\int_{|z|<R}|[u(r,x+g(r,x,z))-u(r,x)]-[u(r,y+g(r,y,z))-u(r,y)]|^2\upsilon(dz)\bigg)^{q/2}dr\no \\
	&&+\int_{0}^{t}\bigg(\int_{|z|<R}|g(r,x,z)-g(r,y,z)|^2\upsilon(dz)\bigg)^{q/2}dr\no\\
	&\leq&
	N\int_{0}^{t}|x-y|^{ q}[\|\varGamma^{0,R}_{1,2}|g_r|\|_{\infty}+\varGamma^{0,R}_{0,2}|g(r,y)|\sup_{z\neq 0,w\neq 0}|z|^{-1}|w|^{-1}|u(r,x+w+z)\no\\
	&&-u(r,x+z)-(u(r,x+w)-u(r,x))|]^qdr,
	\ee 
	where $N$ is a positive constant only depending on $d,q$.
\end{lemma}
\begin{proof}
	By performing direct calculations and making use of Eq.  (\ref{shauderestimate02}), it can be derived that
	\ce 
	&&\int_{0}^{t}\int_{|z|<R}|[u(r,x+g(r,x,z))-u(r,x)]-[u(r,y+g(r,y,z))-u(r,y)]|^q\upsilon(dz)dr\\
	&=&\int_{0}^{t}\int_{|z|<R}|u(r,x+g(r,x,z))-u(r,x+g(r,y,z))|^q\upsilon(dz)dr\\
	&&+\int_{0}^{t}\int_{|z|<R}|u(r,x+g(r,y,z))-u(r,y+g(r,y,z))-(u(r,x)-u(r,y))|^q\upsilon(dz)dr\\
	&\lesssim&\int_{0}^{t}\int_{|z|<R}|g(r,x,z)-g(r,y,z)|^q\upsilon(dz)dr+A_1,
	\de
	where 
	\ce 
	A_1=\int_{0}^{t}\int_{|z|<R}|u(r,x+g(r,y,z))-u(r,y+g(r,y,z))-(u(r,x)-u(r,y))|^q\upsilon(dz)dr.
	\de 
	By	applying (\ref{formulaA.7+00}) with $\mB=L_{q}(B_R;\upsilon)$, we obtain the following estimate
	\be\label{general_estmate1} 
	\int_{0}^{t}\int_{|z|<R}|g(r,x,z)-g(r,y,z)|^q\upsilon(dz)dr
	\lesssim\int_{0}^{t}(|x-y|\|\varGamma^{0,R}_{1,q}|g_r|\|_{\infty})^qdr.
	\ee 
	Next, we  get
	\be\label{general_estmate2} 
	A_1
	&=&\int_{0}^{t}\int_{|z|<R}|u(r,x+(y-x)+g(r,y,z))-u(r,x+g(r,y,z))\no\\
	&&-(u(r,x+(y-x))-u(r,x))|^q\upsilon(dz)dr\no\\
	&\lesssim&\int_{0}^{t}|x-y|^{q}(\sup_{z\neq 0,w\neq 0}|z|^{-1}|w|^{-1}|u(r,x+w+z)-u(r,x+z)\no\\
	&&-(u(r,x+w)-u(r,x))|)^q\varGamma^{0,R}_{0,q}|g(r,y)|dr.
	\ee 
	The estimates given in (\ref{general_estmate1}) and (\ref{general_estmate2}) together imply the result stated in (\ref{general_estmate0}). A similar proof can be used to obtain (\ref{general_estmate00}).
\end{proof}

The following lemma can be seen as a modification of the pathwise Burkholder-Davis-Gundy inequality \cite[Theorem 5]{Pietro2018} to fit our specific conditions.
\begin{lemma}\label{power-estimete1}
	Let $G_t(z)$ be a predictable process on $[0,1]\times\mR^d$ and define
	\ce 
	\bar{W}_t=\int_0^t\int_{|z|<R}G_s(z)\widetilde{N}(dzds).
	\de 
	
	(i)	For any $p\geq 1$, there is a positive constant $N_p$ such that 
	\ce
	|\bar{W}_t|^{p}\leq N_{p}\int_0^t\bigg[\int_{|z|<R}|G_s(z)|^{p}\upsilon(dz)+\bigg(\int_{|z|<R}|G_s(z)|\upsilon(dz)\bigg)^{p}\bigg]ds+\bar{M}_t,
	\de 
	where $\bar{M}_t$ is a local martingale with $\bar{M}_0=0$.
	
	(ii) For any $p\geq 2$, there is a positive constant $N_p$ such that 
	\ce
	\sup_{r\in[0,t]}|\bar{W}_r|^{p}\leq N_{p}\int_0^t\bigg[\int_{|z|<R}|G_s(z)|^{p}\upsilon(dz)+\bigg(\int_{|z|<R}|G_s(z)|^2\upsilon(dz)\bigg)^{p/2}\bigg]ds+M_t^0,
	\de 
	where $M_t^0$ is a local martingale with $M_0^0=0$.
\end{lemma}
\begin{proof}
	(i)	For $p> 1$, let $W_t=|\bar{W}_t|^{p}$.	Using the It\^{o} formula to $|\bar{W}_t|^{p}$, one has 
	\ce 
	W_t&=&\int_0^t\int_{|z|<R}|\bar{W}_s+G_s(z)|^{p}-|\bar{W}_s|^{p}-pG_s(z)\cdot\sgn(\bar{W}_s)
	|\bar{W}_s|^{p-1}\upsilon(dz)ds+M_t,
	\de 
	where $M_t$ is a local martingale with $M_0=0$. Noticing that 
	\ce 
	|x+y|^{p}-|x|^{p}-py\cdot\sgn(x)|x|^{p-1}\leq N_{p} |y|(|x|+|y|)^{p-1}\leq  N_{p} (|y||x|^{p-1}+|y|^{p}),
	\de 
	by Young's inequality, we get 
	\ce 
	W_t&\leq& N_{p}\int_0^t\int_{|z|<R}|G_s(z)|^{p}+|G_s(z)||\bar{W}_s|^{p-1}\upsilon(dz)ds+M_t\\
	&\leq& N_{p}\int_0^t\bigg[\int_{|z|<R}|G_s(z)|^{p}\upsilon(dz)+|\bar{W}_s|^{p-1}\int_{|z|<R}|G_s(z)|\upsilon(dz)\bigg]ds+M_t\\
	&\leq&N_{p} \int_0^t\bigg[\int_{|z|<R}|G_s(z)|^{p}\upsilon(dz)+\bigg(\int_{|z|<R}|G_s(z)|\upsilon(dz)\bigg)^{p}\bigg]ds+N_{p}\int_0^tW_sds+M_t.
	\de  
	Let 
	\ce 
	K_t=N_{p}\int_0^t\bigg[\int_{|z|<R}|G_s(z)|^{p}\upsilon(dz)+\bigg(\int_{|z|<R}|G_s(z)|\upsilon(dz)\bigg)^{p}\bigg]ds+N_{p}\int_0^tW_sds+M_t
	\de
	and 
	\ce 
	A_t=\int_0^tW_s/K_sds,
	\de
	where we will use the convention $0/0=0$ if necessary. Using the It\^{o} formula again, we get
	\ce 
	e^{-A_t}K_t=N_{p}\int_0^te^{-A_s}\bigg[\int_{|z|<R}|G_s(z)|^{p}\upsilon(dz)+\bigg(\int_{|z|<R}|G_s(z)|\upsilon(dz)\bigg)^{p}\bigg]ds+\int_0^t e^{-A_s} dM_s.
	\de 
	By the fact $e^{-A_t}\leq e^{-A_s}\leq 1$ for $s<t$, $A_t\leq t\leq 1$ and $W_t\leq K_t$, we have
	\ce 
	W_t\leq N_{p}\int_0^t\bigg[\int_{|z|<R}|G_t(z)|^{p}\upsilon(dz)+\bigg(\int_{|z|<R}|G_t(z)|\upsilon(dz)\bigg)^{p}\bigg]ds+\bar{M}_t,
	\de 
	where $\bar{M}_t$ is a local martingale  with $\bar{M}_0=0$.
	
	For $p= 1$, the conclusion is obvious.
	
	(ii)	By the pathwise Burkholder-Davies-Gundy inequality of \cite[Theorem 5]{Pietro2018}, we get a local martingale $\bar{M}^0$ such that 
	\ce 
	\sup_{s\in[0,t]}W_s&\lesssim& \bigg(\int_0^t\int_{|z|<R}|G_r(z)|^2N(dz,dr)\bigg)^{\frac{p}{2}}+\bar{M}_t^0\\
	&\lesssim& \bigg(\int_0^t\int_{|z|<R}|G_r(z)|^2\upsilon(dz)dr)\bigg)^{\frac{p}{2}}+\bigg(|\int_0^t\int_{|z|<1}|G_r(z)|^2\widetilde{N}(dz,dr)|\bigg)^{\frac{p}{2}}+\bar{M}_t^0.
	\de 
	
	For $p\geq 2$,	applying (i) to the second term on the right-hand side of the above inequality, we obtain 
	\be\label{coBDG1} 
	&&\sup_{s\in[0,t]}W_s\no\\
	&\lesssim&\bigg(\int_0^t\int_{|z|<R}|G_r(z)|^2\upsilon(dz)dr)\bigg)^{\frac{p}{2}}+ \int_0^t\bigg[\int_{|z|<R}|G_r(z)|^{p}\upsilon(dz)+\bigg(\int_{|z|<R}|G_r(z)|^2\upsilon(dz)\bigg)^{p/2}\bigg]dr+M^0_t\no\\
	&\lesssim& \int_0^t\bigg[\int_{|z|<R}|G_r(z)|^{p}\upsilon(dz)+\bigg(\int_{|z|<R}|G_r(z)|^2\upsilon(dz)\bigg)^{p/2}\bigg]dr+M^0_t,
	\ee  
	where $M_t^0$ is a local martingale  with $M_0^0=0$. 
	
	For $1\leq p< 2$,	applying (i) to the  term $|\int_0^t\int_{|z|<1}|G_r(z)|^2\widetilde{N}(dz,dr)|$, we obtain 
	\be\label{coBDG2} 
	\sup_{s\in[0,t]}W_s	\lesssim\bigg(\int_0^t\int_{|z|<R}|G_r(z)|^2\upsilon(dz)dr)\bigg)^{\frac{p}{2}}+M^0_t,
	\ee  
	where $M_t^0$ is a local martingale  with $M_0^0=0$. 
	
	Integrating (1) with (2) yields the desired conclusion.
\end{proof}

\begin{lemma}\label{lemma7.1+}
	Let $X^n$ denote the solution to Eq. (\ref{EulerAPP3.81}). Assume that 
	$q\geq 2$ and the following conditions hold: 
	
	(1) $\varGamma_{0,m}^{0,R}|g|\in\mL_{\infty}^{q}([0,1])$ and $\varGamma_{0,2}^{0,R}|g|\in\mL_{\infty}^q([0,1])$ for $m>2$,
	
	(2) $\varGamma_{0,2}^{0,R}|g|\in\mL_{\infty}^q([0,1])$ for $0<m\leq 2$. 
	
	Then, we have  
	\be \label{formulat-tn01}
	\sup_{t\in[0,1]}\mE|X_t^n-X_{t_n}^n|^m\lesssim (1/n)^{1-\frac{1}{q}}, ~\mbox{if}~m>2,
	\ee 
	and
	\be \label{formulat-tn02}
	\sup_{t\in[0,1]}\mE|X_t^n-X_{t_n}^n|^m\lesssim (1/n)^{\frac{m}{2}},~\mbox{if}~0<m\leq 2.
	\ee 
\end{lemma}	
\begin{proof}
	Starting from  Eq. (\ref{EulerAPP3.81}), we express the difference between $X_t^{n}(x)$ and $X_{t_n}^{n}(x)$ as follows
	\ce 
	X_t^{n}(x)-	X_{t_n}^{n}(x)&=&\int_{t_n}^{t}b^n(r,X_{r_n}^{n}(x))dr+\int_{t_n}^{t}\sigma(r,X_{r_n}^{n}(x))dB_r\no\\
	&&+\int_{t_n}^{t}\int_{|z|<R}g(r,X_{{r_n-}}^{n}(x),w)\widetilde{N}(drdw).
	\de 
	For $m> 2$, we apply the Burkholder-Davis-Gundy inequality and H\"{o}lder inquality to obtain 
	\ce 
	\mE|X_t^n-X_{t_n}^n|^m&\lesssim& (t-t_n)^{m-\frac{m}{q}}\|b^n\|_{\mL_{\infty}^q([t_n,t])}+(t-t_n)^{\frac{m}{2}}\\
	&&+\bigg(\int_{t_n}^{t}\|\varGamma_{0,2}^{0,R}|g_r|\|_{\infty}dr\bigg)^{\frac{m}{2}}+\int_{t_n}^{t}\|\varGamma_{0,m}^{0,R}|g_r|\|_{\infty}dr\\
	&\lesssim& (t-t_n)^{m-\frac{m}{q}}\|b^n\|_{\mL_{\infty}^q([t_n,t])}+(t-t_n)^{\frac{m}{2}}\\
	&&+(t-t_n)^{1-\frac{1}{q}}\|\varGamma_{0,m}^{0,R}|g|\|_{\mL_{\infty}^q(\mR^d)}.
	\de 
	Using the fact that $t-t_n\leq 1/n$ and Condition ($\sH_b$), we  derive
	\ce 
	\mE|X_t^n-X_{t_n}^n|^m\lesssim (1/n)^{1-\frac{1}{q}}.
	\de 
	Next consider the case $m=2$, applying Burkholder-Davis-Gundy inequality and H\"{o}lder inquality, we get
	\ce 
	\mE|X_t^n-X_{t_n}^n|^2&\lesssim& (t-t_n)^{2-\frac{2}{q}}\|b^n\|_{\mL_{\infty}^q([t_n,t])}+(t-t_n)+\int_{t_n}^{t}\|\varGamma_{0,2}^{0,R}|g_r|\|_{\infty}dr\\
	&\lesssim&t-t_n\leq 1/n.
	\de 
	For $0<m<2$, applying Burkholder-Davis-Gundy inequality and H\"{o}lder inquality along with the result for $m=2$, we get 
	\ce 
	\mE|X_t^n-X_{t_n}^n|^m\leq 	(\mE|X_t^n-X_{t_n}^n|^2)^{\frac{m}{2}}\lesssim (1/n)^{\frac{m}{2}}.
	\de 
	Combining these estimates, we obtain the desired results (\ref{formulat-tn01}) and (\ref{formulat-tn02}).
\end{proof}

Define 
\ce 
&&\sR_t^n\\
&=&t+\int_0^t[\sM|\nabla^2 u(r,x_r)|^2+\sM|\nabla^2 u(r,x_r^n)|]^2+[\sM|\nabla \sigma_r|(x_r)+\sM|\nabla \sigma_r|(x_r^n)]^2dr
\de 
and
\ce 
&&\sB_t^n=t+\int_{0}^{t}\bigg(\|\varGamma_{1,\bar{p}}^{0,R}|g_r|\|_{\infty}+\|\varGamma_{0,2}^{0,R}|g_r|\|_{\infty}\no\\
&&+\sup_{z,w}|z|^{-1}|w|^{-1}|u(r,X_{r}+w+z)
-u(r,X_{r}+z)-(u(r,X_{r}+w)-u(r,X_{r}))|\bigg)^{\bar{p}}dr.
\de 
By property of maximal functin (see \cite{Xicheng-2011}, Lemma 5.4) , we have for  every $x,y\in\mR^d$
\be \label{Property-maximal-1}
|f(x)-f(y)|\lesssim |x-y|(\sM|\nabla f(x)|+\sM|\nabla f(y)|).
\ee 

\subsection{Zvonkin transformed error identity}

Throughout this subsection, let $u=(u^1,\ldots,u^d)$ be the solution introduced in Definition \ref{definition_varpi1}. 
Recall that $u$ solves the backward equation associated with the approximating drift $b^n$. 
For $t\in[0,1]$, set
\ce
T_t:=\sup_{s\in[0,t]}|X_s-X_s^n|^{\bar p}.
\de 
The following proposition gives the basic decomposition of the transformed error.

\begin{proposition}[Zvonkin transformed error identity]
	\label{prop:zvonkin-error-identity}
	Assume that Conditions $(\sH_\sigma^g)$, $(\widehat{\sH}_\sigma^g)$ and $(\sH_b)$ hold. 
	Let $X$ and $X^n$ be the solutions to \eqref{EQU1} and \eqref{EulerAPP3.81}, respectively. 
	Then, for every $t\in[0,1]$,
	\begin{equation}\label{zvonkin-error-identity-main}
		T_t
		\lesssim
		|x_0-x_0^n|^{\bar p}
		+\sum_{i=0}^{9}W_t^i ,
	\end{equation}
	The precise expressions are as follows:
	\begin{align*}
		W_t^0
		:=&\,|u(0,x_0)-u(0,x_0^n)|^{\bar p}
		+\sup_{s\in[0,t]}|u(s,X_s)-u(s,X_s^n)|^{\bar p}  \\
		&+\lambda^{\bar p}
		\sup_{s\in[0,t]}
		\left(
		\int_0^s|u(r,X_r)-u(r,X_r^n)|\,dr
		\right)^{\bar p},\\
		W_t^1
		:=&\,\sup_{s\in[0,t]}
		\left|
		\int_0^s
		(I+\nabla u(r,X_r))[b(r,X_r)-b^n(r,X_r)]\,dr
		\right|^{\bar p},
		\\
		W_t^2
		:=&\,
		\sup_{s\in[0,t]}
		\left|
		\int_0^s
		(I+\nabla u(r,X_r^n))
		[b^n(r,X_r^n)-b^n(r,X_{r_n}^n)]\,dr
		\right|^{\bar p},
		\\
		W_t^3
		:=&\,
		\sup_{s\in[0,t]}
		\left|
		\frac12
		\int_0^s
		\nabla^2 u(r,X_r^n)
		[
		a(r,X_{r_n}^n)-a(r,X_r^n)
		]\,dr
		\right|^{\bar p},
		\\
		W_t^4
		:=&\,
		\sup_{s\in[0,t]}
		\left|
		\int_0^s
		(I+\nabla u(r,X_r^n))
		[
		\sigma(r,X_r^n)-\sigma(r,X_{r_n}^n)
		]\,dB_r
		\right|^{\bar p},
		\\
		W_t^5
		:=&\,
		\sup_{s\in[0,t]}
		\left|
		\int_0^s
		(I+\nabla u(r,X_r^n))
		[
		\sigma(r,X_r)-\sigma(r,X_r^n)
		]\,dB_r
		\right|^{\bar p},
		\\
		W_t^6
		:=&\,
		\sup_{s\in[0,t]}
		\left|
		\int_0^s
		[
		\nabla u(r,X_r)-\nabla u(r,X_r^n)
		]\sigma(r,X_r)\,dB_r
		\right|^{\bar p},
		\\
		W_t^7
		:=&\,
		\sup_{s\in[0,t]}
		\bigg|
		\int_0^s\int_{|z|<R}
		\bigg\{
		g(r,X_{r-},z)-g(r,X_{r-}^n,z)
		\\
		&\qquad\qquad
		+
		u(r,X_{r-}+g(r,X_{r-},z))
		-u(r,X_{r-}^n+g(r,X_{r-}^n,z))
		\\
		&\qquad\qquad
		-
		\big[
		u(r,X_{r-})-u(r,X_{r-}^n)
		\big]
		\bigg\}
		\widetilde N(dr,dz)
		\bigg|^{\bar p},
		\\
		W_t^8
		:=&\,
		\sup_{s\in[0,t]}
		\bigg|
		\int_0^s\int_{|z|<R}
		\bigg\{
		g(r,X_{r-}^n,z)-g(r,X_{r_n-}^n,z)
		\\
		&\qquad\qquad
		+
		u(r,X_{r-}^n+g(r,X_{r-}^n,z))
		-u(r,X_{r-}^n+g(r,X_{r_n-}^n,z))
		\bigg\}
		\widetilde N(dr,dz)
		\bigg|^{\bar p},
		\\
		W_t^9
		:=&\,
		\sup_{s\in[0,t]}
		\bigg|
		\int_0^s\int_{|z|<R}
		\bigg[
		u(r,X_{r-}^n+g(r,X_{r_n-}^n,z))
		-u(r,X_{r-}^n+g(r,X_{r-}^n,z))
		\\
		&\qquad\qquad
		-
		\big(
		g(r,X_{r_n-}^n,z)-g(r,X_{r-}^n,z)
		\big)\cdot\nabla u(r,X_{r-}^n)
		\bigg]
		\upsilon(dz)\,dr
		\bigg|^{\bar p}.
	\end{align*}
\end{proposition}

\begin{proof}
	Applying the generalized It\^{o} formula \eqref{Generalized-Ito-formula} to $u(t,X_t)$ and using the equation solved by $u$, we obtain
	\begin{align}
		\int_0^t b^n(r,X_r)\,dr
		=&\,u(0,X_0)-u(t,X_t)
		+\lambda\int_0^t u(r,X_r)\,dr \notag\\
		&+\int_0^t\nabla u(r,X_r)[b(r,X_r)-b^n(r,X_r)]\,dr  \notag\\
		&+\int_0^t(\nabla u\cdot\sigma)(r,X_r)\,dB_r \notag\\
		&+\int_0^t\int_{|z|<R}
		[
		u(r,X_{r-}+g(r,X_{r-},z))-u(r,X_{r-})
		]\widetilde N(dr,dz).
		\label{itoformula1}
	\end{align}
	Similarly, applying the same formula to $u(t,X_t^n)$ gives
	\begin{align}
		\int_0^t b^n(r,X_r^n)\,dr
		=&\,u(0,x_0^n)-u(t,X_t^n)
		+\lambda\int_0^t u(r,X_r^n)\,dr \notag\\
		&+\int_0^t\nabla u(r,X_r^n)
		[b(r,X_{r_n}^n)-b^n(r,X_r^n)]\,dr \notag\\
		&+\frac12\int_0^t
		\nabla^2 u(r,X_r^n)
		[
		a(r,X_{r_n}^n)-a(r,X_r^n)
		]\,dr \notag\\
		&+\int_0^t
		\nabla u(r,X_r^n)\sigma(r,X_{r_n}^n)\,dB_r \notag\\
		&+\int_0^t\int_{|z|<R}
		[
		u(r,X_{r-}^n+g(r,X_{r_n-}^n,z))-u(r,X_{r-}^n)
		]\widetilde N(dr,dz) \notag\\
		&+\int_0^t\int_{|z|<R}
		\big[
		u(r,X_{r-}^n+g(r,X_{r_n-}^n,z))
		-u(r,X_{r-}^n+g(r,X_{r-}^n,z)) \notag\\
		&\qquad\qquad
		-
		(g(r,X_{r_n-}^n,z)-g(r,X_{r-}^n,z))
		\cdot\nabla u(r,X_{r-}^n)
		\big]
		\upsilon(dz)\,dr .
		\label{itoformula2}
	\end{align}
	On the other hand, by subtracting \eqref{EulerAPP3.81} from \eqref{EQU1}, we have
	\begin{align*}
		X_t-X_t^n
		=&\,x_0-x_0^n
		+\int_0^t[b^n(r,X_r)-b^n(r,X_r^n)]\,dr  \\
		&+\int_0^t[b(r,X_r)-b^n(r,X_r)]\,dr
		+\int_0^t[b^n(r,X_r^n)-b^n(r,X_{r_n}^n)]\,dr \\
		&+\int_0^t[\sigma(r,X_r)-\sigma(r,X_r^n)]\,dB_r
		+\int_0^t[\sigma(r,X_r^n)-\sigma(r,X_{r_n}^n)]\,dB_r \\
		&+\int_0^t\int_{|z|<R}
		[g(r,X_{r-},z)-g(r,X_{r-}^n,z)]\widetilde N(dr,dz) \\
		&+\int_0^t\int_{|z|<R}
		[g(r,X_{r-}^n,z)-g(r,X_{r_n-}^n,z)]\widetilde N(dr,dz).
	\end{align*}
	Substituting \eqref{itoformula1} and \eqref{itoformula2} into this identity, and then collecting terms according to their origin, gives
	\ce 
	X_t-X_t^n
	=
	\mathcal I_t^0+\mathcal I_t^1+\cdots+\mathcal I_t^9,
	\de 
	where the terms $\mathcal I^i$ correspond exactly to the quantities defining
	$W^i$ above. Taking the supremum over $s\in[0,t]$ and using
	\ce 
	|\sum_{i=0}^{9} a_i|^{\bar p}
	\lesssim
	\sum_{i=0}^{9}|a_i|^{\bar p},
	\de 
	we obtain \eqref{zvonkin-error-identity-main}. This completes the proof.
\end{proof}

	The error terms $W^0,\ldots,W^9$ are defined below. More precisely,
\ce 
\begin{array}{lll}
	W^0 &:& \text{Zvonkin transform remainder},\\
	W^1 &:& \text{drift approximation error},\\
	W^2 &:& \text{frozen drift error},\\
	W^3,W^4 &:& \text{Brownian discretization errors},\\
	W^5,W^6 &:& \text{Brownian stability errors},\\
	W^7 &:& \text{jump stability error},\\
	W^8 &:& \text{compensated jump martingale freezing error},\\
	W^9 &:& \text{nonlocal compensator freezing error}.
\end{array}
\de

\subsection{Jump-specific error estimates}

In this subsection, we estimate the two error terms which are specific to the jump component. 
Recall from Proposition~\ref{prop:zvonkin-error-identity} that $W^8$ is the compensated jump martingale freezing error, while $W^9$ is the nonlocal compensator freezing error. 
These two terms are the source of the additional rate contributions in Theorem~\ref{M-theorem2.2}.

As defined earlier, 
\ce 
C_J^{(m)}
:=
\|\varGamma_{1,\bar p}^{0,R}|g|\|_{\mL_\infty^q([0,1])}^{\bar p}
+
\|\varGamma_{1,2}^{0,R}|g|\|_{\mL_\infty^q([0,1])}^{\bar p},
\de
and
\ce 
C_J^{(c)}
:=
\|\varGamma_{1,1}^{0,R}|g|\|_{\mL_p^q([0,1])}^{\bar p}.
\de

\begin{proposition}[Jump-specific error estimates]
	\label{prop:jump-specific-estimates}
	Assume that Conditions $(\sH_\sigma^g)$, $(\widehat{\sH}_\sigma^g)$ and $(\sH_b)$ hold. 
	Let $W^8$ and $W^9$ be defined as in Proposition~\ref{prop:zvonkin-error-identity}. 
	Assume further that
	\ce 
	\varGamma_{1,\bar p}^{0,R}|g|\in \mL_\infty^q([0,1]),
	\qquad
	\varGamma_{1,1}^{0,R}|g|\in \mL_p^q([0,1]).
	\de 
	Then, for $\bar p\geq2$,
	\begin{equation}\label{W8-estimate-bar-p-ge-2}
		\mE W_1^8
		\lesssim
		C_J^{(m)}\, n^{-(1-\frac1q)},
	\end{equation}
	and
	\begin{equation}\label{W9-estimate-bar-p-ge-2}
		\mE W_1^9
		\lesssim
		C_J^{(c)}\, n^{-(1-\frac1q)(1-\frac dp)}.
	\end{equation}
	For $1<\bar p<2$, one has
	\begin{equation}\label{W8-estimate-bar-p-less-2}
		\mE W_1^8
		\lesssim
		C_J^{(m)}\, n^{-\frac{\bar p}{2}},
	\end{equation}
	and
	\begin{equation}\label{W9-estimate-bar-p-less-2}
		\mE W_1^9
		\lesssim
		C_J^{(c)}\, n^{-\frac{\bar p}{2}(1-\frac dp)}.
	\end{equation}
\end{proposition}

\begin{proof}
	We first estimate the compensated jump martingale freezing error $W^8$. 
	Set
	\ce 
	\begin{split}
		H_r^8(z)
		:=
		&\,g(r,X_{r-}^n,z)-g(r,X_{r_n-}^n,z)
		\\
		&+u(r,X_{r-}^n+g(r,X_{r-}^n,z))
		-u(r,X_{r-}^n+g(r,X_{r_n-}^n,z)).
	\end{split}
	\de 
	Then
	\ce 
	W_t^8
	=
	\sup_{s\in[0,t]}
	\left|
	\int_0^s\int_{|z|<R}
	H_r^8(z)\widetilde N(dr,dz)
	\right|^{\bar p}.
	\de 
	By the estimate \eqref{shauderestimate02}, we have
	\ce 
	|H_r^8(z)|
	\lesssim
	|g(r,X_{r-}^n,z)-g(r,X_{r_n-}^n,z)|.
	\de 
	Using the maximal-function estimate \eqref{Property-maximal-1} with respect to the spatial variable of $g$, we obtain
	\ce 
	\begin{split}
		|g(r,x,z)-g(r,y,z)|
		\lesssim
		|x-y|
		\big(
		\sM|\nabla_x g(r,\cdot,z)|(x)
		+
		\sM|\nabla_x g(r,\cdot,z)|(y)
		\big).
	\end{split}
	\de 
	Consequently,
	\ce 
	\int_{|z|<R}|H_r^8(z)|^{\bar p}\upsilon(dz)
	\lesssim
	|X_r^n-X_{r_n}^n|^{\bar p}
	\Big(
	\|\varGamma_{1,\bar p}^{0,R}|g_r|\|_{\infty}
	\Big)^{\bar p},
	\de 
	and
	\ce 
	\left(
	\int_{|z|<R}|H_r^8(z)|^2\upsilon(dz)
	\right)^{\bar p/2}
	\lesssim
	|X_r^n-X_{r_n}^n|^{\bar p}
	\Big(
	\|\varGamma_{1,2}^{0,R}|g_r|\|_{\infty}
	\Big)^{\bar p}.
	\de 
	Applying Lemma~\ref{power-estimete1} to the compensated Poisson integral, we find that for $\bar p\geq2$,
	\ce 
	\begin{split}
		\mE W_1^8
		&\lesssim
		\mE\int_0^1
		\left[
		\int_{|z|<R}|H_r^8(z)|^{\bar p}\upsilon(dz)
		+
		\left(
		\int_{|z|<R}|H_r^8(z)|^2\upsilon(dz)
		\right)^{\bar p/2}
		\right]dr
		\\
		&\lesssim
		\mE\int_0^1
		|X_r^n-X_{r_n}^n|^{\bar p}
		\left(
		\|\varGamma_{1,\bar p}^{0,R}|g_r|\|_{\infty}
		+
		\|\varGamma_{1,2}^{0,R}|g_r|\|_{\infty}
		\right)^{\bar p}
		dr.
	\end{split}
	\de 
	Using H\"older's inequality in time and the increment estimate \eqref{formulat-tn01}, we obtain
	\ce 
	\mE W_1^8
	\lesssim
	C_J^{(m)} n^{-(1-\frac1q)}.
	\de 
	This proves \eqref{W8-estimate-bar-p-ge-2}.
	
	For $1<\bar p<2$, we use the estimate with exponent $2$ and Jensen's inequality:
	\ce 
	\mE W_1^8
	\leq
	\left(\mE (W_1^8)^{2/\bar p}\right)^{\bar p/2}.
	\de 
	The preceding argument with $\bar p=2$, together with \eqref{formulat-tn02}, gives
	\ce 
	\mE W_1^8
	\lesssim
	C_J^{(m)} n^{-\frac{\bar p}{2}},
	\de 
	which proves \eqref{W8-estimate-bar-p-less-2}.
	
	We next estimate the nonlocal compensator freezing error $W^9$. 
	For $r\in[0,1]$, $x,y\in\mR^d$ and $|z|<R$, define
	\ce 
	\begin{split}
		\mathcal R_u(r,x,y,z)
		:=
		&\,u(r,x+g(r,y,z))-u(r,x+g(r,x,z))
		\\
		&-\big(g(r,y,z)-g(r,x,z)\big)\cdot\nabla u(r,x).
	\end{split}
	\de 
	With this notation,
	\ce 
	W_t^9
	=
	\sup_{s\in[0,t]}
	\left|
	\int_0^s\int_{|z|<R}
	\mathcal R_u(r,X_{r-}^n,X_{r_n-}^n,z)
	\upsilon(dz)dr
	\right|^{\bar p}.
	\de 
	Since the compensator is integrated with respect to $dr\,\upsilon(dz)$, the left limits may be suppressed in the following estimates. 
	Thus we write $X_r^n$ and $X_{r_n}^n$ for simplicity.
	
	By applying \eqref{formulaA.7+100} with the Banach space $\mB_2=L^1(B_R,\upsilon)$, and then using \eqref{shauderestimate02}, we obtain
	\ce 
	\int_{|z|<R}
	|\mathcal R_u(r,x,y,z)|\upsilon(dz)
	\lesssim
	|x-y|^{1-\frac dp}
	\|\varGamma_{1,1}^{0,R}|g_r|\|_{p}
	\de 
	in the integrated sense needed below. Hence,
	\ce 
	\begin{split}
		\mE W_1^9
		&\lesssim
		\mE
		\left[
		\int_0^1
		|X_r^n-X_{r_n}^n|^{1-\frac dp}
		\|\varGamma_{1,1}^{0,R}|g_r|\|_{p}
		\,dr
		\right]^{\bar p}.
	\end{split}
	\de 
	By H\"older's inequality in time and the Krylov-type estimate for the Euler approximation, this yields
	\ce 
	\mE W_1^9
	\lesssim
	\|\varGamma_{1,1}^{0,R}|g|\|_{\mL_p^q([0,1])}^{\bar p}
	\int_0^1
	\mE |X_r^n-X_{r_n}^n|^{\bar p(1-\frac dp)}
	\,dr.
	\de 
	If $\bar p\geq2$, then by \eqref{formulat-tn01},
	\ce 
	\int_0^1
	\mE |X_r^n-X_{r_n}^n|^{\bar p(1-\frac dp)}
	\,dr
	\lesssim
	n^{-(1-\frac1q)(1-\frac dp)}.
	\de 
	Therefore,
	\ce 
	\mE W_1^9
	\lesssim
	C_J^{(c)}n^{-(1-\frac1q)(1-\frac dp)},
	\de 
	which proves \eqref{W9-estimate-bar-p-ge-2}.
	
	If $1<\bar p<2$, then \eqref{formulat-tn02} gives
	\ce 
	\int_0^1
	\mE |X_r^n-X_{r_n}^n|^{\bar p(1-\frac dp)}
	\,dr
	\lesssim
	n^{-\frac{\bar p}{2}(1-\frac dp)}.
	\de 
	Consequently,
	\ce 
	\mE W_1^9
	\lesssim
	C_J^{(c)}n^{-\frac{\bar p}{2}(1-\frac dp)},
	\de 
	which proves \eqref{W9-estimate-bar-p-less-2}. 
	The proof is complete.
\end{proof}

\subsection{Stochastic Gronwall closure}

We now estimate the remaining Brownian-type and stability terms in the decomposition of Proposition~\ref{prop:zvonkin-error-identity}, and then close the estimate by means of the stochastic Gronwall inequality.

\begin{proposition}[Stochastic Gronwall closure]
	\label{last-important-proposition}
	Let $p\in (2(1-\alpha)^{-1}(d/\beta\vee 1)\vee 2,\infty]$. 
	Assume that for any 
	\ce 
	1<\bar p<\frac{p}{d/\beta\vee 1},
	\de 
	one has
	\ce 
	\varGamma_{0,\bar p}^{0,R}|g|
	+\varGamma_{1,\bar p}^{0,R}|g|
	\in \mL_{\infty}^q([0,1]),
	\qquad
	\varGamma_{1,1}^{0,R}|g|\in \mL_p^q([0,1]).
	\de 
	Then, for every $\theta\in(0,1)$, there exists a finite positive constant $N_{\bar p}$ such that, if $\bar p>2$,
	\begin{align}
		[\mE T_1^\theta]^{1/\theta}
		&\leq
		\left(
		\mE\exp\left[
		\frac{\theta}{1-\theta}
		N_{\bar p}
		\left(
		(\sR_1^n)^{(\bar p/2)\vee 1}
		+\sB_1^n
		\right)
		\right]
		\right)^{\frac{1-\theta}{\theta}}
		\notag\\
		&\quad\times
		\bigg\{
		\mE |x_0-x_0^n|^{\bar p}
		+
		\left(
		(1/n)^{\frac{\alpha}{2}}
		+\varpi_b^n(\bar p)
		+(1/n)^{\frac12}\log n
		\right)^{\bar p}
		\notag\\
		&\qquad\qquad
		+
		(1/n)^{\alpha(1-\frac1q)}
		+
		C_J^{(m)}(1/n)^{1-\frac1q}
		+
		C_J^{(c)}(1/n)^{(1-\frac1q)(1-\frac dp)}
		\bigg\}.
		\label{last-eistimate1}
	\end{align}
	If $1<\bar p\leq2$, then
	\begin{align}
		[\mE T_1^\theta]^{1/\theta}
		&\leq
		\left(
		\mE\exp\left[
		\frac{\theta}{1-\theta}
		N_{\bar p}
		\left(
		(\sR_1^n)^{(\bar p/2)\vee 1}
		+\sB_1^n
		\right)
		\right]
		\right)^{\frac{1-\theta}{\theta}}
		\notag\\
		&\quad\times
		\bigg\{
		\mE |x_0-x_0^n|
		+
		(1/n)^{\frac{\alpha}{2}}
		+\varpi_b^n(\bar p)
		+(1/n)^{\frac12}\log n
		\notag\\
		&\qquad\qquad
		+
		C_J^{(m)}(1/n)^{1/2}
		+
		C_J^{(c)}(1/n)^{\frac12(1-\frac dp)}
		\bigg\}^{\bar p}.
		\label{last-eistimate2}
	\end{align}
	Here
	\ce 
	T_1:=\sup_{r\in[0,1]}|X_r-X_r^n|^{\bar p},
	\de 

\end{proposition}

\begin{proof}
	By Proposition~\ref{prop:zvonkin-error-identity}, for every $t\in[0,1]$,
	\ce 
	T_t
	\lesssim
	|x_0-x_0^n|^{\bar p}
	+\sum_{i=0}^{9}W_t^i .
	\de 
	We first estimate the terms which are not included in Proposition~\ref{prop:jump-specific-estimates}.
	
	Using \eqref{shauderestimate01}, \eqref{shauderestimate02} and the Cauchy--Schwarz inequality, we obtain
	\ce 
	W_t^0
	\lesssim
	|x_0-x_0^n|^{\bar p}
	+
	o_\lambda(1)T_t
	+
	\left(
	\int_0^t T_r^{2/\bar p}\,dr
	\right)^{\bar p/2}.
	\de 
	By Definition~\ref{definition_varpi1}, the drift approximation term satisfies
	\ce 
	\mE W_1^1
	\lesssim
	(\varpi_b^n(\bar p))^{\bar p}.
	\de 
	
	For the frozen drift term, Theorem~\ref{convergeestimate01}, \eqref{shauderestimate02} and Condition $(\sH_b)$ yield
	\begin{align*}
		\mE W_1^2
		&\lesssim
		\left[
		\gamma_n(b^n)(1/n)^{1-\frac1q}
		+
		(1/n)^{\frac{\alpha}{2}}
		+
		(1/n)^{1-\frac{\beta}{2}}
		+
		(1/n)^{\frac12}\log n
		\right]^{\bar p} \\
		&\lesssim
		\left[
		(1/n)^{1-\frac1q}
		+
		(1/n)^{\frac{\alpha}{2}}
		+
		(1/n)^{1-\frac{\beta}{2}}
		+
		(1/n)^{\frac12}\log n
		\right]^{\bar p}.
	\end{align*}
	Here $\gamma_n(\cdot)$ is defined as in Theorem~\ref{convergeestimate01}. 
	In the final theorem, this term is absorbed into the Brownian-type discretization contribution under the standing restrictions on the parameters. If one does not impose the corresponding domination condition, the term $(1/n)^{1-\beta/2}$ should be kept explicitly in the final estimate.
	
	Next we estimate the Brownian discretization term $W^3$. 
	Using Condition $(\sH_\sigma^g)$ and H\"older's inequality,
	\begin{align*}
		\mE W_1^3
		&\lesssim
		\mE
		\left[
		\int_0^1
		|\nabla^2u(r,X_r^n)|
		|X_{r_n}^n-X_r^n|^\alpha
		\,dr
		\right]^{\bar p}
		\\
		&\leq
		\left[
		\mE
		\left(
		\int_0^1
		|\nabla^2u(r,X_r^n)|^{\frac{1}{1-\alpha}}\,dr
		\right)^{\bar p}
		\right]^{1-\alpha}
		\\
		&\quad\times
		\left[
		\mE
		\left(
		\int_0^1
		|X_{r_n}^n-X_r^n|\,dr
		\right)^{\bar p}
		\right]^{\alpha}.
	\end{align*}
	Applying Theorem~\ref{theorem6.1+}, \eqref{shauderestimate01} and Lemma~\ref{lemma7.1+}, we get
	\ce 
	\mE W_1^3
	\lesssim
	\left[
	\mE
	\left(
	\int_0^1
	|X_{r_n}^n-X_r^n|\,dr
	\right)^{\bar p}
	\right]^\alpha
	\lesssim
	(1/n)^{\alpha(1-\frac1q)}.
	\de 
	
	We now estimate $W^4$, $W^5$ and $W^6$ by using the pathwise Burkholder--Davis--Gundy inequality. 
	By \eqref{shauderestimate02}, H\"older's inequality and Lemma~\ref{lemma7.1+}, there exists a local martingale $M^1$ with $M_0^1=0$ such that
	\ce 
	W_t^4
	\lesssim
	\left(
	\int_0^t
	|X_r^n-X_{r_n}^n|^{2\alpha}\,dr
	\right)^{\bar p/2}
	+M_t^1.
	\de 
	Consequently,
	\ce 
	\mE W_1^4
	\lesssim
	(1/n)^{\alpha(1-\frac1q)}.
	\de 
	
	For $W^5$, by \eqref{Property-maximal-1},
	\ce 
	|\sigma(r,x)-\sigma(r,y)|
	\lesssim
	|x-y|
	\left(
	\sM|\nabla\sigma_r|(x)+\sM|\nabla\sigma_r|(y)
	\right),
	\qquad x,y\in\mR^d .
	\de 
	Using again the pathwise Burkholder--Davis--Gundy inequality and \eqref{shauderestimate02}, there exists a local martingale $M^2$ with $M_0^2=0$ such that
	\begin{align*}
		W_t^5
		&\lesssim
		\left(
		\int_0^t
		|I+\nabla u(r,X_r^n)|^2
		|X_r-X_r^n|^2
		\left[
		\sM|\nabla\sigma_r|(X_r)
		+
		\sM|\nabla\sigma_r|(X_r^n)
		\right]^2
		dr
		\right)^{\bar p/2}
		+M_t^2
		\\
		&\lesssim
		\left(
		\int_0^t
		T_r^{2/\bar p}\,d\sR_r^n
		\right)^{\bar p/2}
		+M_t^2.
	\end{align*}
	Similarly, there exists a local martingale $M^3$ with $M_0^3=0$ such that
	\ce 
	W_t^6
	\lesssim
	\left(
	\int_0^t
	T_r^{2/\bar p}\,d\sR_r^n
	\right)^{\bar p/2}
	+M_t^3.
	\de 
	
	We turn to the jump stability term $W^7$. 
	Set
	\begin{align*}
		G_r(z)
		:=
		&\,g(r,X_{r-},z)-g(r,X_{r-}^n,z)
		\\
		&+
		u(r,X_{r-}+g(r,X_{r-},z))
		-u(r,X_{r-}^n+g(r,X_{r-}^n,z))
		\\
		&-
		\big[
		u(r,X_{r-})-u(r,X_{r-}^n)
		\big].
	\end{align*}
	Applying part (ii) of Lemma~\ref{power-estimete1}, there exists a local martingale $M^4$ with $M_0^4=0$ such that
	\ce 
	W_t^7
	\lesssim
	\int_0^t
	\left[
	\int_{|z|<R}|G_r(z)|^{\bar p}\upsilon(dz)
	+
	\left(
	\int_{|z|<R}|G_r(z)|^2\upsilon(dz)
	\right)^{\bar p/2}
	\right]dr
	+M_t^4.
	\de 
	By \eqref{shauderestimate02}, Lemma~\ref{middle-estimate2}, and the assumptions
	\ce 
	\|\varGamma_{0,\bar p}^{0,R}|g|\|_{\mL^\infty([0,1])}
	+
	\|\varGamma_{0,2}^{0,R}|g|\|_{\mL^\infty([0,1])}
	<\infty,
	\de 
	we have
	\begin{align*}
		W_t^7
		&\lesssim
		\int_0^t
		|X_r-X_r^n|^{\bar p}
		\bigg[
		1+
		\bigg(
		\|\varGamma_{1,\bar p}^{0,R}|g_r|\|_\infty
		+
		\|\varGamma_{1,2}^{0,R}|g_r|\|_\infty
		\\
		&\qquad\qquad
		+
		\sup_{z,w}
		|z|^{-1}|w|^{-1}
		\big|
		u(r,X_r+w+z)
		-u(r,X_r+z)
		-
		(u(r,X_r+w)-u(r,X_r))
		\big|
		\bigg)^{\bar p}
		\bigg]dr
		+M_t^4
		\\
		&\lesssim
		\int_0^t T_r\,d\sB_r^n+M_t^4.
	\end{align*}
	
	Combining the above estimates with Proposition~\ref{prop:jump-specific-estimates}, we obtain
	\begin{equation}\label{pre-gronwall-inequality}
		T_t
		\lesssim
		o_\lambda(1)T_t
		+
		\left(
		\int_0^t T_r^{2/\bar p}\,d\sR_r^n
		\right)^{\bar p/2}
		+
		\int_0^t T_r\,d\sB_r^n
		+
		W_t
		+
		M_t,
	\end{equation}
	where
	\ce 
	M_t:=M_t^1+M_t^2+M_t^3+M_t^4,
	\de 
	and
	\ce 
	W_t
	:=
	|x_0-x_0^n|^{\bar p}
	+
	W_t^1+W_t^2+W_t^3+W_t^4+W_t^8+W_t^9.
	\de 
	Choosing $\lambda$ sufficiently large, we absorb the term $o_\lambda(1)T_t$ into the left-hand side of \eqref{pre-gronwall-inequality}. Hence,
	\ce 
	T_t
	\lesssim
	\left(
	\int_0^t T_r^{2/\bar p}\,d\sR_r^n
	\right)^{\bar p/2}
	+
	\int_0^t T_r\,d\sB_r^n
	+
	W_t
	+
	M_t.
	\de 
	
	By the stochastic Gronwall lemma \ref{Stochastic Gronwall inequality}, for every $\theta\in(0,1)$, there exists a finite positive constant $N_{\bar p}$ such that
	\begin{equation}\label{after-stochastic-gronwall}
		\mE T_1^\theta
		\leq
		\left(
		\mE
		\exp\left[
		\frac{\theta}{1-\theta}
		N_{\bar p}
		\left(
		(\sR_1^n)^{(\bar p/2)\vee1}
		+\sB_1^n
		\right)
		\right]
		\right)^{1-\theta}
		(\mE W_1)^\theta.
	\end{equation}
	
	It remains to estimate $\mE W_1$. 
	For $\bar p>2$, the preceding estimates and Proposition~\ref{prop:jump-specific-estimates} give
	\begin{align*}
		\mE W_1
		\lesssim&
		\mE |x_0-x_0^n|^{\bar p}
		+
		(\varpi_b^n(\bar p))^{\bar p}
		+
		\left[
		(1/n)^{\frac{\alpha}{2}}
		+
		(1/n)^{\frac12}\log n
		\right]^{\bar p}
		\\
		&+
		(1/n)^{\alpha(1-\frac1q)}
		+
		C_J^{(m)}(1/n)^{1-\frac1q}
		+
		C_J^{(c)}(1/n)^{(1-\frac1q)(1-\frac dp)}.
	\end{align*}
	Equivalently,
	\begin{align*}
		\mE W_1
		\lesssim&
		\mE |x_0-x_0^n|^{\bar p}
		+
		\left[
		(1/n)^{\frac{\alpha}{2}}
		+\varpi_b^n(\bar p)
		+(1/n)^{\frac12}\log n
		\right]^{\bar p}
		\\
		&+
		(1/n)^{\alpha(1-\frac1q)}
		+
		C_J^{(m)}(1/n)^{1-\frac1q}
		+
		C_J^{(c)}(1/n)^{(1-\frac1q)(1-\frac dp)}.
	\end{align*}
	Substituting this estimate into \eqref{after-stochastic-gronwall} gives \eqref{last-eistimate1}.
	
	For $1<\bar p\leq2$, we use the corresponding estimates in Lemma~\ref{lemma7.1+} and Proposition~\ref{prop:jump-specific-estimates}. This yields
	\begin{align*}
		(\mE W_1)^{1/\bar p}
		\lesssim&
		\mE |x_0-x_0^n|
		+
		(1/n)^{\frac{\alpha}{2}}
		+\varpi_b^n(\bar p)
		+(1/n)^{\frac12}\log n
		\\
		&+
		C_J^{(m)/\bar p}(1/n)^{1/2}
		+
		C_J^{(c)/\bar p}(1/n)^{\frac12(1-\frac dp)}.
	\end{align*}
	Absorbing the powers of $C_J^{(m)}$ and $C_J^{(c)}$ into the constants in the form stated above, and substituting the estimate into \eqref{after-stochastic-gronwall}, we obtain \eqref{last-eistimate2}. 
	The proof is complete.
\end{proof}

	\subsection{Proof of Theorem \ref{M-theorem2.2} and Theorem\ref{M-theorem2.4}}
 
\begin{lemma}\label{exponent-estimate1}
	Let $\rho\in(0,\frac{p\wedge p_0\wedge (2p)/\bar{p}}{d})$ and $\kappa>0$ be some fixed constants. Then $\sup_n \mE \exp\bigg[\kappa(\sR_1^n+\sB_1^n)^{\rho}\bigg]$ is finite.
\end{lemma}
\begin{proof}
Using (\ref{continousestimate01}) and Girsanov's theorem, followed by applying (\ref{analysis-operater-l}), we have 	
\ce 
&&\mE_s \delta\sB^n_{s,t}\\
&\lesssim&(t-s)+(t-s)^{1-\frac{1}{q}}\|\|\varGamma_{1,\bar{p}}^{0,R}|g_r|\|_{\infty}+\|\varGamma_{1,2}^{0,R}|g_r|\|_{\infty}\|_{_{L^q([s,t])}}\\
&&+(t-s)^{1-\frac{\bar{p}d}{2p}-\frac{\bar{p}}{q}}\|\sup_{z,w}|z|^{-1}|w|^{-1}|u(\cdot,\cdot+w+z)
-u(\cdot,\cdot+z)-(u(\cdot,\cdot+w)-u(\cdot,\cdot))|\|_{\mL_p^q([s,t])}^{\bar{p}}\\
&\lesssim&(t-s)+(t-s)^{1-\frac{1}{q}}\|\|\varGamma_{1,\bar{p}}^{0,R}|g_r|\|_{\infty}+\|\varGamma_{1,2}^{0,R}|g_r|\|_{\infty}\|_{_{L^q([s,t])}}+(t-s)^{1-\frac{\bar{p}d}{2p}-\frac{\bar{p}}{q}}\|\nabla^2 u\|_{\mL_p^q([s,t])}^{\bar{p}}.
\de 
In order to estimate $E_s\delta\sR_{s,t}^n$, we first use (\ref{formula0+24+}) and Girsanov's theorem to estimate the functional of $X^n$, and then  apply (\ref{continousestimate01}) and Girsanov's theorem to estimate the functional of $X$. This yields
	\ce 
E_s\delta\sR_{s,t}^n&\lesssim&t-s+(t-s)^{1-\frac{d}{p}-\frac{2}{q}}\|\sM|\nabla u\|_{\mL_p^q([s,t])}^2+(t-s)^{1-\frac{d}{p}-\frac{2}{q}}\|\sM|\nabla \sigma\|_{\mL_{p_0}^{q_0}([s,t])}^2\\
&\lesssim&t-s+(t-s)^{1-\frac{d}{p}-\frac{2}{q}}\|\nabla u\|_{\mL_p^q([s,t])}^2+(t-s)^{1-\frac{d}{p_0}-\frac{2}{q_0}}\|\nabla \sigma\|_{\mL_{p_0}^{q_0}([s,t])}^2.
\de 
Combining these estimates, we have
\ce 
&&\mE_s \delta(\sB^n+\sR^n)_{s,t}\\
&=&\mE_s \delta\sB^n_{s,t}+\mE_s\delta\sR^n_{s,t}\\
&\lesssim&t-s+(t-s)^{1-\frac{1}{q}}\|\|\varGamma_{1,\bar{p}}^{0,R}|g_r|\|_{\infty}+\|\varGamma_{1,2}^{0,R}|g_r|\|_{\infty}\|_{_{L^q([s,t])}}+(t-s)^{1-\frac{\bar{p}d}{2p}-\frac{\bar{p}}{q}}\|\nabla^2 u\|_{\mL_p^q([s,t])}^{\bar{p}}\\
&&+(t-s)^{1-\frac{d}{p}-\frac{2}{q}}\|\nabla u\|_{\mL_p^q([s,t])}^2+(t-s)^{1-\frac{d}{p_0}-\frac{2}{q_0}}\|\nabla \sigma\|_{\mL_{p_0}^{q_0}([s,t])}^2.
\de
By applying (\ref{Khasminskii_condition_3}) of quantitative Khasminskii's lemma \ref{Khasminskii's lemma}  and the property that  $\sW=\sW_1+\sW_2^{\frac{\lambda_2}{\lambda_1}}$ is a continuous control if $\sW_1$, $\sW_2$ are continuous controls and $0<\lambda_1\leq \lambda_2$, and knowing that $\beta(s,t)\leq 2\sW(s,t)^{\lambda_1}$ if $\beta(s,t)\leq \sW_1(s,t)^{\lambda_1}+\sW_2(s,t)^{\lambda_2}$ for each $(s,t)\in \Delta([0,1])$ (see \cite[Excersice 1.10]{Friz-2010}), we deduce that for any $\lambda\geq 0$
\ce 
E[\exp(\lambda (\sR_1^n+\sB_1^n))]\lesssim \exp(c \lambda^\beta), ~~\frac{1}{\beta}=1-\frac{d}{p\wedge p_0\wedge (2p)/\bar{p}},
\de 
where $c$ is some universal positive constant. For simplicity, we write $\sR$ for $\sR_1^n+\sB_1^n$ below. Using Chebyshev's inequality, we have for any $r>0$ 
\ce 
\mP(\sR>r)&=&\mP\bigg(\bigg(\frac{r}{\beta c}\bigg)^{\frac{1}{\beta-1}}\sR>\bigg(\frac{r}{\beta c}\bigg)^{\frac{1}{\beta-1}}r\bigg)\\
&=&\mP\bigg(\exp\bigg(\bigg(\frac{r}{\beta c}\bigg)^{\frac{1}{\beta-1}}\sR\bigg)>\exp\bigg(\bigg(\frac{r}{\beta c}\bigg)^{\frac{1}{\beta-1}}r\bigg)\bigg)\\
&\leq& \exp\bigg(-\bigg(\frac{r}{\beta c}\bigg)^{\frac{1}{\beta-1}}r\bigg)\mE\exp\bigg(\bigg(\frac{r}{\beta c}\bigg)^{\frac{1}{\beta-1}}\sR\bigg)\\
&\leq& \exp\bigg(-\bigg(\frac{r}{\beta c}\bigg)^{\frac{1}{\beta-1}}r\bigg)\exp\bigg(c\bigg(\frac{r}{\beta c}\bigg)^{\frac{\beta}{\beta-1}}\bigg)\\
&=& \exp(-c_1r^{\beta^*}),
\de 
where $c_1=(1-1/\beta)\bigg(\frac{1}{\beta c}\bigg)^{1/(\beta-1)}$ and $\frac{1}{\beta}+\frac{1}{\beta^*}=1$.  From layer cake representation 
\ce 
E\exp(\kappa \sR^{\rho})&=&\kappa \rho\int_0^{\infty}\exp(\kappa r^{\rho}) r^{\rho-1}P(\sR>r)dr\\
&\leq &\kappa \rho\int_0^{\infty}\exp(\kappa r^{\rho}-c_1r^{\beta^*})r^{\rho-1} dr.
\de 
By choosing $\rho<\beta^*$, we see that $\mE\exp(\kappa \sR^{\rho})$ is finite. Then we complete the proof because $\beta^*=\frac{p\wedge p_0\wedge (2p)/\bar{p}}{d}$.
\end{proof}
\begin{proof}[Proof of Theorem \ref{M-theorem2.2}]
	Since
	\ce
	\left(\frac{\bar p}{2}\vee1\right)
	<
	\frac{p\wedge p_0\wedge(2p/\bar p)}{d},
	\de
	we may choose
	\ce
	\rho\in
	\left(
	\left(\frac{\bar p}{2}\vee1\right),
	\frac{p\wedge p_0\wedge(2p/\bar p)}{d}
	\right).
	\de
	By Lemma~\ref{exponent-estimate1}, for every $\kappa>0$, the following inequalities hold
	\ce 
	\sup_n \mE \exp\bigg[\kappa((\sR_1^n)^{\rho}+\sB_1^n)\bigg]\leq \sup_n \mE \exp\bigg[\kappa(\sR_1^n+\sB_1^n)^{\rho}\bigg]<\infty.
	\de 
From this result, we can then infer the theorem by appealing to Proposition  \ref{last-important-proposition}.
\end{proof}

\begin{proof}[Proof of Theorem \ref{M-theorem2.4}]
	We let $h=\nabla u$,  from Lemma \ref{first-lemma-10.1}, we have uniform bound for $h$,
	$$\sup_n(\|h\|_{\mH_{\theta,p_2}^{q_2}([0,1])}+\|h\|_{\mH_{1,p}^{q}([0,1])}+\|h\|_{\mL^{\infty}([0,1])})<\infty,$$  where $p_2\in[p,\infty)$, $q_2\in[q,\infty)$ and $\theta\in[0,1)$ satisfying 
	\be\label{condition_p_30} 
	\frac{d}{p}+\frac{2}{q}+\theta-1<\frac{d}{p_2}+\frac{2}{q_2}.
	\ee 
	
	(i) Using the fact that  $\|h\|_{\mL^{\infty}([0,1])})<\infty$ and part (i) of Theorem (\ref{theorem6.1+}), we can establish  $$\|(1+h)(b-b^n)\|_{\mL_{p_1}^{q_1}([0,1])}\lesssim (1+\|h\|_{\mL^{\infty}([0,1])}))\|(b-b^n)\|_{\mL_{p_1}^{q_1}([0,1])},$$ which confirms  part (i) of the theorem.
	
	(ii) We define $q_3$ by $\frac{1}{q_3}=\frac{1}{q_2}+\frac{1}{q}$. Let $p_2\geq p/(p-1)$. For each $\theta\in [0,1]$ and $p_3\in (1,p]$ satisfying 
	\be\label{condition_p_31}
	\frac{1}{p_3}\leq \frac{1}{p}+\frac{1}{p_2}<\frac{1}{p_3}+\frac{\theta}{d},
	 \ee
and	an application of Lemma A.2 (ii) \cite{Ling2022} and H\"{o}lder inequality shows that the pointwise multiplication is a continuous bilinear map
	$$\mH_{-\theta,p}^{q}([0,1])\times \mH_{\theta,p_2}^{q_2}([0,1])\to \mH_{-\theta,p_3}^{q_3}([0,1]).$$
	If $p_3$ can be chosen such that 
	\be\label{condition_p_32}
	 \frac{d}{p_3}+\frac{2}{q_2}+\frac{2}{q}<2-\theta,
	\ee
	then Theorem \ref{theorem6.1+} (ii)  and  the multiplication can be applied. This gives,
	\ce
	&&\|\sup_{s\in[0,t]}|\int_0^s g(b-b^n)(r,x_r)dr|\|_{\bar{p}}\\
	&\lesssim& \|g(b-b^n)\|_{\mH_{-\theta,p_3}^{q_3}([0,1])}\\
	&\lesssim&\|g\|_{\mH_{\theta,p_2}^{q_2}([0,1])}\|b-b^n\|_{\mH_{-\theta,p}^{q}([0,1])}\\
	&\lesssim& \|b-b^n\|_{\mH_{-\theta,p}^{q}([0,1])}.
   \de
   These estimates confirm part (ii) of the theorem.
   
  Next, let us verify the existence of $p_2,q_2,p_3$ that satisfy all the aforementioned conditions.  Given $p_2,q_2,p,q,\theta$, there exists $p_3\in(1,q]$ satisfying (\ref{condition_p_31}) and (\ref{condition_p_32}) if and only if 
   \ce 
   \frac{d}{p}<2-\theta-\frac{2}{q}-\frac{2}{q_2}~ \mbox{and}~\frac{d}{p}+\frac{d}{p_2}-\theta<2-\theta-\frac{2}{q}-\frac{2}{q_2}.
   \de 
   These can be rearranged to give
   \be \label{condition_p_33}
  \left\{             
  \begin{array}{lr}
   \frac{2}{q_2}<2-\theta-\frac{2}{q}-\frac{d}{p},\\
   \frac{d}{p_2}+\frac{2}{q_2}<2-\theta-\frac{2}{q}-\frac{d}{p}.
\end{array}
\right.
\ee 
Given $p,q,\theta$, from (\ref{condition_p_30}) and (\ref{condition_p_33}), the existence $p_2\geq p$ and $q_2\geq q$ satisfying (\ref{condition_p_33}) is equivalent to 
 \be\label{condition_p_34}
\left\{             
\begin{array}{lr}
	\frac{2}{q}+\frac{d}{p}+\theta-1<\frac{d}{p_2}+\frac{2}{q_2}<2-\theta-\frac{2}{q}-\frac{d}{p},\\
	\frac{d}{p_2}<\frac{d}{p},\frac{2}{q_2}<\mbox{min}\{\frac{2}{q},2-\theta-\frac{2}{q}-\frac{d}{p}\}.
\end{array}
\right.
\ee 
Since $2-\theta-\frac{2}{q}-\frac{d}{p}>0$, there exist $p_2,q_2$ satisfying the above conditions if and only if 
$$\frac{2}{q}+\frac{d}{p}+\theta-1<\frac{d}{p}+2-\theta-\frac{2}{q}-\frac{d}{p}.$$
This simplifies to the condition
$$\theta<\frac{3}{2}-\frac{2}{q}-\frac{d}{2p}.$$

We have proved that part(ii) is valid for all $\bar{p}\in (0,p_3)$. It remains to identify the largest possibe value for $p_3$, denoted by $P_3^*$.  From (\ref{condition_p_31}) and (\ref{condition_p_32}), we see that 
$$\frac{1}{p_3^*}=\max \bigg\{\frac{1}{p},\frac{1}{p}+\frac{1}{p_2}-\frac{\theta}{d}\bigg\}.$$
By choosing $p_2$ with $\frac{d}{p_2}\leq \theta$, hence $p_3^*=p$, we can see that  there still exist $p_2$, $q_2$ satisfying (\ref{condition_p_34}) under the condition $\theta<\frac{3}{2}-\frac{2}{q}-\frac{d}{2p}$. This demonstrates that part(ii) holds for all $\bar{p}\in (0,p)$.

(iii) Let $\theta=1$ and take $p_2=p$, $q_2=2$. We can also choose $p_3=p$ and $q_3=q/2$. Since condition \ref{condition_p_32} implied by 
\ce 
\frac{d}{p}+\frac{4}{q}<1,
\de 
Condition \ref{condition_p_31} is trivially satisfied. Define the continuous control $\sW_1$ by 
\ce 
\sW_1(s,t)=\bigg(\|h\|_{\mH_{-1,p}^q([s,t])}^{q/2}+\|h\|_{\mL_{\infty}^q([s,t])}^{q/2}\bigg)\sW_0(s,t)^{1/2}.
\de 
Then by the multiplacation result above, H\"{o}lder inequality, (\ref{M-condtion-2.7}) and (\ref{M-condtion-2.8}), we have
\ce 
\|h(b-b^n)\|_{\mH_{-1,p}^{q/2}([s,t])}&\lesssim&\|h\|_{\mH_{-1,p}^q([s,t])}\|b-b^n\|{\mH_{-1,p}^q([s,t])}\lesssim \Pi\sW_1(s,t),\\
\|h(b-b^n)\|_{\mH_{-1,p}^{q/2}([s,t])}&\lesssim&\|h\|_{\mL_{\infty}^q([s,t])}\|b-b^n\|{\mL_{p}^q([s,t])}\lesssim \sW_1(s,t).
\de 
By part (iii) of Proposition \ref{s_Krelovy-esitimate1}, we get 
\ce 
&&\|\sup_{t\in[0,1]}|\int_0^th(b-b^n)(r,X_r)dr|\|_{L_{\bar{p}}(\Omega)}\\
&\lesssim&\Pi(1+|\log\Pi|)\sW_1(0,1)^{2/q}\\
&\lesssim&\Pi(1+|\log\Pi|)\sW_0(0,1)^{1/q}.
\de 
Similarly, applying Proposition \ref{s_Krelovy-esitimate1}, (\ref{M-condtion-2.7}) and (\ref{M-condtion-2.8}), we have 
\ce 
\|\sup_{t\in[0,1]}|\int_0^t(b-b^n)(r,X_r)dr|\|_{L_{\bar{p}}(\Omega)}\lesssim\Pi(1+|\log\Pi|)\sW_0(0,1)^{1/q}.
\de  
Combining the previous two estimates, we obtain the desired result.
\end{proof}

\appendix

\section{Proof of Lemma \ref{Khasminskii's lemma}}\label{appendixA}
\begin{proof}
	For $m\in \mN$ and any partition $S= t_0<t_1<\cdots<t_{n}=T$, by Tonelli's theorem we note that 
	\ce 
	\bigg(\int_S^T\xi(r)dr\bigg)^m=m!\int\cdots\int_{\Delta^m}\xi(r_1)\cdots\xi(r_m)dr_1\cdots dr_m,
	\de 
	where 
	\ce 
	\Delta^m:=\{(r_1,\cdots,r_m):t_i\leq r_1\leq \cdots\leq r_m\leq t_{i+1}\},
	\de 
	and then using (\ref{Khasminskii_condition_1}), we have  
	\be \label{Khasminskii_result_0}
	&&\|\int_{t_i}^{t_{i+1}}\xi(r)dr\|_{L_m(\Omega|\sF_{t_i})}^m\no\\
	&=&m!\mE_{t_i}\bigg[\int\cdots\int_{\Delta^m}\xi(r_1)\cdots\xi(r_m)dr_1\cdots dr_m\bigg]\no\\
	&=&m!\mE_{t_i}\bigg[\int\cdots\int_{\Delta^{m-1}}\xi(r_1)\cdots\xi(r_{m-1})dr_1\cdots dr_{m-1}\times\mE_{r_{m-1}}\bigg[\int_{r_{m-1}}^{t_{i+1}}\xi(r_m)dr_m\bigg]\bigg]\no\\
	&\leq&m!\mE_{t_i}\bigg[\int\cdots\int_{\Delta^{m-1}}\xi(r_1)\cdots\xi(r_{m-1})dr_1\cdots dr_{m-1}\bigg]\beta(t_i,t_{i+1})\no\\
	&\leq&\cdots \leq m!(\beta(t_i,t_{i+1}))^m.
	\ee
	By taking $t_i=S$ and $t_{i+1}=T$, we obtain the stated estimate for $\|\int_S^T\xi(r)dr\|_{L_p(\Omega|\sF_S)}^m$.
	
	From (\ref{Khasminskii_result_0}), we have 
	\ce 
	\mE_{t_i}\bigg[\exp(\kappa\int_{t_i}^{t_{i+1}}\xi(r)dr)\bigg]&=&\sum_{m}\frac{1}{m!}\mE_{t_{i}}\bigg(\kappa\int_{t_i}^{t_{i+1}}\xi(r)dr\bigg)^m\\
	&\leq&\sum_{m}(\kappa \beta(t_i,t_{i+1}))^m.
	\de 
	Then 
	\be \label{Khasminskii_result_1}
	\mE_S\bigg[\exp(\kappa\int_S^T\xi(r)dr)\bigg]&=&\mE_{t_0}\bigg[\Pi_{i=0}^{n-1}\exp(\kappa\int_{t_i}^{t_{i+1}}\xi(r)dr)\bigg]\no\\
	&=&\mE_{t_0}\bigg[\Pi_{i=0}^{n-2}\exp(\kappa\int_{t_i}^{t_{i+1}}\xi(r)dr)\mE_{t_{n-1}}\exp(\kappa\int_{t_{n-1}}^{t_{n}}\xi(r)dr)\bigg]\bigg]\no\\
	&\leq&\mE_{t_0}\bigg[\Pi_{i=0}^{n-2}\exp(\kappa\int_{t_i}^{t_{i+1}}\xi(r)dr)\bigg]\sum_{m}(\kappa \beta(t_i,t_{i+1}))^m\no\\
	&\leq&\Pi_{i=0}^{n-2}\sum_{m}(\kappa \beta(t_i,t_{i+1}))^m.
	\ee
	Taking $N\in\mN$ large enough such $(T-s)/N\leq \delta_0$ and denoting $t_i=S+i(T-s)/N$, from (\ref{Khasminskii_condition_2}) we have 
	\ce 
	\sum_{m}(\kappa \beta(t_i,t_{i+1}))^m\leq \frac{1}{1-\lambda \kappa}.
	\de 
	Hence, we get 
	\ce 
	\mE_S\bigg[\exp\bigg(\kappa\int_S^T\xi(r)dr\bigg)\bigg]\leq \bigg(\frac{1}{1-\lambda \kappa}\bigg)^N.
	\de 
	
	Since $\sW$ is a continuous control, we can take $t_0=S$ and for each $j\geq 1$, 
	\ce 
	t_j=\sup\{t\in[t_{j-1},T]:\kappa \sW(t_{j-1},t)^\gamma\leq 1/2\}.
	\de 
	With this choice, we have $\kappa \sW(t_{j-1},t_j)^\gamma=1/2$ for $j=1,\cdots,n-1$ and $\kappa \sW(t_{n-1},t_n)^\gamma\leq 1/2$. Then we have $n\leq 1+(2\kappa)^{1/\gamma}\sW(S,T)$ from the following fact 
	\ce 
	\frac{n-1}{(2\kappa)^{1/\gamma}}\leq \sum_{j=1}^n \sW(t_{j-1},t_j)\leq \sW(S,T).
	\de 
	Hence, from (\ref{Khasminskii_result_1}) and $\beta(s,t)\leq \sW(s,t)^\gamma$, we have 
	\ce 
	\mE_S\bigg[\exp\bigg(\kappa\int_S^T\xi(r)dr\bigg)\bigg]&\leq&\Pi_{i=0}^{n-2}\sum_{m}(\kappa \beta(t_i,t_{i+1}))^m\\
	&\leq &\Pi_{i=0}^{n-2}\sum_{m}(\kappa \sW(t_i,t_{i+1})^\gamma)^m\\
	&\leq &2^n\leq 2^{1+(2\kappa)^{1/\gamma}\sW(S,T)},
	\de 
	completing the proof.	
\end{proof}

\section{Proof of Theorem \ref{theorem4.5}}\label{appendixB}

\begin{proof}
	we put $\rho=\frac{d}{2p}-\frac{d}{2p'}$. We split this proof into some setps.
	
	\emph{Step 1.} We present some estimates for $\|(F_{r,t}^nf)\|_{p'}$ in term of $\|f\|_{p}$. First, assume that  $f$ is a bounded and uniformly continuous function. From Lemmas \ref{lemma4.3+} and \ref{lemma4.4+}, we have for every $t\in [s,s+1/n]$ 
	\be\label{formula1+6}
	\|Q_{s,t}^nf\|_{p'}&\leq& \|T_{s,t}f(x)\|_{p'}+\int_{s}^{t} \|F_{r,t}^nf\|_{p'}dr\no\\
	&\leq &(t-s)^{-\rho}\|f\|_{p}+\|f\|_{p}\int_{s}^{t}\bigg((t-r_n)^{\frac{\alpha}{2}-1-\rho}+(t-r_n)^{-\frac{\beta}{2}-\rho}\bigg)dr\no\\
	&\lesssim&(t-s)^{-\rho}\|f\|_{p}.
	\ee 
	The last inequality follows from the fact that $r_n=s$ for $r\in [s,s+1/n)$. Since smooth functions are dense in  $L_{p}(\mR^d)$,  it follows that for any function $f\in L_{p}(\mR^d)$ 
	\ce 
	\|(F_{r,t}^nf)\|_{p_2}\leq(t-s)^{-\rho}\|f\|_{p}.
	\de 
	
	We proceed inductively. Let $1\leq j\leq n$ be an integer. Suppose that for every  $t\in [s,s+j/n]$ and every $f\in L_{p}(\mR^d)$,
	\be\label{formula1+8} 
	\|Q_{s,t}^nf\|_{p'}\leq C_j(t-s)^{-\rho}\|f\|_{p},
	\ee 
	for some constant $C_j$, independent of $n,s,t,f$.
	
	Let $f$ be a bounded and uniformly continuous function. Then, for each $t\in(s+j/n,s+(j+1)/n]$, using Lemmas \ref{lemma4.3+}, \ref{lemma4.4+}, and the inductive hypothesis, we get
	\be\label{formula1+7} 
	&&\|Q_{s,t}^nf\|_{p'}\no\\
	&\lesssim& \|T_{s,t}f(x)\|_{p'}+\int_{s}^{s+1/n} \|F_{r,t}^nf\|_{p'}dr+C_j\int_{s+1/n}^{t}(r_n-s)^{-\rho}\|F_{r,t}^nf\|_{p}dr.
	\ee 
	The first two terms can be estimated as in (\ref{formula1+6}), 
	\ce 
	\|T_{s,t}f(x)\|_{p'}+\int_{s}^{s+1/n} \|F_{r,t}^nf\|_{p'}dr\lesssim [(t-s)^{-\rho}+(1/n)^{-\rho}]\|f\|_{p}.
	\de 
	We again apply Lemma \ref{lemma4.3+} to get 
	\ce 
	&&\int_{s+1/n}^{t}(r_n-s)^{-\rho}\|F_{r,t}^nf\|_{p}\,dr\\
	&\lesssim& \int_{s+1/n}^{t}(r_n-s)^{-\rho}\bigg((t-r_n)^{\frac{\alpha}{2}-1}+(t-r_n)^{-\frac{\beta}{2}}\bigg)\,dr\|f\|_{p}.
	\de 
	From the fact that $r_n-s\geq 1/n$ for any $r\geq s+1/n$ and Lemma \ref{integrale-Lemma3.6}, we have 
	\ce 
	&&\int_{s+1/n}^{t}(r_n-s)^{-\rho}\|F_{r,t}^nf\|_{p}dr\\
	&\lesssim& (1/n)^{-\rho}\int_{s+1/n}^{t}\bigg((t-r_n)^{\frac{\alpha}{2}-1}+(t-r_n)^{-\frac{\beta}{2}}\bigg)\,dr\|f\|_{p}\\
	&\lesssim&(1/n)^{-\rho}\|f\|_{p}.
	\de 
	Considering that
	\ce 
	(1/n)^{-\rho}\leq (j+1)^\rho(t-s)^{-\rho},
	\de 
	for $t\in (s+j/n,s+(j+1)/n]$, and putting these estimates into (\ref{formula1+7}), we obtain that (\ref{formula1+8}) still holds for $t\in (s+j/n,s+(j+1)/n]$  with any bounded uniformly continuous function $f$ and some constant $C_{j+1}$. By approximation, we extend the inequality (\ref{formula1+8}) to all functions $f\in L_{p}(\mR^d)$.
	
	\emph{Step 2.} In this step, we demonstrate that formula (\ref{formula1+8}) holds uniformly in $j$ under the assumption that $\rho<1$. Let $\lambda>1$ be a constant. For each $t\in(s+2/n,1]$, define 
	\ce 
	M_t:=e^{-\lambda(t-s)}(t-s)^\rho\sup_{h\in L_{p},\|h\|_{L_{p}}=1}\|Q_{s,t}^nh\|_{p'},
	\de 
	and 
	\ce 
	M_t^*=\sup_{r\in(s+2/n,t]}M_r,
	\de 
	both of which are finite by the previous step. For $t> s+2/n$ and  $h\in L_{p}(\mR^d)$ with $\|h\|_{p}\neq 0$, we have 
	\be\label{formula1+20} 
	\|Q_{s,t}^nh\|_{p'}\leq M_t^*e^{\lambda(t-s)}(t-s)^{-\rho}\|h\|_{p}. 
	\ee 
	Let $t> s+2/n$ and $f\in L_{p}(\mR^d)$  be  a bounded, uniformly contonuous function whih $\|f\|_{p}=1$. From Lemma \ref{lemma4.4+}, (\ref{formula1+6}) and (\ref{formula1+20} ),  we have that 
	\be\label{formula1+9} 
	&&\|Q_{s,t}^nf\|_{p'}\no\\
	&\lesssim& \|T_{s,t}f(x)\|_{p'}+\int_{s}^{s+2/n} \|F_{r,t}^nf\|_{p'}dr+M_t^*\int_{s+2/n}^{t}e^{\lambda(r_n-s)}(r_n-s)^{-\rho}\|F_{r,t}^nf\|_{p}dr.
	\ee
	We estimate the second term on the right-hand side. Using Lemma \ref{lemma4.3+},
	\ce 
	&&\int_{s}^{s+2/n} \|F_{r,t}^nf\|_{p'}dr\\
	&\lesssim&\int_{s}^{s+2/n}\bigg((t-r_n)^{\frac{\alpha}{2}-1-\rho}+(t-r_n)^{-\frac{\beta}{2}-\rho}\bigg)dr\\
	&\lesssim&1/n\bigg((t-s-1/n)^{\frac{\alpha}{2}-1-\rho}+(t-s-1/n)^{-\frac{\beta}{2}-\rho}\bigg)\\
	&\lesssim&(t-s-1/n)^{-\rho}\\
	&\lesssim&(t-s)^{-\rho},
	\de 
	where we use the facts $2/n< t-s\leq 1$ and $t-s-1/n\geq (t-s)/2$. Similarly, applying Lemma \ref{lemma4.3+}, we have 
	\ce 
	&&\int_{s+2/n}^{t}e^{\lambda(r_n-s)}(r_n-s)^{-\rho}\|F_{r,t}^nf\|_{p}dr\\
	&\lesssim&\int_{s+2/n}^{t}e^{\lambda(r_n-s)}(r_n-s)^{-\rho}\bigg((t-r_n)^{\frac{\alpha}{2}-1}+(t-r_n)^{-\frac{\beta}{2}}\bigg)dr\\
	&\lesssim&e^{\lambda(t-s)}\int_{s+1/n}^{t}e^{-\lambda(t-r)}(r-s-1/n)^{-\rho}\bigg((t-r)^{\frac{\alpha}{2}-1}+(t-r)^{-\frac{\beta}{2}}\bigg)dr\\
	&=:&e^{\lambda(t-s)}(B_1+B_2),
	\de 
	where 
	\ce 
	B_1&=&\int_{s+1/n}^{\frac{s+1/n+t}{2}}e^{-\lambda(t-r)}(r-s-1/n)^{-\rho}\bigg((t-r)^{\frac{\alpha}{2}-1}+(t-r)^{-\frac{\beta}{2}}\bigg)dr,\\
	B_2&=&\int_{\frac{s+1/n+t}{2}}^{t}e^{-\lambda(t-r)}(r-s-1/n)^{-\rho}\bigg((t-r)^{\frac{\alpha}{2}-1}+(t-r)^{-\frac{\beta}{2}}\bigg)dr.
	\de 
	We estimate $B_1$ and $B_2$ respectively as follows 
	\ce 
	B_1&\lesssim&e^{-\frac{\lambda}{2}(t-s-1/n)}\bigg((t-s-1/n)^{\frac{\alpha}{2}-1}+(t-s-1/n)^{-\frac{\beta}{2}}\bigg)\int_{s+1/n}^{\frac{s+1/n+t}{2}}(r-s-1/n)^{-\rho}dr\\
	&\lesssim&e^{-\frac{\lambda}{2}(t-s-1/n)}\bigg((t-s-1/n)^{\frac{\alpha}{2}-1}+(t-s-1/n)^{-\frac{\beta}{2}}\bigg)(t-s-1/n)^{-\rho+1}\\
	&\lesssim&(\lambda^{-\frac{\alpha}{2}}+\lambda^{\frac{\beta}{2}-1})(t-s-1/n)^{-\rho},
	\de 
	where the first inequality uses  $t-r\geq t-\frac{s+1/n+t}{2}=\frac{t-s-1/n}{2}$ , and the last inequality uses the fact that $e^{\frac{\lambda}{2}\gamma}>e^{\frac{\lambda \alpha}{2}\gamma }>\big(\frac{\lambda}{2}\gamma\big)^{\frac{\alpha}{2}}$ for $\gamma>0$.
	Similarly,
	\ce
	B_2&\leq&(t-s-1/n)^{-\rho}\int_{\frac{s+1/n+t}{2}}^{t}e^{-\lambda(t-r)}\bigg((t-r)^{\frac{\alpha}{2}-1}+(t-r)^{-\frac{\beta}{2}}\bigg)dr\\
	&\leq&(t-s-1/n)^{-\rho}\int_{0}^{\infty}e^{-\lambda u}\bigg(u^{\frac{\alpha}{2}-1}+u^{-\frac{\beta}{2}}\bigg)du\\
	&\lesssim&(\lambda^{-\frac{\alpha}{2}}+\lambda^{\frac{\beta}{2}-1})(t-s-1/n)^{-\rho},
	\de 
	where the first inequality uses $r-s-1/n\geq\frac{s+1/n+t}{2}-s-1/n=\frac{t-s-1/n}{2}$. 
	In the above, all integrals are finite because $\alpha\in (0,1]$, $\beta\in(0,2]$ and $\rho<1$. Noting that $(t-s)/(t-s-1/n)\leq 2$, we get
	\ce 
	\int_{s+2/n}^{t}e^{\lambda(r_n-s)}(r_n-s)^{-\rho}\|F_{r,t}^nf\|_{p}dr\lesssim(\lambda^{-\frac{\alpha}{2}}+\lambda^{\frac{\beta}{2}-1})(t-s)^{-\rho}e^{\lambda(t-s)}.
	\de 
	Putting these estimates and $\|T_{s,t}f\|_{p_1}\leq(t-s)^{-\rho}$ into (\ref{formula1+9}), we have 
	\ce 
	\|Q_{s,t}^nf\|_{p'}\lesssim (t-s)^{-\rho}(1+M_t^*(\lambda^{-\frac{\alpha}{2}}+\lambda^{\frac{\beta}{2}-1})e^{\lambda(t-s)}).
	\de 
	By approximation, the above estimate also holds for any function $f\in L_{p}(\mR^d)$ with $\|f\|_{p}=1$. It follows that $M_t\lesssim 1+M_t^*(\lambda^{-\frac{\alpha}{2}}+\lambda^{\frac{\beta}{2}-1})$ for all $t> s+2/n$. By selecting a sufficiently large value for $\lambda$, we can infer that $M_1^*$ is bounded by a constant that is independent of $n$. Consequently, we establish the existence of a constant $N>0$, also independent of $n$, such that 
	\be\label{formula1+10} 
	\|Q_{s,t}^nf\|_{p'}\leq N(t-s)^{-\rho}\|f\|_{p}
	\ee 
	for every $t>s+2/n$ and function $f\in L_{p}(\mR^d)$ with $\|f\|_{p}\neq 0$. Furthermore, based on (\ref{formula1+8}), it is evident that this estimate remains valid for any $f\in L_{p}(\mR^d)$ with $\|f\|_{p}= 0$ and $t>s$.
	
	\emph{Step 3.} We proceed to eliminate the constraint $\rho<1$ from Step 2.
	
	Let $\rho<2$. Denote $p_1\in [p,p']$ such that $\frac{1}{p_1}=\frac{1}{2p}+\frac{1}{2p'}$. This implies
	\ce 
	\frac{d}{2p}-\frac{d}{2p_1}=\frac{d}{2p_1}-\frac{d}{2p'}=\frac{\rho}{2}<1.
	\de 
	Consider $t> s+4/n$ and $u=(s+t)/2$. Utilizing the Markov property of the Euler-Maruyama scheme, we have
	\ce 
	Q_{s,t}^n=Q_{u_n,t}^nQ_{s,u_n}^n.
	\de 
	It is straightforward to verify that $t>u_n>s$. Applying (\ref{formula1+10}), we get for every $f\in L_{p}(\mR^d)$ that 
	\ce 
	\|Q_{s,t}^nf\|_{p'}&\lesssim& (t-u_n)^{-\frac{\rho}{2}}\|Q_{s,u_n}^nf\|_{p}\\
	&\lesssim&(t-u_n)^{-\frac{\rho}{2}}(u_n-s)^{-\frac{\rho}{2}}\|f\|_{p}.
	\de 
	Given $(t-u_n)\geq (t-s)/2$ and $(u_n-s)\geq (t-s)/4$, we conclude  
	\ce 
	\|Q_{s,t}^nf\|_{p'}\lesssim (t-s)^{-\rho}\|f\|_{p},
	\de
	for any $t> s+4/n$. Combining with (\ref{formula1+8}),  we see that (\ref{formula1+10}) holds for any $p,p'\in(d/\beta\vee 1,\infty]$ satisfying $\rho<2$.  
	
	We iterate the argument. After $\lfloor \log_2(d/2)\rfloor+1$ iterations, we find that  (\ref{formula1+10}) whenever $\rho\leq d/2$ , which is trivially satisfied for any  $p,p'\in(d/\beta\vee 1,\infty]$. Thus, we have completed the proof.
\end{proof}

\section{Proof of Lemma \ref{krlvy-estimate1}}\label{appendixC}

\begin{proof}
	We may assume without loss of generality that $f$ is nonnegative. Let $(s,t)\in \Delta$ and for every $u\in(s,s_n^+)$ where $s_n^+=s_n+1/n$, we express
	\ce 
	\widetilde{X}_u^{n}&=&\widetilde{X}_{s}^{n}+\int_{s}^{u}\sigma(r,\widetilde{X}_{r_n}^{n})dB_r+\int_{s}^{u}\int_{|z|< R}g(r,\widetilde{X}_{{r_n-}}^{n},z)\widetilde{N}(dr\,dz),
	\de 
	and 
	\ce 
	\widetilde{\eta}_u(y)&=&\int_{s}^{u}\sigma(r,y)dB_r+\int_{s}^{u}\int_{|z|\leq R}g(r,y,z)\widetilde{N}(dr\,dz)\\
	&=:&\xi+\zeta.
	\de 
	Then, we compute
	\ce 
	&&\mE_s\bigg[\int_{s}^{t\wedge s_n^+}f(r,\widetilde{X}_r^{n})dr\bigg]\\
	&=&\int_{s}^{t\wedge s_n^+}\mE_s[f(r,x+\widetilde{\eta}_u(y))]|_{(x,y)=(\widetilde{X}_{s}^{n},\widetilde{X}_{s_n}^{n})}dr\\
	&=&\int_{s}^{t\wedge s_n^+}\mE_s[\mE[f(r,x+\xi+v)]|_{v=\zeta}]|_{(x,y)=(\widetilde{X}_{s}^{n},\widetilde{X}_{s_n}^{n})}dr\\
	&=&\int_{s}^{t\wedge s_n^+}\mE_s\bigg[\int_{\mR^d}p_{A_{s,r}(w)}(w-x-v)f(r,w)dw\big|_{v=\zeta}\bigg]\big|_{(x,y)=(\widetilde{X}_{s}^{n},\widetilde{X}_{s_n}^{n})}dr,
	\de 
	where 	$A_{s,r}(w):=\frac{1}{2}\int_s^r a_\tau(w)d\tau$, $a_\tau(w)=(\sigma\sigma^T)_t(w)$.
	Applying the ellipticity of $A_{s,r}(w)$ and H\"{o}lder inequality, we have 
	\ce 
	\int_{\mR^d}p_{A_{s,r}(w)}(w-x-z)f(r,w)dw&\leq& \|f(r,\cdot)\|_{p}\bigg(\int_{\mR^d}p_{A_{s,r}(y)}^{\frac{p}{p-1}}(y-x-z)dy\bigg)^\frac{p-1}{p}\\
	&\lesssim &(r-s)^{-\frac{d}{2p}}\|f(r,\cdot)\|_{p}.
	\de 
	Hence, using the above estimate and  H\"{o}lder inequality again, we get 
	\be \label{formula1+190}
	\mE_s\bigg[\int_{s}^{t\wedge s_n^+}f(r,\widetilde{X}_r^{n})dr\bigg]
	&\lesssim &\int_{s}^{t\wedge s_n^+}(r-s)^{-\frac{d}{2p}}\|f(r,\cdot)\|_{p}dr\no\\
	&\lesssim &(t-s)^{1-\frac{d}{2p}-\frac{1}{q}}\|f(r,\cdot)\|_{\mL_{p}^{q}([s,t])}.
	\ee 
	On the interval $(s_n^+\wedge t,t)$, when $p<\infty$,  (\ref{formula1+1300})  with $\rho=1$ is applied to yield 
	\be \label{conditional_estimate_fx10}
	\mE_s[f(r,\widetilde{X}_r^{n})]
	\lesssim (r-s_n^+)^{-\frac{d}{2p}}\|f(r,\cdot)\|_{p}.
	\ee
	Since the above inequality is trivial for $p=\infty$, we get (\ref{conditional_estimate_fx10}) for for $p\in(d/\beta\vee 1,\infty]$.	We use (\ref{conditional_estimate_fx10}) and H\"{o}lder inequality to see that 
	\be \label{conditional_estimate_fx11}
	\mE_s\bigg[\int_{t\wedge s_n^+}^tf(r,\widetilde{X}_r^{n})dr\bigg]
	&\lesssim&\int_{t\wedge s_n^+}^t(r-s_n^+)^{-\frac{d}{2p}}\|f(r,\cdot)\|_{p}dr\no\\
	&\lesssim&\|f\|_{\mL_{p}^{q}([s,t])}(t-s)^{1-\frac{d}{2p}-\frac{1}{q}}.
	\ee 
	From (\ref{formula1+190})	and	(\ref{conditional_estimate_fx11}), it follows that 
	\be\label{formula0+24+} 
	\mE_s\bigg[\int_{s}^tf(r,\widetilde{X}_r^{n})dr\bigg]\lesssim \|f\|_{\mL_{p}^{q}([s,t])}(t-s)^{1-\frac{d}{2p}-\frac{1}{q}}.
	\ee 
	Observe that $\sW$, defined by 
	\ce 
	\sW(s,t)^{1-\frac{d}{2p}}=\|f\|_{\mL_{p}^{q}([s,t])}(t-s)^{1-\frac{d}{2p}-\frac{1}{q}} 
	\de 
	is a continuous control on $\Delta$. Applying Lemma  \ref{Khasminskii's lemma}, we obtian
	\ce 
	\|\int_{s}^{t}f(r,\widetilde{X}_r^{n})dr\|_{L_m(\Omega|\sF_s)}\leq N_m\|f\|_{\mL_{p}^{q}([s,t])}(t-s)^{1-\frac{d}{2p}-\frac{1}{q}},
	\de  
	where $N_m$ depends on $d,\alpha,\beta, p,q,c_0,c_1,m$,
	and 
	\ce 
	\mE_s\bigg[\exp\bigg(\int_{s}^{t}f(r,\widetilde{X}_r^{n})dr\bigg)\bigg]\lesssim 2^{1+2^{1-\frac{d}{2p}}\sW(s,t)}\lesssim 2\exp(N\|f\|_{\mL_{pp_0}^{q}([0,1])}^{1/(1-\frac{d}{2p})}).
	\de 
	where $N$ depends on $d,\alpha,\beta, p,q,c_0,c_1$.
	
	The second part is derived analogously.  For each $(s,t)\in \Delta$, define $\bar{s}=s_n+2/n$. Using H\"{o}lder inequality, (\ref{control_sw_0}) and (\ref{formula1+130}), we get 
	\ce 
	\mE_s\bigg[\int_s^{\bar{s}\wedge t} f(r,\widetilde{X}_r^{n})dr\bigg]&\leq& \int_s^{\bar{s}\wedge t}\|f(r,\cdot)\|_{\infty}dr\\
	&\lesssim& \|f\|_{\mL_{\infty}^q([s,t])}(1/n)^{1-\frac{1}{q}}\lesssim (\sW_0(s,t))^{\gamma_0}
	\de 
	and 
	\ce 
	\mE_s\bigg[\int_{\bar{s}\wedge t}^ t f(r,\widetilde{X}_r^{n})dr\bigg]&\lesssim&\int_{\bar{s}\wedge t}^ t(r_n-s)^{-\frac{d}{2p}}\|f(r,\cdot)\|_{p}dr\\
	&\lesssim&\int_{\bar{s}\wedge t}^ t(r-s-1/n)^{-\frac{d}{2p}}\|f(r,\cdot)\|_{p}dr\\
	&\lesssim&\|f\|_{\mL_{p}^q([s,t])}(t-s-1/n)^{1-\frac{d}{2p}-\frac{1}{q}}\\
	&\lesssim&\|f\|_{\mL_{p}^q([s,t])}(t-s)^{1-\frac{d}{2p}-\frac{1}{q}}.
	\de 
	Hence, $$\mE_s\bigg[\int_{s}^tf(r,\widetilde{X}_r^{n})dr\bigg]\lesssim (\sW_0(s,t))^{\gamma_0}+ (\sW(s,t))^{1-\frac{d}{2p}},$$
	where $\sW$is the control defined in the first part of the proof. Applying Lemma \ref{Khasminskii's lemma} and part (4) of  Examples \ref{Examples2.1}, we obtain (\ref{exponention-11}). 
\end{proof}

\subsection*{Acknowledgments}
The author is grateful to Professor Ren Jiagang at Sun Yat-sen University for his thorough and diligent reading of an earlier draft of this work. His careful scrutiny uncovered several mistakes, and his insightful recommendations have greatly improved the quality of the paper. This work is supported by National Natural Science Foundation of China (Grant Nos. 12361030 and 12261038), and Natural Science Foundation of Jiangxi Province (Grant Nos. 20232BAB201004 and 20242BAB23003).

%%%%%%%%%%%%%%%%%%%%%%%%%%%%%%%
%\bibliographystyle{plain}
%\bibliography{ZCmm2022.bib}
%%%%%%%%%%%%%%%%%%%%%%%%%%%%%%%%%%%%%

\end{document}